\documentclass[12pt]{amsart}
\usepackage[backend=biber,
style=alphabetic, sorting=nyt, doi=false,isbn=false, url=false]{biblatex} 
\usepackage{url}
\usepackage{amsmath}
\usepackage{csquotes}
\usepackage{worldflags}
\usepackage{enumerate}
\theoremstyle{remark}
\usepackage{amsxtra}
\usepackage{scalerel}
\usepackage{hyperref}
\usepackage{enumitem}
\usepackage{floatrow}
\usepackage[new]{old-arrows}
\usepackage{amssymb}
\usepackage{tikz-cd}
\theoremstyle{plain}
\newtheorem{Theo}[subsubsection]{Theorem}

\newtheorem{problem}[subsubsection]{Problem}

\newtheorem*{Theorem}{Theorem}

\newtheorem{cor}[subsubsection]{Corollary}

\newtheorem{Conj}{Conjecture}

\theoremstyle{definition}

\newtheorem{fact}[subsubsection]{Fact} 
\newtheorem{question}[subsubsection]{Question}
\newtheorem{Prop}[subsubsection]{Proposition}

\newtheorem{Lemm}[subsubsection]{Lemma} 

\newtheorem{definition}[subsubsection]{Definition}
\theoremstyle{remark}
\newtheorem{notation}[subsubsection]{Notation}
\newtheorem{Obs}[subsubsection]{Observation}
\newtheorem{rem}[subsubsection]{Remark}
\newtheorem{remark}[subsubsection]{Remark}
\newtheorem{example}[subsubsection]{Example}

\newtheorem{ex}[subsection]{Exercise}

\newcommand{\nc}{\newcommand}
\newcommand{\lmod}{{\operatorname{-Mod}}}
\newcommand{\quotes}[1]{``#1''}
\nc{\bp}{\begin{Prop}}
\nc{\ep}{\end{Prop}}
\nc{\ssn}{\pagebreak \section}
\nc{\df}{{\bf Definition}\ }
\nc{\bl}{\begin{Lemm}}
\nc{\el}{\end{Lemm}}
\nc{\bex}{\begin{ex} \rm}
\nc{\eex}{\end{ex}}
\nc{\bt}{\begin{Theo}}
\nc{\et}{\end{Theo}}
\nc{\bq}{\begin{question}}
\nc{\eq}{\end{question}}
\nc{\bc}{\begin{cor}}
\nc{\ec}{\end{cor}}
\nc{\bob}{\begin{Obs}}
\nc{\eob}{\end{Obs}}
\nc{\N}{\mathbb{N}}
\nc{\Q}{\mathbb{Q}}
\nc{\Z}{\mathbb{Z}}
\nc{\Ss}{{\mathbb{S}}}
\nc{\Cc}{{\mathbb{C}}}
\nc{\F}{{\mathbb{F}}}
\nc{\Oo}{\mathcal{O}}
\nc{\Qq}{\mathbb{Q}}
\nc{\ulim}{\text{ulim}\ }
\nc{\Hom}{\text{Hom}}
\nc{\Ext}{\text{Ext}}
\nc{\Tor}{\text{Tor}}
\nc{\Ob}{\text{Ob}}
\nc{\id}{\text{id}}

\renewcommand{\mkbegdispquote}[2]{\itshape}
 \nc{\Spec}{{\mathop{\operatorname{\rm Spec}}}}
  \nc{\spec}{{\mathop{\operatorname{\rm Spec}}}}

\usepackage{yhmath}

\usepackage[utf8]{inputenc}

\title{Beyond the Fontaine-Wintenberger theorem}
\author{Franziska Jahnke and Konstantinos Kartas}

\thanks{During this research, F.J. 
was supported by the Deutsche Forschungsgemeinschaft (DFG, German Research Foundation) under Germany's Excellence Strategy EXC 2044-390685587, Mathematics M\"unster: Dynamics-Geometry-Structure, as well as by a Fellowship from the Daimler and Benz Foundation. K.K. received funding from the European Union's Horizon 2020 research and innovation programme under the Marie Sk\l odowska-Curie grant agreement No 101034255\worldflag[width=1.5mm, length=1.5mm]{EU} Both F.J. and K.K were also supported by the program GeoMod ANR-19-CE40-0022-01 (ANR-DFG)}

\begin{document}

\maketitle

\begin{abstract}
Given a perfectoid field, we find an elementary extension and a henselian defectless valuation on it, whose value group is divisible and whose residue field is an elementary extension of the tilt. This specializes to the almost purity theorem over perfectoid valuation rings and Fontaine-Wintenberger.  Along the way, we prove an Ax-Kochen/Ershov principle for certain deeply ramified fields, which also uncovers some new model-theoretic phenomena in positive characteristic. Notably, we get that the perfect hull of $\F_p(t)^h$ is an elementary substructure of the perfect hull of $\F_p(\!(t)\!)$. 
\end{abstract}
\setcounter{tocdepth}{1}
\tableofcontents
\pagebreak

\section{Introduction}
The local function field analogy, i.e., the analogy between the field of $p$-adic numbers $\Q_p$ and the field of Laurent series $\F_p(\!(t)\!)$, has been highly influential in arithmetic. Historically, there have been two important justifications of this philosophy, which were also used to transfer certain results from one domain to the other. Both methods are of asymptotic nature but are substantially different from one another (cf. pg. 2 \cite{ScholzeICM}). 
\subsection{The method of letting $p\to \infty$} \label{aketranfser}
The idea here is that $\Q_p$ and $\F_p(\!(t)\!)$ become more and more similar as $p$ goes to infinity. This is made precise by the work of Ax-Kochen \cite{AK12} and Ershov \cite{Ershov}, which shows that for any given sentence $\phi$ in the language of rings, there is $N=N(\phi)\in \N$ such that 
$$\phi \mbox{ holds in }\Q_p \iff \phi \mbox{ holds in } \F_p(\!(t)\!)$$ 
for any $p>N$. The crucial ingredient behind this, is that the theory of a henselian valued field of residue characteristic zero is fully determined by the theories of its residue field and value group. This is now known to hold in more general settings and has become a guiding principle in the model theory of valued fields. Using this as a black box, one easily checks that 
$$\prod_{p\in \mathbb{P}} \Q_p /U \equiv \prod_{p\in \mathbb{P}} \F_p(\!(t)\!)/U $$
for any non-principal ultrafilter $U$ on the set $\mathbb{P}$ of prime numbers. Then, the asymptotic transfer principle which was stated above follows routinely. Later, this principle was generalized in a motivic framework by work of Denef-Loeser \cite{DL} and Cluckers-Loeser \cite{CluckLoes}, allowing an asymptotic comparison between measures of definable sets. 

This transfer principle between $\Q_p$ and $\F_p(\!(t)\!)$ enabled Ax-Kochen to resolve a conjecture by Artin, at least in its asymptotic form. Artin conjectured that $\Q_p$ is a $C_2$ field, i.e., that every homogeneous polynomial of degree $d$ in more than $d^2$ variables will have a non-zero solution over $\Q_p$. The corresponding statement was already known for $\F_p(\!(t)\!)$ by a theorem of Lang. Using the above transfer principle, Ax-Kochen proved an asymptotic version of Artin's conjecture.
Namely, for every $d\in \N$, Artin's conjecture holds for homogeneous polynomials of degree $d$ in $\Q_p$ with $p$ sufficiently large.
We note that $\Q_p$ is actually not $C_2$, as was demonstrated by Terjanian in \cite{Terj1} for $p=2$ and \cite{Terj2} for general $p$. 
\subsection{The method of letting $e\to \infty$}  \label{methodperf}
The second method is the one which will be of interest to us throughout the paper. Now fix $p$ and instead replace the fields $\Q_p$ and $\F_p(\!(t)\!)$ with highly ramified finite extensions thereof. This has the effect of making the similarities between the two fields even stronger, a phenomenon which was first observed by Krasner \cite{Krasner}. 
This method works especially well once we pass to infinite wildly ramified extensions. For instance, consider the extensions $\Q_p(p^{1/p^{\infty}})$ and $\F_p(\!(t)\!)(t^{1/p^{\infty}})$, obtained by adjoining $p$-power roots of $p$ and $t$ respectively.
Then one has the celebrated theorem of Fontaine-Wintenberger \cite{FontWint}, saying that the absolute Galois groups of those two fields are isomorphic. This was vastly generalized by Scholze \cite{Scholze} within the framework of perfectoid fields and spaces. 

A \textit{perfectoid field} is a non-discrete complete real-valued field $K$ of residue characteristic $p>0$ such that $\Oo_K/(p)$ is \textit{semi-perfect}, i.e., the Frobenius map $\Phi:\Oo_K/(p)\to \Oo_K/(p):x\mapsto x^p$ is surjective. 
Given a perfectoid field $K$, one can define its \textit{tilt} $K^{\flat}$, a perfectoid field of positive characteristic which is in many respects similar to $K$. We note that tilting is not \quotes{injective}, meaning that there are several perfectoid fields of characteristic zero which tilt to the same perfectoid field of characteristic $p$. These are parametrized by a regular noetherian scheme of dimension one, called the Fargues-Fontaine curve. Thus, in order to untilt, one also needs to specify a parameter. This parameter is a certain Witt vector $\xi_K \in W(\Oo_{K^{\flat}})$, only well-defined up to a unit. The valuation ring of the corresponding untilt is simply $\Oo_K=W(\Oo_{K^{\flat}})/(\xi_K)$.

Even though tilting involves some loss of information, a lot can be learned about $K$ from studying $K^{\flat}$. For instance, one has a generalized Fontaine-Wintenberger theorem, saying that the absolute Galois groups of $K$ and $K^{\flat}$ are isomorphic. This version was stated and proved by Scholze \cite{Scholze} but also follows from parallel work of Kedlaya-Liu \cite{KedlayaLiu}.
Most strikingly, Scholze obtained an appropriate geometric generalization of the above arithmetic statement within the context of perfectoid spaces. Several arithmetic applications have been found since; see \cite{ScholzeICM} for a survey and also \cite{PerfBook} for a more up-to-date account. 

\subsection{Motivating question} \label{motivatingquestion}
Given how successful the method of transferring information between $K$ and $K^{\flat}$ has been, one can ask:
\bq \label{questionformal}
How are $K$ and $K^{\flat}$ related model-theoretically? More specifically:
\begin{enumerate}[label=(\roman*)]
\item To what extent does the theory of $K^{\flat}$ determine the theory of $K$? 
\item How close are $K$ and $K^{\flat}$ to being elementary equivalent? 
\end{enumerate}
\eq 
This question was posed by M. Morrow at a conference talk in 2018 but was more or less folklore before. Some answers have been given since:
\begin{itemize}
\item The perfectoid fields $K$ and $K^{\flat}$ are bi-interpretable within the framework of continuous logic. This was an early formative (yet unpublished) observation by Rideau-Scanlon-Simon. We emphasize that their interpretation of $K$ in $K^{\flat}$ requires parameters for the Witt vector coordinates of $\xi_K \in W(\Oo_{K^{\flat}})$.

\item The situation becomes subtle within the framework of ordinary first-order logic and parameters play a prominent role. For instance, one can prove that $K$ is \quotes{in practice} decidable relative to $K^{\flat}$. More precisely, this is true if $\xi_K \in W(\Oo_{K^{\flat}})$ is \textit{computable}, meaning that an algorithm exists which prints one by one the Witt vector coordinates of $\xi_K$. This is the main theorem in \cite{KK1} and was used to obtain some new decidability results in mixed characteristic.
\end{itemize}
Even though $K$ and $K^{\flat}$ have different characteristics, we naturally expect a large set of sentences to be transferrable. For instance, sentences encoding Galois theoretic information certainly are, as $K$ and $K^{\flat}$ have isomorphic absolute Galois groups.\footnote{Recall that the elementary theory of the absolute Galois group---in the formalism of Cherlin-van den Dries-Macintyre \cite{CDM}---is encoded in the language of rings.} We provide a new perspective on tilting , illustrating how it enables the transfer of first-order information between $K$ and $K^{\flat}$.
%
\subsection{Main theorem}
By passing to an elementary extension of $K$, we find a
well-behaved valuation, whose residue field is an elementary
extension of $K^{\flat}$:
\begin{Theorem}[see Theorem \ref{modeltheoreticFont}]\label{modeltheoreticFont1}
Let $(K,v)$ be a perfectoid field and $\varpi\in \mathfrak{m}\backslash \{0\}$. Let $U$ be a non-principal ultrafilter on $\N$ and $(K_U,v_U)$ be the corresponding ultrapower. Let $w$ be the coarsest coarsening of $v_U$ such that $w\varpi>0$. Then: 
\begin{enumerate}[label=(\Roman*)]
\item Every finite extension of $(K_U,w)$ is unramified.

\item The tilt $(K^{\flat},v^{\flat})$ embeds elementarily into $(k_w,\overline{v})$, where $\overline{v}$ is the induced valuation of $v_U$ on $k_w$.
\end{enumerate} 
Moreover, the isomorphism $G_{K_U}\cong G_{k_w}$ restricts to $G_K\cong G_{K^{\flat}}$.
\end{Theorem}
By a finite unramified extension of henselian fields, we mean one which induces a separable residue field extension of the same degree.
Part (I) equivalently says that the valuation $w$ is henselian defectless with divisible value group and perfect residue field. In particular, it is a \textit{tame} valuation in the sense of Kuhlmann \cite{Kuhl}. The fact that the value group is of infinite (uncountable) rank may be somewhat intimidating. We note that this can easily be remedied by replacing the ultrapower with a bounded variant and then $w$ becomes a spherically complete valuation with value group $\mathbb{R}$ (cf. \S \ref{rank1version}). 

Tame valuations are by and large well-understood by work of Kuhlmann, not only from the point of view of valuation theory but also model-theoretically. In \S \ref{aketranfser}, we mentioned that the theory of a henselian valued field is often determined by the theories of its residue field and value group. Unfortunately, this is not quite true for tame fields of mixed characteristic. There, the situation is a bit subtle and goes hand in hand with the fact that the theory of $K^{\flat}$ does not fully determine the theory of $K$ (cf. Remark \ref{tameakerem}). Still, knowing the theory of $k_w$ (equivalently of $K^{\flat}$) tells us a lot about the theory of $K_U$ (equivalently of $K$), see  Remark \ref{transferproperties}.

Next, we discuss the two parts of the main theorem:
\subsection{Part I} \label{PARTI}
The fact that every finite extension of $(K_U,w)$ is unramified will not come as a surprise to experts. It is a non-standard version of the \textit{almost purity} theorem over perfectoid valuation rings, as we will explain in \S \ref{almostmathintro}. 
Interestingly, our non-standard version allows for a significantly shorter proof in mixed characteristic compared to existing proofs.  We outline the proof in each case.
\subsubsection{Positive characteristic} \label{poschar}
This is essentially identical with the (easy) proof of almost purity in $\text{char}(K)=p>0$. The key observation is that over a perfect field of characteristic $p$, an irreducible polynomial 
$$f(X)=X^n+a_{n-1}X^{n-1}+...+a_0$$ 
generates the same extension as 
$$f_m(X)=X^n+a_{n-1}^{1/p^m}X^{n-1}+...+a_0^{1/p^m}$$ 
for any $m\in \N$. Looking at discriminants, we get that $\text{disc}(f_m)=\text{disc}(f)^{1/p^m}$ and therefore
$$v(\text{disc}(f_m))=1/p^m\cdot v(\text{disc}(f))$$ 
This procedure allows us to generate a given finite extension of $K$ by a monic irreducible polynomial with discriminant of arbitrarily small valuation. As for $K_U$, we may even form the limit polynomial $f_*(X)=\ulim f_m(X)$. By \L o\'s, this still generates the same extension and now $\text{disc}(f_*)\in \Oo_w^{\times}$. The reduction of $f_*(X)$ is therefore separable---also irreducible by Hensel's Lemma---and the extension is unramified. The case of a general extension of $K_U$ is similar.
\subsubsection{Characteristic zero} \label{char0}
We need to show that $w$ is henselian, with divisible value group, perfect residue field and that there is no defect. The first three conditions are more or less clear. We only remark that the residue field $k_w=\text{Frac}(\Oo_{\overline{v}})$ is perfect because the semi-perfect ring $(\Oo_K/(p))_U$ surjects onto $\Oo_{\overline{v}}$. The only non-trivial thing to show is that $\Oo_w$ is defectless. 

Here, we utilize the method of standard decomposition, a classical tool in the model theory of valued fields. Consider the following diagram of places 
$$K_U\stackrel{v_0} \longrightarrow K_U^{\circ} \stackrel{\overline{w}} \longrightarrow k_{w}\longrightarrow k_{v_U} $$ 
where $v_0$ is the finest coarsening of $v_U$ of residue characteristic $0$. We write $K_U^{\circ}$ for the residue field with respect to $v_0$, often called the \textit{core field}, and $\overline{w}$ for the induced valuation of $w$ on $K_U^{\circ}$, which is of rank $1$. 
One knows that---as a consequence of $\aleph_1$-saturation---the valued field $(K_U^{\circ},\overline{w})$ is spherically complete. This means that the intersection of any decreasing sequence of balls is nonempty and it implies that $\overline{w}$ is defectless.  
The valuation $v_0$ is defectless, being of residue characteristic $0$. We conclude that $w$ is also defectless, being the composite of such valuations. 
\subsubsection{ } \label{generalcase}
In \S \ref{tamingperfectoids}, we work with more general valued fields. We consider the class of henselian valued fields  (not necessarily of rank $1$) $(K,v)$ of residue characteristic $p>0$ with $\Oo_K/(p)$ semi-perfect, together with a distinguished $\varpi\in \mathfrak{m}\backslash \{0\}$ such that $\Oo_K[\varpi^{-1}]$ is algebraically maximal. This class clearly includes perfectoid fields and, unlike perfectoid fields, forms an elementary class. In this generality, one gets a slightly weaker conclusion for the coarsening $w$, namely that it is \textit{roughly tame}. 
In \S \ref{modeltheoryoftame}, we show that Kuhlmann's results also hold in the broader context of roughly tame fields.
This method of \quotes{taming the valuation} enables us to prove a new Ax-Kochen/Ershov principle, which is the crucial ingredient for Part II but which is also of independent interest. Before discussing this, we wish to clarify the connection with the almost purity theorem which was hinted at above.


%

\subsection{Almost mathematics} \label{almostmathintro}

Part I of our main theorem is a non-standard version of the almost purity theorem over perfectoid valuation rings:
\begin{Theorem} [Almost purity] \label{tategr}
Let $K$ be a perfectoid field and $K'/K$ be a finite extension. Then $K'/K$ is almost unramified. More precisely, the extension $\Oo_{K'}/\Oo_K$ is almost finite \'etale. 
\end{Theorem} 
The notion \quotes{almost \'etale} is made precise in almost mathematics, introduced by Faltings \cite{Faltings} and developed systematically in the work of Gabber-Ramero \cite{GR}. 
We note that almost purity holds more generally over perfectoid $K$-algebras, see \cite{Faltings} and \cite{Scholze}.
The version over perfectoid valuation rings stated here is often attributed to Tate \cite{Tate} and Gabber-Ramero \cite{GR} (cf. Theorem 2.4 \cite{ScholzeSurvey}). We first explain some of the key ideas in almost mathematics and then clarify their model-theoretic content.



\subsubsection{ } \label{examplealmost}
For a concise summary, see \S 4 \cite{Scholze}. We start with an easy example of an almost \'etale extension. Fix $p\neq 2$ and let $K$ be the $p$-adic completion of $\Q_p(p^{1/p^{\infty}})$ and $K'=K(p^{1/2})$. Classically, $K'/K$ is a totally ramified extension. However, $\Oo_{K'}/\Oo_K$ is in some sense very close to being \'etale. To see this, let us focus at a finite level. Consider the local fields $K_n=\Q_p(p^{1/p^n})$ and $K_n'=K_n(p^{1/2})$ and note that $\Oo_{K'_n}=\Oo_{K_n}[p^{1/2p^n}]$. To see if $\Oo_{K'_n}/\Oo_{K_n}$ is \'etale, one can compute the module of K\"ahler differentials and check whether it is zero. Although this is not quite the case here, we have that the module is killed by $p^{1/p^n}$, i.e.,
$$p^{1/p^n}\cdot \Omega_{\Oo_{K_n'}/\Oo_{K_n}}=0$$
For large $n$, the element $p^{1/p^n}$ is very close to being a unit and $\Omega_{\Oo_{K_n'}/\Oo_{K_n}}$ is thus close to being zero. As for the module $\Omega_{\Oo_{K'/K}}$, we get that $\mathfrak{m}\cdot \Omega_{\Oo_{K'}/\Oo_K}=0$ by passing to direct limits. In almost mathematics, such a module is called \textit{almost zero} and must be ignored. 
We now outline the general setup which makes all this precise.

Let $K$ be a rank-$1$ valued field with a non-discrete valuation (e.g., a perfectoid field) and $\Oo_K$ denote its valuation ring. We write $\Oo_K \lmod$ for the category of $\Oo_K$-modules. In order to ignore $\mathfrak{m}$-torsion, we pass to the Serre (or Gabriel) quotient 
$$\Oo_K^a\lmod =\Oo_K\lmod /(\mathfrak{m}\text{-} \mbox{torsion modules})$$ 
The resulting category $\Oo_K^a\lmod$ is called the category of \textit{almost modules}. It comes with a natural \textit{almostification} (or localization) functor 
$$\Oo_K\lmod\to \Oo_K^a\lmod: M\mapsto M^a$$
To understand the nature of almost modules, we make a few observations. Note that if $\mathfrak{m}\cdot M=0$, then certainly $M\otimes K=0$. Therefore, the functor $\Oo_K\lmod \to K\lmod:M\mapsto M\otimes K$ factors naturally through $\Oo_K^a\lmod$ as in the diagram below
\[
\begin{tikzcd}[column sep=4.5em]
\Oo_K \lmod\arrow{rr}{M\mapsto M \otimes K} \arrow{dr}{M\mapsto M^a} && K\lmod \\
& \Oo_K^a\lmod \arrow{ur}{M^a \mapsto M \otimes K} 
\end{tikzcd}
\]
Thus, given an $\Oo_K$-module $M$, the object $M^a$ can be viewed as a slightly generic (or almost integral) structure, sitting between the integral structure $M$ and its generic fiber $M\otimes K$. However, as there are no intermediate rings $\Oo_K\subsetneq R \subsetneq K$, the object $M^a$ does not actually arise from a localization of the base ring $\Oo_K$. There is not even an underlying set of the object $M^a$. It will still be useful to think of $M^a$ as an algebraic object, but the formal development of the theory will inevitably be abstract.

One then proceeds to develop ring theory within $\Oo_K^a\lmod $ in a hybrid fashion, half-algebraic and half-categorical. 
After some work, one arrives at the notion of an \'etale extension in the almost context. 
An important theme in almost mathematics is that properties over the generic fiber extend to the almost integral level (see \S 4 \cite{Scholze}). This is reflected in the almost purity theorem; \'etaleness at the generic level implies \'etaleness at the almost integral level.
\subsubsection{ }
To see how model theory enters the picture, let us return to the example of the quadratic extension $K'/K$ from \S \ref{examplealmost}. In that situation, it would have really been convenient to take the \quotes{limit} of $p^{1/p^n} $ as $n\to \infty$ and produce some kind of unit element. We cannot make sense of that limit topologically but we can model-theoretically. Write $K_U$ and $K_U'$ for the corresponding ultrapowers. We can now form the ultralimit $\pi=\ulim p^{1/p^n}$ and use \L o\'s to write $K'_U=K_U(\pi^{1/2})$.  Ramification still persists in $K_U'/K_U$ but is now detected only at the infinitesimal level. To make it disappear, we just need lower resolution lenses: Let $w$ (resp. $w'$) be the coarsest coarsenings with $wp >0$ (resp. $w'p>0$). This has the effect of collapsing infinitesimal valuations to zero; e.g., we have that $w\pi=0$. One easily checks that $\Oo_{w'}=\Oo_w[\pi^{1/2}]$ and hence $\Oo_{w'}/\Oo_w$ is an honest finite \'etale extension. 
We spell out the details in Example \ref{sqrootexmplefg}.
Next, we outline how this fits into a bigger picture. 

Let $S\subseteq \Oo_{v_U}$ be the set of elements of \textit{infinitesimal} valuation, i.e., smaller than any rational multiple of $v\varpi$. Given an $\Oo_K$-module $M$, it is not hard to check that 
$$\mathfrak{m}\cdot M=0\iff S^{-1}M_U=0$$ 
The module $S^{-1}M_U$ is a slightly generic fiber of the ultrapower $M_U$, now in a literal sense. 
This yields an exact faithful monoidal functor of tensor abelian categories
$$S^{-1}({-})_U:\Oo_K^a\lmod\longrightarrow \Oo_w\lmod:M^a\mapsto S^{-1}M_U $$
Crucially, it also preserves tensor products under some finite type assumption on the almost modules being tensored. 
After writing this paper, O. Gabber informed us that he recently used a similar construction to simplify a step in the proof of the direct summand conjecture (see Remark \ref{Gabberscons}).  

With the functor $S^{-1}({-})_U$, we achieve a systematic translation from almost ring theory to usual ring theory. However, we believe that this is not just a matter of phrasing things in a different language. Having actual algebraic objects to work with enables us to see some new phenomena. Such is the case of Part II, which we discuss next.
\subsection{Part II: Embedding the tilt in the residue field} \label{partII}
\subsubsection{ }
We first describe the embedding of $K^{\flat}$ in $k_w$ (see \S \ref{elementarysubfield}). Since $K^{\flat}=\varprojlim_{x\mapsto x^p}K$, we have a natural embedding of multiplicative monoids
$$\natural: K^{\flat} \to K_U: (x,x^{1/p},...)\mapsto \ulim x^{1/p^n}$$
It is also additive modulo $p$ and the image is contained in $\Oo_w^{\times}$. Composing it with the reduction map of $w$, gives us a valued field embedding $\iota:(K^{\flat},v^{\flat})\hookrightarrow (k_w,\overline{v})$. Moreover, the induced embedding $\Oo_{K^{\flat}}/(t) \to \Oo_{\overline{v}}/(\iota(t))$ is \textit{essentially} a diagonal embedding into an ultrapower (i.e., modulo some non-standard Frobenius twist) and is elementary by \L o\'s.
\subsubsection{ } \label{elembeddingintro}
In order to show that $\iota:(K^{\flat},v^{\flat})\hookrightarrow (k_w,\overline{v})$ is elementary, we prove an Ax-Kochen/Ershov principle:
\bt \label{perfprecversion}
Let $(K,v)\subseteq (K',v')$ be two henselian valued fields of residue characteristic $p>0$ such that $\Oo_v/(p)$ and $\Oo_{v'}/(p)$ are semi-perfect. Suppose there is $\varpi \in \mathfrak{m}_v\backslash \{0\}$ such that both $\Oo_v[\varpi^{-1}]$ and $\Oo_{v'}[\varpi^{-1}]$ are algebraically maximal. Then: 
$$ (K,v)\preceq (K',v') \iff \Oo_v/(\varpi )\preceq \Oo_{v'}/(\varpi) \mbox{ and }\Gamma_v \preceq \Gamma_{v'}$$
Moreover, if $\Gamma_v$ and $\Gamma_{v'}$ are regularly dense, then the value group condition on the right hand side can be omitted. 
\et 
We refer the reader to \S 5, which also contains analogous Ax-Kochen/Ershov principles for elementary equivalence (over a base) and existential closedness (see \S \ref{akesec}). All proofs go via a coarsening argument, namely by \quotes{taming} the valuations as outlined in \S \ref{generalcase} and thereby reducing to a setting where Kuhlmann's theory applies. These new Ax-Kochen/Ershov principles are also of independent interest, as they uncover some new model-theoretic phenomena, outlined below.
\subsubsection{Other applications} \label{otherappintro}
It is an open problem whether $\F_p(t)^h$ is an elementary substructure of $\F_p(\!(t)\!)$ and, alas, the theory of $\F_p(\!(t)\!)$ remains unknown. In \S \ref{newaxkochenphenomena}, we obtain the following \quotes{perfected} variant:
\bc \label{corolfp(t)}
The perfect hull of $\F_p(t)^h$ is an elementary substructure of the perfect hull of $\F_p(\!(t)\!)$.
\ec 
Unfortunately, the theory of the perfect hull of $\F_p(\!(t)\!)$ still appears to be intractable. By \quotes{taming} the valuation, the difficulties do not disappear but rather get absorbed in the truncated valuation ring $\Oo_v/(t)$. Still, for certain concrete situations---like Corollary \ref{corolfp(t)}---one may directly compare the two truncated valuation rings and see that they are equal on the nose. Other Ax-Kochen/Ershov phenomena also become accessible in this fashion and are collected in \S \ref{newaxkochenphenomena}.

As another application, we obtain a converse to the decidability transfer theorem of \cite{KK1}: 
\bc 
The $L_{\text{val}}$-theory of $K^{\flat}$ is decidable relative to the one of $K$.
\ec

Finally, we show that tilting respects elementary equivalence over a base:
\bc \label{subcomperf}
Let $K_1,K_2$ be two perfectoid fields extending a perfectoid field $K$. Then: 
$$(K_1,v_1) \equiv_{(K,v)} (K_2,v_2) \iff (K_1^{\flat},v_1^{\flat}) \equiv_{(K^{\flat},v^{\flat})} (K_2^{\flat},v_2^{\flat}) $$
\ec 
We stress that working over a base perfectoid field is essential, as there are numerous examples of perfectoid fields which all have the same tilt but which are not elementary equivalent. This is reminiscent of the fact that tilting yields an equivalence between the categories of perfectoid field extensions of $K$ and those of $K^{\flat}$, but not between the categories of perfectoid fields of characteristic zero and those of characteristic $p$.
%

\newpage
\addtocontents{toc}{\protect\setcounter{tocdepth}{0}}
\section*{Notation}

\begin{itemize}
\item If $(K,v)$ is a valued field, we denote by $\mathcal{O}_v$ the valuation ring and by $\mathfrak{m}_v$ the maximal ideal. If the valuation is clear from the context, we shall also write $\mathcal{O}_K$ for the valuation ring. When both the valuation and the field are clear from the context, we will write $\mathfrak{m}$ for the maximal ideal.
\item We write $\Gamma_v$ for the value group and $k_v$ for the residue field. If the valuation $v$ is clear from the context, we also denote them by $\Gamma$ and $k$ respectively. Given $x\in \Oo_v$, we write $\overline{x}$ for the image of $x$ in $k_v$ via the residue map of $v$. 
\item We write $L_{\text{rings}}=\{0,1,+,\cdot\}$ for the language of rings, $L_{\text{oag}}=\{0,+,<\}$ for the language of ordered abelian groups. 
\item We write $L_{\text{val}}$ for the language of valued fields, construed as a three-sorted language with a sort for the valued field in the language $L_{\text{rings}}$, a sort for the value group in the language $L_{\text{oag}}$, a sort for the residue field in the language $L_{\text{rings}}$ and function symbols for the valuation and residue maps. We also let $L_{\text{val}}(\varpi)=L_{\text{val}}\cup \{\varpi\}$, where $\varpi$ is a constant symbol whose intended interpretation will be a specified element in the valued field.

\item If $M$ is an $L$-structure and $A\subseteq M$ is an arbitrary subset, we write $L(A)$ for the language $L$ enriched with a constant symbol $c_a$ for each element $a \in A$. The $L$-structure $M$ can be updated into an $L(A)$-structure in the obvious way.

\item  Given $L$-structures $M$ and $N$ with a common substructure $A$, we use the notation $M\equiv_A N$ to mean that the structures $M$ and $N$ are elementary equivalent in $L(A)$.

\item Given $L$-structures $M$ and $N$ with $M\subseteq N$, we write $M\preceq N$ to mean that $M$ is an elementary substructure of $N$ and $M\preceq_{\exists} N$ to mean that $M$ is existentially closed in $N$.

\item Given a ring $R$ with nilradical $N=\text{Nil}(R)$, we write $R_{\text{red}}=R/N$ for the associated reduced ring. We write $R-$\text{f\'et} for the category of finite \'etale extensions of $R$.

\end{itemize} 
\addtocontents{toc}{\protect\setcounter{tocdepth}{1}}
\section{Preliminaries}

\subsection{Ultraproducts} \label{ultrasection}
Fix an ultrafilter $U$ on some index set $I$. Given a language $L$ and a sequence $(M_i)_{i\in I}$ of $L$-structures, we let $\prod_{i\in I} M_i/U$ be the associated \textit{ultraproduct}. The underlying set of $\prod_{i\in I} M_i/U$ is the quotient $\{(a_i)_{i\in I}:a_i\in M_i\}/\sim$, where $(a_i)_{i\in I} \sim (b_i)_{i\in I}$ if and only if $\{i\in I:a_i=b_i\}\in U$. We will write $\ulim a_i$ for the equivalence class of $(a_i)_{i\in I}$ in the ultraproduct. The ultraproduct $\prod_{i\in I} M_i/U$ readily becomes an $L$-structure. We will often use the more succinct notation $\ulim M_i$ for the ultraproduct. In case $M_i=M$, the structure $\prod_{i\in I} M/U$ is called the \textit{ultrapower} of $M$ along $U$ and is denoted by $M_U$. 
\subsubsection{\L o\'s' Theorem}
If the $M_i$'s are groups (resp. rings, fields, valued fields etc.), then so is the ultraproduct. More generally, \L o\'s' Theorem ensures that any elementary property which holds for \textit{almost all} of the $M_i$'s will also hold in the ultraproduct (and vice versa): 
\begin{fact} [\L o\'s' Theorem] \label{los}
Fix $n\in \N$ and choose $a_i\in M_i^n$ for each $i\in I$. For any $L$-formula $\phi(x)$ with $x=(x_1,...,x_n)$, one has that 
$$\{i\in I: M_i \models \phi(a_i)\}\in U\iff \prod_{i\in I} M_i/U\models \phi (\ulim a_i)$$
\end{fact}
We identify $M$ as a substructure of $M_U$ via the diagonal embedding 
$$\delta:M\to M_U: m\mapsto \ulim m$$
which is elementary by \L o\'s' Theorem.
\subsubsection{Saturation}
We refer the reader to \S 6.1 \cite{changkeisler} for a detailed discussion on ultraproducts and saturation. 
\begin{definition}
Fix an infinite cardinal $\kappa$. Given an $L$-structure $M$, we say that $M$ is $\kappa$\textit{-saturated} if for any subset $A\subseteq M$ with $|A| < \kappa$, every collection $(D_{\alpha})_{\alpha \in I}$ of $A$-definable sets having the finite intersection property has non-empty intersection. We say that $M$ is \textit{saturated} if it is $|M|$-saturated.
\end{definition}
\begin{fact} [Theorem 5.1.13 \cite{changkeisler}] \label{uniquenessofsat}
Let $M$ and $M'$ be two elementary equivalent saturated $L$-structures of the same size. Then $M\cong M'$.
\end{fact}
\begin{rem}
\begin{enumerate}[label=(\roman*)]
\item By Lemma 5.1.4 \cite{changkeisler}, it is always possible to achieve $\kappa$-saturation by passing to a suitable elementary extension. 

\item Contrary to (i), the existence of saturated models is not guaranteed in general without further background set-theoretic assumptions (GCH). It is nevertheless harmless to assume it for certain tasks at hand and later remove it by an absoluteness argument, cf. \cite{HaleviKaplan}.  
\end{enumerate}
\end{rem}
If $L$ is countable and $\kappa=\aleph_1$, then $M$ is $\aleph_1$-saturated precisely when every decreasing chain of non-empty definable subsets with parameters from $M$, say
$$D_1\supseteq D_2 \supseteq...$$ 
has non-empty intersection. Structures in everyday mathematical practice often fail to be $\aleph_1$-saturated: 
\begin{example}
The ordered set $(\N,<)$ is not $\aleph_1$-saturated because one can take $D_n=\{x\in \N: x>n\}$ and clearly $\bigcap_{n\in \N} D_n = \emptyset$. 
\end{example}
On the other hand, non-principal ultraproducts are always $\aleph_1$-saturated: 
\begin{fact} [Theorem 6.1.1 \cite{changkeisler}] 
Let $(M_n)_{n\in \N}$ be a sequence of $L$-structures and $U$ be a non-principal ultrafilter on $\N$. Then the ultraproduct $\ulim M_i$ is $\aleph_1$-saturated.
\end{fact}
In particular, it is possible to achieve $\aleph_1$-saturation by passing to a non-principal ultrapower, something which will be used extensively throughout. 
\subsubsection{Some ultracommutative algebra}
Let $I$ be an index set and $(R_i)_{i\in I}$ be a sequence of rings.  We work in the product category 
$$\prod_{i\in I} (R_i\lmod)$$ 
whose objects are sequences $(M_i)_{i\in I}$, where $M_i$ is  an $R_i$-module, and where morphisms are given as follows
$$\Hom((M_i)_{i\in I},(N_i)_{i\in I})=\prod_{i\in I} \Hom(M_i,N_i)$$
\begin{rem}
The following observations are clear:
\begin{enumerate}[label=(\roman*)]
\item The category $\prod (R_i\lmod)$ is an abelian tensor category, where kernels, cokernels and tensor products are defined in the unique way compatible with their definition in $R_i\lmod$. For instance, we have
$$(M_i)_{i\in I}\otimes (N_i)_{i\in I}= (M_i\otimes N_i)_{i\in I}$$
\item If $R_i=R$, we have a diagonal functor
$$\Delta: R\lmod \to (R\lmod)^I:M\mapsto (M,M,...)$$
which is an exact faithful functor which preserves tensor products.
\end{enumerate}
\end{rem}

\begin{definition}
Let $(R_i)_{i\in I}$ be a sequence of rings. For each $i\in I$, let $M_i$ be an $R_i$-module.
\begin{enumerate}[label=(\roman*)]
\item The sequence $(M_i)_{i\in I}$ is called \textit{uniformly finitely generated} if there exists $m\in \N$ such that each $M_i$ can be generated by $m$ elements as an $R_i$-module. 
\item The sequence $(M_i)_{i\in I}$ is called \textit{uniformly finitely presented} if there exist $m,n\in \N$ such that for each $M_i$ there exists a short exact sequence
$$0\longrightarrow R_i^m \stackrel{\phi_i} \longrightarrow R_i^n \stackrel{\psi_i} \longrightarrow M_i \longrightarrow 0 $$ 
for some maps $\phi_i$ and $\psi_i$.
\end{enumerate}

\bl \label{ultramodules}
Let $I$ be a set and $U$ be an ultrafilter on $I$. Let $(R_i)_{i\in I}$ be a sequence of rings and $R=\ulim R_i$. Consider the assignment 
$$\mathcal{U}: \prod_{i\in I} (R_i\lmod) \to R\lmod: (M_i)_{i\in I}\mapsto \ulim M_i$$
acting on morphisms in the obvious way. Then:
\begin{enumerate}[label=(\roman*)]
\item $\mathcal{U}$ is an exact functor. 

\item $\mathcal{U}$ is monoidal, i.e., for any $R_i$-modules $M_i,N_i$ we have a natural morphism 
$$F:\ulim M_i\otimes_{R} \ulim N_i \to \ulim (M_i\otimes_{R_i} N_i)$$ 
sending $\ulim a_i\otimes \ulim b_i\mapsto \ulim  (a_i\otimes b_i)$.
\item The map $F$ from (ii) is an isomorphism if $(M_i)_{i\in I}$ and $(N_i)_{i\in I}$ are uniformly finitely generated.
\end{enumerate}
\el 
\begin{proof}
(i) It is clear that $\mathcal{U}$ is a functor. For exactness, note that for a sequence of maps $f_i:M_i\to N_i$ of $R_i$-module we have $\text{Ker}(\ulim f_i)=\ulim \text{Ker}(f_i)$ (resp. $\text{Im}(\ulim f_i)=\ulim \text{Im}(f_i)$) by \L o\'s.  \\
(ii) This is Lemma 3.1 \cite{Dietz}; we briefly recall the argument. Consider the bilinear map
$$\ulim M_i\times \ulim N_i \to \ulim  (M_i\otimes_{R_i} N_i): (\ulim a_i, \ulim b_i)\mapsto \ulim (a_i\otimes b_i) $$
By the universal property for tensor products, this naturally induces a morphism $F$ of $R$-modules
$$F:\ulim M_i\otimes_{R} \ulim N_i\to  \ulim (M_i\otimes_{R_i} N_i) $$ 
mapping $\ulim a_i\otimes \ulim b_i\mapsto \ulim  (a_i\otimes b_i)$ and extending additively to all of $\ulim M_i\otimes_{R} \ulim N_i$. \\
(iii) Let $M_i=\langle f_{1,i},...,f_{s,i} \rangle $ and $N_i=\langle g_{1,i},...,g_{s,i}\rangle$. 
Set $f_j=\ulim f_{j,i}$ and $g_k=\ulim g_{k,i}$. Suppose that $a_1,...,a_s\in \ulim M_i$ are such that
$$0=F( \sum_{j=1}^s a_j\otimes g_j)=\ulim \sum_{j=1}^s a_{j,i}\otimes g_{j,i}$$ 
where $a_j=\ulim a_{j,i}$. For almost all $i$, we will then have that $\sum_{j=1}^s a_{j,i}\otimes g_{j,i}=0$. By Tag 04VX \cite{sp}, this gives $c_{kj,i} \in R$ such that
$$a_{j,i}= \sum_{k=1}^s c_{kj,i} f_{k,i} \mbox{ and } \sum_{j=1}^s c_{kj,i} g_{j,i}=0 $$
for almost all $i$. Let $c_{kj}=\ulim c_{kj,i}$ and note that
$$a_j= \sum_{k=1}^s c_{kj} f_k \mbox{ and } \sum_{j=1}^s c_{kj} g_j=0 $$
This implies that $\sum_{j=1}^s a_j\otimes g_j=0$, which means that $F$ is injective.  To see that $F$ is also surjective, observe that $\ulim (M_i\otimes_{R_i} N_i)=\langle \ulim f_{j,i}\otimes g_{k,i}: 1\leq j,k \leq s\rangle$ and that $F(f_j\otimes g_k)=\ulim f_{j,i}\otimes g_{k,i}$. 
\end{proof}

\subsection{Ordered abelian groups} \label{oagsection}
Throughout the paper, we view ordered abelian groups as $L_{\text{oag}}$-structures, where $L_{\text{oag}}=\{0,+,<\}$. We will need a special class of ordered abelian groups which was studied by Robinson-Zakon \cite{Rob}:
\begin{definition} 
An ordered abelian group $\Gamma$ is called \textit{regularly dense} if for any given $\gamma_1<\gamma_2$ and $n\in \N$, there is $\gamma\in \Gamma$ such that $n\cdot \gamma\in (\gamma_1,\gamma_2)$. 
\end{definition}
The typical example is a dense additive subgroup of $\mathbb{R}$. In fact, any regularly dense ordered abelian group is elementary equivalent to such a subgroup. Regularly dense ordered abelian groups are readily seen to be \textit{almost divisible}, i.e., if $\{0\}\subsetneq \Delta\subseteq \Gamma$ is any proper convex subgroup, then $\Gamma/\Delta$ is divisible.

For later use, we record the following:
\bl \label{valuegrouplemma}
Suppose $\Gamma \preceq_{\exists} \Gamma'$ for two ordered abelian groups and $\gamma \in \Gamma$. Suppose $\Gamma$ is $\aleph_1$-saturated and let $\Delta$ (resp. $\Delta'$) be the maximal convex subgroup of $\Gamma$ (resp. $\Gamma'$) not containing $\gamma$. Then $\Gamma/\Delta\preceq_{\exists} \Gamma'/\Delta'$.
\el 
\begin{proof}
It suffices to show that any system of linear equations and inequalities with parameters from $\Gamma/\Delta$, which has a solution in $\Gamma'/\Delta'$, also has a solution in $\Gamma/\Delta$.
Let $A_1,A_2 \in M_{n\times n}(\Z)$ and $\gamma_1,\gamma_2\in \Gamma^n$. Suppose $\gamma' \in \Gamma'^n$ is such that 
$$A_1 \cdot \gamma' +\gamma_1 \in \Delta'^n \mbox{ and }A_2\cdot \gamma' +\gamma_2 > \Delta'^n$$ 
The former condition means that $A_1\cdot \gamma' +\gamma_1 \in  (-\frac{1}{m}\gamma, \frac{1}{m}\gamma)^n$ for all $m\in \N$. The latter condition means that there exists $M\in \N$ such that $A_2\cdot \gamma' +\gamma_2  \in (\frac{1}{M} \gamma,\infty)^n$. Since $\Gamma \preceq_{\exists} \Gamma'$, for each $m\in \N$ we can find $\gamma_m \in \Gamma^n$ such that 
$$A_1\cdot \gamma_m +\gamma_1 \in  (-\frac{1}{m}\gamma, \frac{1}{m}\gamma)^n\mbox{ and }A_2\cdot \gamma_m +\gamma_2  \in (\frac{1}{M} \gamma,\infty)^n$$ 
Since $\Gamma$ is $\aleph_1$-saturated, there is $\gamma_0\in \Gamma^n$ satisfying the above conditions simultaneously for all $m\in \N$. In particular, this means that 
$$A_1\cdot \gamma_0 +\gamma_1 \in \Delta' \mbox{ and }A_2\cdot \gamma_0 +\gamma_2 > \Delta'$$
which is what we wanted to show.
\end{proof}
\subsection{Coarsenings of valuations}
Our exposition follows closely \S 7.4 \cite{vdd}. Let $(K,v)$ be a valued field with valuation ring $\Oo_v$, value group $\Gamma_v$ and residue field $k_v$. A \textit{coarsening} $w$ of $v$ is a valuation on $K$ such that $\Oo_v\subseteq \Oo_w$. 

\subsubsection{Coarsenings and convex subgroups}
Recall the usual 1-1 correspondence between coarsenings of $v$ and convex subgroups of $\Gamma_v$: 
\begin{itemize}
\item Given a convex subgroup $\Delta \subseteq \Gamma_v$, we define a new valuation 
$$w:K^{\times} \to \Gamma_v/\Delta:x\mapsto vx+\Delta$$ 
One readily verifies that $\Oo_w=\{x\in K: vx> \delta \mbox{ for some }\delta \in \Delta\}$ and $\mathfrak{m}_w=\{x\in K: vx> \Delta\}$.
\item Conversely, a coarsening $w$ of $v$ gives rise to the convex subgroup $\Delta=\{vx: x\in K \mbox{ and }wx=0\}$.
\end{itemize}

\subsubsection{Induced valuation}
Given a coarsening $w$ of $v$, we define the \textit{induced} valuation $\overline{v}$ of $v$ on $k_w$ as follows
$$\overline{v}:k_w^{\times} \to \Delta: x + \mathfrak{m}_w \mapsto vx$$ 
Note that $\Oo_{\overline{v}}=\Oo_v/\mathfrak{m}_w$ and $\mathfrak{m}_{\overline{v}}=\mathfrak{m}_v/\mathfrak{m}_w$. We thus get an isomorphism $\Oo_{\overline{v}}/\mathfrak{m}_{\overline{v}} \cong \Oo_v/\mathfrak{m}_v$ which allows us to identify $k_{\overline{v}}$ with $k_v$. Let $P_v:K\to k_v$ be the place associated to $v$ (similarly for the other valuations). From the above discussion, we have that $P_v=P_{\overline{v}} \circ P_w$, as in the diagram below 
\[ 
  \begin{tikzcd}[column sep=4.5em]
    K \arrow[r, "w"] \arrow[rr, bend left, "v"] & k_{w} \arrow[r, "\overline{v}"] & k_v  
     \end{tikzcd}
\]
By abuse of notation, we also say that $v$ is the \textit{composition} of $w$ and $\overline{v}$ and write $v=w\circ \overline{v}$. 

Given a ring $R$ with nilradical $N=\text{Nil}(R)$, we write $R_{\text{red}}=R/N$.
\bl \label{reducedring}
Let $(K,v)$ be a valued field and $I\subseteq \mathfrak{m}_v$. Let $w$ be the coarsest coarsening of $v$ such that $w \varpi >0$ for all $\varpi \in I$. We then have that $\Oo_{\overline{v}}=(\Oo_v/I)_{\text{red}}$.
\el 
\begin{proof}
By construction, we have that $\mathfrak{m}_w=\{x\in \Oo_v: x^n\in I \mbox{ for some }n\in \N\}=\sqrt{ I}$. It follows that $\Oo_{\overline{v}}=\Oo_v/\mathfrak{m}_w=\Oo_v/\sqrt{I}=(\Oo_v/I)_{\text{red}}$.
\end{proof}

\subsubsection{Standard decomposition}
Let $(K,v)$ be an $\aleph_1$-saturated valued field and $\varpi \in \mathfrak{m}_v$. Let $w$ be the coarsest coarsening of $v$ such that $w\varpi >0$ and $w^+$ be the finest coarsening of $v$ with $w^+\varpi=0$ (i.e., $\Oo_{w^+}=\Oo_v[\varpi^{-1}]$).  This gives rise to the \textit{standard decomposition with respect to }$\varpi$ below
\[
  \begin{tikzcd}
    K \arrow[r, "\Gamma_v/\Delta^+"] & k_{w^+} \arrow[r, "\Delta^+/\Delta"] & k_w \arrow[r, "\Delta"] & k_v  
     \end{tikzcd}
\]
where each arrow is labelled with the corresponding value group. Here, $\Delta^+$ is the minimal convex subgroup with $v\varpi \in \Delta^+$ and $\Delta$ is the maximal convex subgroup with $v\varpi \notin \Delta$. The quotient $\Delta^+/\Delta$ is an ordered abelian group of rank $1$. 
\begin{rem} \label{corefield}
Classically, one takes $(K,v)$ of mixed characteristic $(0,p)$ and $\varpi=p$. In that case, we have that $w^+=v_0$ is the finest coarsening of residue characteristic $0$ and $w$ is the coarsest coarsening of residue characteristic $p$. This leads to the standard decomposition (with respect to $p$)
\[
  \begin{tikzcd}
    K \arrow[r, "\Gamma_v/\Delta_0"] & K^{\circ}\arrow[r, "\Delta_0/\Delta"] & k_w \arrow[r, "\Delta"] & k_v  
     \end{tikzcd}
\]
where $K^{\circ}$ denotes the residue field of $v_0$ and is sometimes called the \textit{core field}. The induced valuation of $v$ on $K^{\circ}$ will be denoted by $v^{\circ}$.
\end{rem}
\subsubsection{Saturation and spherical completeness}
Recall that a \textit{spherically complete} field is a valued field of rank $1$ such that for any decreasing sequence of balls, say 
$$B_1\supseteq B_2 \supseteq...$$ 
their intersection is non-empty.
The following is well-known but we sketch a proof for completeness. 
\bl [cf. Theorem 1.13 \cite{KuhlAns}]\label{sphericalcompletelemma}
Let $(K,v)$ be an $\aleph_1$-saturated valued field and $\varpi \in \mathfrak{m}_v$. Let $w$ be the coarsest coarsening of $v$ such that $w\varpi >0$ and $w^+$ be the finest coarsening of $v$ with $w^+\varpi=0$. Then the valued field $(k_{w^+},\overline{w})$ is spherically complete with value group isomorphic to either $\Z$ or $\mathbb{R}$. 
\el 
\begin{proof}
Let $B_1\supseteq B_2 \supseteq ...$ be a decreasing chain of closed balls in $(k_{w^+},\overline{w})$. We need to show that $\bigcap_{n\in \N} B_n \neq \emptyset$. For each $i\in \N$, we lift the center of $B_i$ to obtain a ball $\tilde{B}_i$ in $(K,w)$ with center in $\Oo_w$ and the same radius as $B_i$. We may assume that $\tilde{B}_i\subseteq \Oo_w$ by replacing $\tilde{B}_i$ with $\tilde{B}_i\cap \Oo_w$ if necessary. We then lift the radius of $\tilde{B}_i$ (i.e., choose a representative of the corresponding element in $\Gamma_v/\Delta_0$)  to obtain a ball $B_i'$ in $(K,v)$ with the same center. This yields a countable decreasing chain of closed balls $B_1'\supseteq B_2' \supseteq ...$ in $(K,v)$. These balls are (nonempty) definable sets and  by $\aleph_1$-saturation we have $\bigcap_{n\in \N} B_i' \neq \emptyset$. But clearly $B_i'\subseteq \tilde{B}_i$ and therefore we can find some $a\in \bigcap_{n\in \N} \tilde{B}_i$. Since $\tilde{B}_i\subseteq \Oo_w$, we get $ a\in \Oo_w$ and thus $b \in \bigcap_{n\in \N} \tilde{B}_i \neq \emptyset$, where $b=a\mod \mathfrak{m}_{w^+}$.
For the value group assertion, we have that $\Gamma_{\overline{w}}$ is either discrete or dense, being an additive subgroup of $\mathbb{R}$. In the latter case, we also get that every Dedekind cut is realized in $\Gamma_{\overline{w}}$ from the $\aleph_1$-saturation of $\Gamma_v$. Thus,  $\Gamma_{\overline{w}}=\mathbb{R}$.
\end{proof}
%
\subsection{Fundamental inequality}
We restrict ourselves to \textit{henselian} fields, see \S 3.3 \cite{engprest} for a more general treatment. Given a henselian valued field $(K,v)$ and a finite extension $K'/K$ of degree $n$, let $v'$ be the unique extension of $v$ to $K'$. We write $\Gamma_v$ (resp. $\Gamma_{v'}$) and $k_v$ (resp. $k_{v'}$) for the value groups and residue fields. We then have the \textit{fundamental inequality}
$$n\geq  e(v'/v)\cdot f(v'/v) $$
where $e(v'/v)=[\Gamma_{v'}:\Gamma_v]$ is the \textit{ramification degree} and $f(v'/v)=[k_{v'}:k_v]$ is the \textit{inertia degree}. 
\subsubsection{Defect}
The multplicative error in the fundamental inequality, namely 
$$d(v'/v)=\frac{n}{e(v'/v)\cdot f(v'/v)}$$ 
is called the \textit{defect} of the extension. By the Lemma of Ostrowski, the defect is always a power of the characteristic exponent of $k_v$. In particular, it is equal to $1$ in case $\text{char}(k_v)=0$. We say that $(K',v')/(K,v)$ (or $\Oo_{v'}/\Oo_v$) is \textit{defectless} if $d(v'/v)=1$. Otherwise we say that $(K',v')/(K,v)$ (or $\Oo_{v'}/\Oo_v$) is a defect extension. The valued field $(K,v)$ is called \textit{defectless} if any finite extension as above is defectless. In that case, we also say that $\Oo_v$ (or just $v$) is defectless.
\begin{fact}  \label{sphericaldefect}
Spherically complete fields are henselian and defectless.
\end{fact}
\begin{proof}
See Proposition 15, pg. 163 \cite{bosch}.
\end{proof}
\subsubsection{Immediate extensions}
An important special case of defect extensions are the immediate ones. We say that $(K',v')/(K,v)$ (or $\Oo_{v'}/\Oo_v$) as above is an \textit{immediate} extension if  $e(v'/v)= f(v'/v) =1$. Note that if $K'/K$ is a proper extension, then $(K',v')/(K,v)$ is indeed a defect extension. We say that $(K,v)$ (or $\Oo_v$, or $v$) is \textit{algebraically maximal} if $(K,v)$ does not admit any proper finite immediate extensions.
\subsubsection{Unramified extensions} \label{unramext}
There are several equivalent definitions of \quotes{unramified} in the literature, which we briefly review below.
\begin{definition}
A finite extension of henselian fields $(K',v')/(K,v)$ of degree $n$ is \textit{unramified} if $n=f(v'/v)$ and $k_{v'}/k_v$ is separable. 
\end{definition}
We also say that $K'/K$ is \textit{unramified with respect to} $\Oo_v$. We note that our definition of \quotes{unramified} is not universally accepted, namely some sources do not require the extension to be defectless (e.g., \cite{Endler}) while others do (e.g., \cite{kuhlmann2010defect}). 

Under our henselianity assumption, this is compatible with the notion of $\Oo_{v'}/\Oo_{v}$ being unramified in the sense of commutative algebra: \begin{definition}[Tag 00US \cite{sp})] 
A ring map $R\to S$ is unramified if $S$ is a finitely generated $R$-algebra and $\Omega_{S/R}=0$. We also say that $S/R$ is unramified.
\end{definition}
Given a field $K$ and $a\in \overline{K}$ with minimal polynomial $m_a(X)\in K[X]$, we write $\delta(a)=m_a'(a)$ for the \textit{different} of the element $a$. Given $f(X)=X^n+...+a_1X+a_0\in K[X]$ with roots $r_1,...,r_n \in \overline{K}$ (not necessarily distinct), we define the \textit{discriminant} of $f(X)$ as 
$$\text{disc}(f)=(-1)^{\frac{n(n-1)}{2}} \prod_{i<j} (r_i-r_j)^2$$  
The discriminant is a symmetric polynomial in the roots and can therefore be expressed as an integer polynomial in the $a_i$'s (see Theorem H.1 \cite{rotman}). 

\begin{fact} \label{unramfact}
Let $(K,v)$ be a henselian valued field and $(K',v')/(K,v)$ be a finite extension. Then the following are equivalent: 
\begin{enumerate}[label=(\roman*)]
\item $(K',v')/(K,v)$ is unramified.
\item There exists $a\in \Oo_{v'}$ such that $\delta(a) \in \Oo_{v'}^{\times}$ and $\Oo_{v'}=\Oo_v[a]$.
\item There exists $a\in \Oo_{v'}$ such that $\delta(a) \in \Oo_{v'}^{\times}$ and $K'=K(a)$.
\item There exists a monic irreducible polynomial $f(X)\in \Oo_v[X]$ such that $\text{disc}(f) \in \Oo_{v}^{\times}$ and $K'\cong K[X]/(f(X))$.

\item $\Oo_{v'}/\Oo_v$ is unramified. 

\item $\Oo_{v'}/\Oo_v$ is \'etale. 
\end{enumerate}
\end{fact}
\begin{proof}
We first prove that (i)$\iff$(ii)$\iff$(iii)$\iff$(iv).\\
(i)$\Rightarrow $(ii): Since $k_{v'}/k_v$ is finite separable, we can write $k_{v'}=k_v(\alpha)$ for some $\alpha \in k_{v'}$ with $\delta(\alpha)\neq 0$. Let $a\in \Oo_{v'}$ be any lift of $\alpha$. Note that 
$$[K':K]\geq [K(a):K]\geq [k_v(\alpha):k_v]$$ 
Since $[K':K]=[k_{v'}:k_v]$, equality holds everywhere and $K'=K(a)$. It follows that $\overline{m_a}(X)=m_{\alpha}(X)$ and hence $\overline{\delta(a)}=\delta(\alpha)$. Since $\delta(\alpha)\neq 0$, we have $\delta(a)\in \Oo_{v'}^{\times}$. We finally check that $\Oo_{v'}=\Oo_v[a]$. The inclusion $\Oo_v[a]\subseteq \Oo_{v'}$ is clear. For the other inclusion,  let $n=\text{deg}(m_a)$ and suppose that $c_0+c_1\cdot a+...+c_{n-1}\cdot a^{n-1}\in \Oo_{v'}$ for some $c_i\in K$ (not all zero). Let $i$ be such that $vc_i=\min_j vc_j$. Observe that 
$$\overline{c_0c_i^{-1}}+\overline{c_1c_i^{-1}}\cdot \alpha+...+\overline{c_{n-1}c_i^{-1}}\cdot \alpha^{n-1}  \neq 0$$
since the left hand side is a non-zero polynomial in $\alpha$ with coefficients in $k_v$ of degree $<n$. It follows that
$$c_0c_i^{-1}+c_1c_i^{-1}\cdot a+...+c_{n-1}c_i^{-1}\cdot a^{n-1} \in \Oo_{v'}^{\times}$$ 
and therefore 
$$v(c_i)=v'(c_0+c_1\cdot a+...+c_{n-1}\cdot a^{n-1})\geq 0$$ 
We conclude that $v(c_j)\geq 0$ for all $j$ and $\Oo_{v'}=\Oo_v[a]$. \\
(ii)$\Rightarrow$(iii): Clear.\\
(iii)$\iff$(iv):  For any field $K$ and $a\in K^s$ with minimal polynomial $m_a(X)\in K[X]$, one checks that $\text{disc}(m_a)=m_a'(a)^2=\delta(a)^2$. The rest is easy.\\
(iii)$\Rightarrow$(i): By a variant of Hensel's Lemma (see Theorem 4.1.3 \cite{engprest}) and  since $m_a(X) \in \Oo_v[X]$ is irreducible, we get $\overline{m_a}(X)=f(X)^s$ for some irreducible polynomial $f(X)\in k_v[X]$. Let $\alpha$ be the reduction of $a$ modulo $\mathfrak{m}_{v'}$, so $f(\alpha)=0$. Since $0\neq  \overline{m_a'(a)}= s\cdot f^{s-1}(\alpha)\cdot f'(\alpha)$, this forces $s=1$ and $f'(\alpha)\neq 0$. This implies that $\overline{m_a}(X)=m_{\alpha}(X)$ and $m_{\alpha}(X)$ is separable. It easily follows that $(K',v')/(K,v)$ is unramified.\\
We have established (i)$\iff$(ii)$\iff$(iii)$\iff$(iv). \\
(ii)$\Rightarrow$(v): It is clear that $\Oo_{v}[a]$ is a finitely generated $\Oo_v$-algebra. A direct computation shows that $\Omega_{\Oo_{v'}/\Oo_v}$ is killed by $\delta(a)$. Since $\delta(a)\in \Oo_{v'}^{\times}$, we get that $\Omega_{\Oo_{v'}/\Oo_v}=0$.\\
(v)$\Rightarrow$(vi): Since $\Oo_{v'}/\Oo_v$ is a finitely generated integral extension, it is of finite presentation. 
We also have that $\Oo_{v'}$ is a flat $\Oo_v$-module because it is torsion-free over a valuation ring (see Tag 0539 \cite{sp}).
By Tag 02GV \cite{sp}, we get that $\Oo_{v'}/\Oo_v$ is \'etale. \\
(vi)$\Rightarrow$(i): 	By Chapitre X, Th\'eor\`eme 1 \cite{ray}, $K'$ is contained in the inertia field of $K$. Since $(K,v)$ is henselian, this means that $(K',v')/(K,v)$ is unramified.
\end{proof}
\section{Perfectoid and tame fields}
\subsection{Perfectoid fields} \label{localapprox}
\begin{definition}
A perfectoid field is a complete valued field $(K,v)$ of residue characteristic $p>0$ such that $\Gamma_v$ is a dense subgroup of $\mathbb{R}$ and the Frobenius map $\Phi : \mathcal{O}_K/(p) \to \mathcal{O}_K/(p):x\mapsto x^p$ is surjective.
\end{definition}
\begin{rem} \label{remarksforperfectoid}
\begin{enumerate}[label=(\roman*)]
\item Modulo the other conditions, the last condition is equivalent to $\Phi:\Oo_K/(\varpi)\to \Oo_K/(\varpi):x\mapsto x^p$ being surjective for some (equivalently any) $\varpi\in \mathfrak{m}_v$ with $0<v\varpi \leq vp$.
\item  In characteristic $p$, a perfectoid field is simply a perfect, complete non-archimedean valued field of rank $1$. 
\end{enumerate}
\end{rem}
 
\begin{example}
\begin{enumerate}[label=(\roman*)]
\item The $p$-adic completions of $\Q_p(p^{1/p^{\infty}})$, $\Q_p(\zeta_{p^{\infty}})$ and $\Q_p^{ab}$ are mixed characteristic perfectoid fields.
\item  The $t$-adic completions of $\F_p(\!(t)\!)^{1/p^{\infty}}$ and $\overline{\F}_p(\!(t)\!)^{1/p^{\infty}}$ are perfectoid fields of characteristic $p$.
\end{enumerate}

\end{example}

\subsection{Tilting} \label{perfsec}
%
A construction, originally due to Fontaine, transforms any perfectoid field $K$ to a perfectoid field $K^{\flat}$ of characteristic $p$, called the \textit{tilt} of $K$. We do not insist that $\text{char}(K)=0$ but the tilting construction is only interesting in that case (we have $K^{\flat}\cong K$ if $\text{char}(K)=p$). We refer the reader to \S 3 \cite{ScholzeICM} and \S  3 \cite{Scholze} for details.

As a set, we have $K^{\flat}=\{ (x_n)_{n\in \omega} \in K^{\omega}: x_{n+1}^p=x_n\}$. We define  multiplication in $K^{\flat}$ coordinatewise 
$$(x_n)_{n\in \omega} \cdot (y_n)_{n\in \omega}=(x_n\cdot y_n)_{n\in \omega} $$ 
Addition in $K^{\flat}$ is a bit more involved and is described  by the following rule
$$ (x_n)_{n\in \omega} + (y_n)_{n\in \omega} = (z_n)_{n\in \omega} \mbox{ where }z_n=\lim_{m\to \infty} (x_{n+m}^{p^m} + y_{n+m}^{p^m})$$ 
We have the \textit{sharp} map from $K^{\flat}$ to $K$, which simply picks out $x_0$, namely
$$\sharp: K^{\flat}\to K:(x_n)_{n\in \omega} \mapsto x_0 $$
It is clearly a multiplicative morphism. We define a valuation $v^{\flat}$ on $K^{\flat}$ by 
$$v^{\flat}(x)=v(x^{\sharp})$$ 
We collect some basic facts about tilting for later use:
\bl  \label{lemscholz1}
We have that $(K^{\flat},v^{\flat})$ is a perfectoid field of characteristic $p$. Moreover:
\begin{enumerate}[label=(\roman*)]

\item  We have that 
$$\Oo_{K^{\flat}}=\{ (x_n)_{n\in \omega} \in \Oo_{K}^{\omega}: x_{n+1}^p=x_n\} \mbox{ and }\mathfrak{m}_{K^{\flat}}=\{ (x_n)_{n\in \omega} \in \mathfrak{m}_{K}^{\omega}: x_{n+1}^p=x_n\}$$
\item  We have canonical identifications 
$$\Gamma_{v^{\flat}}=\Gamma_v\mbox{ and }k_{v^{\flat}}=k_v$$
\item Let $\varpi \in \Oo_K$ be a \textit{pseudo-uniformizer}, i.e., a non-zero element  with $0<v \varpi \leq vp$. Then there exists $\varpi^{\flat}\in \Oo_{K^{\flat}}$ such that
$$\Oo_K/(\varpi)\cong \Oo_{K^{\flat}}/(\varpi^{\flat}) $$

\item The sharp map $\sharp: K^{\flat}\to K$ is a multiplicative morphism and is additive modulo $p$, i.e., $(x+y)^{\sharp}\equiv x^{\sharp}+y^{\sharp} \mod p\Oo_{K}$ for any $x,y \in \Oo_{K^{\flat}}$.

\item  If $\text{char}(K)=p$, then $(K^{\flat},v^{\flat})\cong (K,v)$.
\end{enumerate}
\el 
\begin{proof}
See Lemma 3.4 \cite{Scholze} and Lemma 3.2 \cite{ScholzeICM}.
\end{proof}

\begin{rem} [Remark 3.5 \cite{Scholze}] \label{scholzremark}
After replacing $\varpi$ with a unit multiple, we can assume that $\varpi$  comes equipped with a compatible system of $p$-power roots in $K$. We then let $\varpi^{\flat}=(\varpi,\varpi^{1/p},...)$. 
\end{rem}


For a detailed analysis of the examples below, see Corollary 4.4.3 \cite{KK1}:
\begin{example}  \label{perfex}
In the examples below, $\widehat{K}$ stands for the $p$-adic (resp. $t$-adic) completion of the field $K$ depending on whether its characteristic is $0$ or $p$.
\begin{enumerate}[label=(\roman*)]
\item $\widehat{\Q_p(p^{1/p^{\infty}})} ^{\flat}\cong \widehat{\F_p(\!(t)\!)^{1/p^{\infty}}}$ and $t^{\sharp}= p$.\\
\item $ \widehat{\Q_p(\zeta_{p^{\infty}})}^{\flat}\cong \widehat{\F_p(\!(t)\!)^{1/p^{\infty}}}$ and $(t+1)^{\sharp}= \zeta_p$.\\
\item $\widehat{\Q_p^{ab}}^{\flat} \cong \widehat{\overline{ \F}_p(\!(t)\!)^{1/p^{\infty}}}$ and $(t+1)^{\sharp}= \zeta_p$.
\end{enumerate}

\end{example}

\subsection{Tame fields} \label{tamefieldssec}
The algebra and model theory of tame fields was introduced and studied by Kuhlmann \cite{Kuhl}.  We provide an overview for the convenience of the reader, also in a slightly more general setting which will be useful later on.
\subsubsection{Basic algebra of (roughly) tame fields}
\begin{definition}
Let $(K,v)$ be a henselian valued field. A finite valued field extension $(K',v')/(K,v)$ is said to be tame if the following conditions hold:
\begin{enumerate}
\item  If $p=\text{char}(k)>0$, then $p\nmid [\Gamma':\Gamma]$. 
\item The residue field extension $k'/k$ is separable.
\item The extension $(K',v')/(K,v)$ is defectless.
\end{enumerate}
\end{definition}
\begin{definition}
A henselian valued field $(K,v)$ is said to be \textit{tame} if every finite valued field extension $(K',v')/(K,v)$ is tame. 
\end{definition}
A more intrinsic description of tame fields can be given:
\begin{fact} [Theorem 3.2 (5) \cite{Kuhl}] \label{tamechar}
Let $(K,v)$ be henselian valued field of residue characteristic exponent $p$. Then the following are equivalent:
\begin{enumerate}[label=(\roman*)]
\item $(K,v)$ is a tame field.

\item $(K,v)$ is algebraically maximal, $\Gamma$ is $p$-divisible and $k$ is perfect. 
\end{enumerate}
\end{fact} 
Kuhlmann-Rzepka \cite{KuhlBla} relax the value group condition, asking only that $\Gamma$ be \textit{roughly} $p$-divisible in the sense of Johnson \cite{Johnson}. This means that the interval $[-vp,vp]$ is $p$-divisible. 
\begin{definition} 
A valued field $(K,v)$ is called \textit{roughly tame} if it is algebraically maximal, $\Gamma$ is roughly $p$-divisible and $k$ is perfect. 
\end{definition}
\begin{rem} \label{roughtamerem}
\begin{enumerate}[label=(\roman*)]
\item  Note that $(-vp,vp)=\Gamma$ if $\text{char}(K)=p$ and $[-vp,vp]=\{0\}$ if $\text{char}(k)=0$, so roughly tame is the same as tame in equal characteristic. 
\item In residue characteristic zero, being (roughly) tame is the same as being henselian. In mixed characteristic, a roughly tame valued field is one whose core field $(K^{\circ},v^{\circ})$ is tame.
\item Roughly tame fields are defectless. Indeed, roughly tame fields of equal characteristic are just tame fields and in the mixed characteristic case the valuation is the composite of a valuation of equal characteristic zero and a tame valuation, both of which are defectless. 
\end{enumerate}
\end{rem}
\subsubsection{Model theory of roughly tame fields} \label{modeltheoryoftame}
We generalize some of Kuhlmann's results for tame fields to the broader context of roughly tame fields. We start with the henselian rationality property. Throughout the proof, we use the
notation $(L,u)^h$ for the henselization of any valued field $(L,u)$, or $u^h$ for the henselization of $u$ when the supporting field is implicit.
\begin{fact} [cf. Theorem 1.2 \cite{kuhlrat}] \label{henrat}
Let $(K,v)$ be a roughly tame field and $(F,w)/(K,v)$ be an immediate valued field extension such that $F/K$ is finitely generated of transcendence degree $1$. Then there exists $x\in F$ such that we have $(F,w)^h=(K(x),w')^h$, where $w'$ denotes the restriction of $w$ to $K(x)$.
\end{fact}
\begin{proof}
In equal characteristic, this follows directly from Theorem 1.10 \cite{Kuhl}, so we restrict ourselves to the mixed characteristic case. Write $v_0$ (resp. $w_0$) for the finest coarsening of $v$ (resp. $w$) of residue characteristic $0$ and $K^{\circ}$ (resp. $F^{\circ}$) for the residue field. We also denote by $v^{\circ}$ (resp. $w^{\circ}$) the induced valuation of $v$ (resp. $w$) on $K^{\circ}$ (resp. $F^{\circ}$). Note that $(K^{\circ},v^{\circ})$ is tame by Remark \ref{roughtamerem}(ii).

By Lemma 2.8(ii) \cite{AnscombeJahnkeNIP}, one can write $w^h=(w^h)_0\circ w^{\circ h}$, where $(w^h)_0$ is the finest coarsening of $w^h$ of residue characteristic $0$. We may thus identify the residue field of $(F^h,(w^h)_0)$ with the field obtained by taking the henselization of $F^{\circ}$ with respect to $w^\circ$. We summarize the above discussion in the diagram below
\[
  \begin{tikzcd} 
   F^h \arrow[r,"(w^h)_0"] & F^{\circ h} \arrow[r, "w^{\circ h}"] & k_{w^h} \\
 F \arrow[u, no head] \arrow[r,"w_0"] & F^{\circ} \arrow[u, no head] \arrow[r, "w^{\circ}"] & k_w \arrow[u, "=", no head]\\
  K\arrow[u, no head] \arrow[r,"v_0"] & K^\circ \arrow[u, no head] \arrow[r,"v^{\circ}"] & k_{v} \arrow[u, "=", no head] 
     \end{tikzcd}
\]
Since 
$$e(w_0/v_0)\cdot e(w^{\circ}/v^{\circ})=e(w/v)=1$$ 
we get $e(w_0/v_0)=e(w^{\circ}/v^{\circ})=1$. It follows that $(F^{ \circ},w^{\circ})/ (K^{\circ},v^{\circ})$ is immediate. Also, note that
$$\text{tr.deg}(F^{ \circ}/K^{\circ}) \leq \text{tr.deg}(F/K)\leq 1$$ 
Since $(K^{\circ},v^{\circ})$ is tame, it is algebraically maximal (see Fact \ref{tamechar}). Therefore, we must either have that $\text{tr.deg}(F^{ \circ}/K^{\circ})=1$ or $F^{ \circ}=K^{\circ}$. 

First, suppose that $\text{tr.deg}(F^{ \circ}/K^{\circ})=1$. Since $\text{tr.deg}(F^{ \circ}/K^{\circ})=\text{tr.deg}(F/K)$, $w_0$ is an Abhyankar extension of $v_0$. The residue field extension $F^{ \circ}/K^{\circ}$ is then finitely generated (see Corollary 2.3 \cite{Kuhl}).
Therefore $(F^{ \circ},w^{\circ})/ (K^{\circ},v^{\circ})$ is an immediate function field of transcendence degree $1$. By Theorem 1.10 \cite{Kuhl}, we can write $(F^{ \circ }, w^\circ)^h=(K^{\circ}(t), \nu)^h$ for some $t \in F^{\circ}$, where $\nu$ denotes the restriction of $w^\circ$ to $K^\circ(t)$. Let $x\in F$ be any lift of $t$ and let $u_0$ be the restriction of $(w^h)_0$ to $K(x)$. We have 
$$e((w^h)_0/u_0^h)\cdot e(u_0^h/v_0)=e((w^h)_0/v_0)=e(w_0/v_0)=1$$ 
This gives $e((w^h)_0/u_0^h)=1$ and hence $(F^h,(w^h)_0)/(K(x)^h,u_0^h)$ is an immediate algebraic extension. Since $(K(x)^h,u_0^h)$ is henselian of equal characteristic $0$, it is algebraically maximal (see Corollary 4.22 \cite{vdd}). This forces $(K(x),w')^h=(F,w)^h$, where $w'$ denotes the restriction of $w$ to $K(x)$, as required. 

Finally, suppose that $F^{ \circ}=K^{\circ}$. Let $x\in F$ be any element transcendental over $K$ and $u_0$ be the restriction of $w_0$ to $K(x)$. Then $(F^{h},(w^h)_0)/(K(x)^h,u_0^h)$ is an immediate algebraic extension. Arguing as before, this yields $(F,w)^h=(K(x),w')^h$.
\end{proof}

\bl \label{tamelemma}
Let $(K,v)\subseteq (K',v')$ be two tame fields of mixed characteristic $(0,p)$ with $k_v=k_{v'}$. Suppose $\zeta_p \in K$. Then the natural map 
$$K^{\times}/K^{\times p}\to K'^{\times}/K'^{\times p}:a\cdot K^{\times p} \mapsto a \cdot K'^{\times p}$$ 
is bijective.
\el 
\begin{proof} 
Let $a \in K^{\times } \backslash K^{\times p}$ and consider the Kummer extension $K(a^{1/p})/K$. Since $(K,v)$ is defectless, $\Gamma_v$ is $p$-divisible and $\zeta_p\in K$, the extension $K(a^{1/p})/K$ is an unramified Galois extension of degree $p$. It therefore comes from a Galois extension of $k_v$ of degree $p$. By Artin-Schreier theory, the latter is an Artin-Schreier extension. This yields a canonical bijection
$$K^{\times}/K^{\times p} \to k_v^+/\wp(k_v^+): a\cdot K^{\times p} \mapsto \alpha +\wp(k_v) $$  
where $\alpha$ is such that the residue field of $K(a^{1/p})$ is the splitting field of $\wp_\alpha(X)=X^p-X-\alpha$. The above analysis also applies to $K'$. We thus have a commutative diagram 
\[
  \begin{tikzcd} 
 K^{\times}/K^{\times p} \arrow[r] \arrow[d, "\cong"] &  K'^{\times}/K'^{\times p} \arrow[d] \arrow[d, "\cong"] \\
  k_v^+/\wp(k_v^+) \arrow[r] & k_{v'}^+/\wp(k_{v'}^+) 
     \end{tikzcd}
\]
where the vertical arrows are bijections. Since $k_v=k_{v'}$, the bottom horizontal arrow is also a bijection. It follows that the top horizontal arrow is a bijection. 
\end{proof}
\bl \label{ptorsionfree}
Let $(K,v)\subseteq (K',v')$ be two valued fields of mixed characteristic $(0,p)$. Set $K_1=K(\zeta_p)$ (resp. $K_1'=K'(\zeta_p)$) and let $v_1$ (resp. $v_1'$) be a valuation on $K_1$ (resp. $K_1'$) extending $v$ (resp. $v'$). Suppose that $\Gamma_{v_1'}/\Gamma_{v_1}$ is $p$-torsion-free. Then $\Gamma_{v'}/\Gamma_v$ is $p$-torsion-free. 
\el 
\begin{proof}
Let $\gamma \in \Gamma_{v'}$ be such that $p\cdot \gamma \in \Gamma_v$. We have $\Gamma_v\subseteq \Gamma_{v_1}$ and $\Gamma_{v'}\subseteq \Gamma_{v_1'}$. Since $\Gamma_{v_1'}/\Gamma_{v_1}$ is $p$-torsion-free and $p\cdot \gamma \in \Gamma_{v_1}$, this implies that $\gamma \in \Gamma_{v_1}$. Since $[\Gamma_{v_1}:\Gamma_v] \mid p-1$, we get that $(p-1) \cdot \gamma \in \Gamma_v$. Finally, note that $\gamma=p\cdot \gamma -(p-1)\cdot \gamma$ and hence $\gamma\in \Gamma_v$.
\end{proof}
\bl [cf. Lemma 3.7 \cite{Kuhl}] \label{RAC}
Let $(K,v)\subseteq (K',v')$ be two valued fields such that $(K',v')$ is roughly tame and $K$ is relatively algebraically closed in $K'$. Suppose that $k_{v'}/k_v$ is algebraic. Then $(K,v)$ is roughly tame and moreover $\Gamma_{v'}/\Gamma_v$ is torsion-free and $k_{v'}=k_v$.
\el 
\begin{proof}
In equal characteristic, this follows directly from Theorem 1.10 \cite{Kuhl}. 
We may thus restrict ourselves to the mixed characteristic case. With the notation of Remark \ref{corefield}, we consider the decompositions below
\[
  \begin{tikzcd} 
 K' \arrow[r,"v_0'"] & K'^\circ \arrow[r, "v'^{\circ}"] & k_{v'}  \\
  K\arrow[u, no head] \arrow[r,"v_0"] & K^\circ \arrow[u, no head] \arrow[r,"v^{\circ}"] & k_{v} \arrow[u, no head] 
     \end{tikzcd}
\]
Since $K$ is relatively algebraically closed in $K'$ and $(K',v')$ is henselian, we get that $(K,v)$ is henselian. By Corollary 4.1.4 \cite{engprest}, it follows that $(K,v_0)$ is henselian. Since $K$ is relatively algebraically in $K'$ and $K'^{\circ}/K^{\circ}$ is separable, we get that $K^{\circ}$ is relatively algebraically closed in $K'^{\circ}$. Since $(K',v')$ is roughly tame, we have that $(K'^{\circ},v'^{\circ})$ is tame. By Lemma 3.7 \cite{Kuhl}, we get that $(K^{\circ},v^{\circ})$ is tame and moreover that $k_{v'}=k_v$.  Equivalently, we have that $(K,v)$ is roughly tame by Remark \ref{roughtamerem}(ii). 

We claim that $\Gamma_{v'}/\Gamma_v$ is torsion-free. Showing that $\Gamma_{v'}/\Gamma_v$ is $\ell$-torsion-free for $\ell\neq p$ prime is an immediate consequence of henselianity and is left to the reader. We now prove that $\Gamma_{v'}/\Gamma_v$ is $p$-torsion-free. By Lemma \ref{ptorsionfree}, we may assume that $\zeta_p\in K$. Suppose $b^p=a\cdot u'$, where $b\in K'$, $a\in K$ and $u' \in \Oo_{v'}^{\times}$. Write $\overline{u'} \in K'^{\circ}$ for the reduction of $u'$ modulo $\mathfrak{m}_{v_0'}$. Note that $\overline{u'} \neq 0$ because $\Oo_{v'}^{\times}\subseteq \Oo_{v'_0}^{\times}$. Applying Lemma \ref{tamelemma} to $(K'^{\circ},v'^{\circ})/(K^{\circ},v^{\circ})$, we see that there is $u\in K^{\times}$ such that $\overline{u'/u} \in K'^{\circ \times p}$. By Hensel's Lemma for $(K',v_0')$, we get that $u'/u\in K'^{\times p}$. It follows that there is $c\in K'^{\times}$ such that $c^p=a\cdot u$. Since $K$ is relatively algebraically closed in $K'$, this forces $au \in K^{\times p}$. Since $v'b=v'c$, It follows that $v'b\in \Gamma_v$.
\end{proof}
We now prove the \textit{relative embedding property} for roughly tame fields:
\begin{fact} [cf. Theorem 7.1 \cite{Kuhl}] \label{relembprop}
Let $(K,v)$ be a defectless valued field. Let also $(K^*,v^*)$ and $(K',v')$ be two roughly tame fields extending $(K,v)$. Suppose that: 
\begin{itemize}
\item $(K^\ast,v^*)$ is $|K'|^+$-saturated.
\item $\Gamma_{v'}/\Gamma_v$ is torsion-free and $k_{v'}/k_{v}$ separable.
\item There exist embeddings 
$$\rho: \Gamma_{v'} \to\Gamma_{v^*}\mbox{ over }\Gamma_v\mbox{ and }\sigma:k_{v'}\to k_{v^*} \mbox{ over }k_v$$ 
\end{itemize}
Then there is also a valued field embedding $\Phi:(K',v')\to (K^*,v^*)$ over $(K,v)$, which respects $\rho$ and $\sigma$.
\end{fact}
\begin{proof}
By Fact \ref{henrat} and Lemma \ref{RAC} and Lemma 6.4 \cite{Kuhl}.
\end{proof}
From the relative embedding property, many other model-theoretic properties automatically follow. For instance:
\begin{fact} [cf. Theorem 7.1 \cite{Kuhl}] \label{akerelsub}
Let $(K,v)$ be a defectless valued field. Let also $(K^*,v^*)$ and $(K',v')$ be two roughly tame fields extending $(K,v)$. Suppose  that $\Gamma_{v'}/\Gamma_v$ is torsion-free and $k_{v'}/k_{v}$ is separable. Then the following are equivalent: 
\begin{enumerate}[label=(\roman*)]
\item $(K^*,v^*)\equiv_{(K,v)} (K',v')$ in $L_{\text{val}}$. 

\item $\Gamma_{v^*}\equiv_{\Gamma_{v}} \Gamma_{v'}$ in $L_{\text{oag}}$ and $k_{v^*}\equiv_{k_{v}} k_{v'}$ in $L_{\text{rings}}$.
\end{enumerate}
Moreover, this equivalence is even resplendent: 
the result still holds true for expansions of
$L_\text{val}$ which arise by adding further structure to
$L_\text{oag}$ on the value group or
to $L_\textrm{rings}$ 
on the residue
field.
\end{fact}
\begin{proof}
The equivalence follows 
by Fact \ref{relembprop} and Lemma 6.1 \cite{Kuhl}.
The ``moreover'' part is also an immediate consequence
of Fact \ref{relembprop}. As this is not explicitely stated 
in \cite{Kuhl}, we sketch the proof. Assume that we have expansions
 $L_g$ of $L_\text{oag}$ and $L_r$ of $L_\text{rings}$ and
such that $\Gamma_{v^*}\equiv_{\Gamma_{v}} \Gamma_{v'}$ in $L_g$ and $k_{v^*}\equiv_{k_{v}} k_{v'}$ in $L_r$. Let $L$ be the corresponding
expansion of $L_\text{val}$. 

By Fact \ref{relembprop}, $(K',v')$ embeds as an 
$L$-structure into a $|K'|^+$-saturated elementary extension
$(K^{*}_1,v^{*}_1)$
of $(K^*,v^*)$ over $(K,v)$. Likewise, $(K^{*}_1,v^{*}_1)$ embeds as an 
$L$-structure into a $|K_1^*|^+$-saturated elementary extension
$(K'_1,v'_1)$ of $(K',v')$ over $(K',v')$. Continuing in this fashion, we construct an increasing chain of $L$-structures 
$$ (K^*_1,v^*_1)\subseteq (K_1',v_1') \subseteq (K^*_2,v^*_2) \subseteq (K_2',v_2')\subseteq ...   $$
such that $(K'_i,v'_i)\preceq (K'_{i+1},v'_{i+1})$ and $(K^*_i,v^*_i)\preceq (K^*_{i+1},v_{i+1}^*)$ as $L$-structures. We let 
$$(K_\infty,v_\infty)=\bigcup_{n=1}^{\infty} (K^*_n,v^*_n) = \bigcup_{n=1 }^{\infty} (K'_n,v'_n). $$  
By Tarski-Vaught, we get that $(K_\infty,v_\infty)$ is an elementary extension both of  $(K',v')$ and $(K^*,v^*)$ in $L$. It follows that $(K^*,v^*)\equiv_{(K,v)} (K',v')$ in $L$. 
\end{proof}
If the base valued field is $\F_p$ equipped with the trivial valuation, one gets:
\begin{fact} [Theorem 1.4 \cite{Kuhl}] \label{akekuhl}
Let $(K,v)$ and $(K',v')$ be two positive characteristic tame fields. Then $(K,v)\equiv (K',v')$  in $L_{\text{val}}$ if and only if $k_v\equiv k_{v'}$  in $L_{\text{rings}}$ and $\Gamma_v\equiv \Gamma_{v'}$  in $L_{\text{oag}}$.
\end{fact}
The analogous statement in mixed characteristic is not true (cf. Remark \ref{tameakerem}).
\begin{fact} [cf. Theorem 1.4 \cite{Kuhl}] \label{precversion}
Let $(K,v)$ be a roughly tame field and $(K,v)\subseteq (K',v')$. Then the following are equivalent:
\begin{enumerate}[label=(\roman*)]
\item $(K,v)\preceq_{\exists} (K',v')$ in $L_{\text{val}}$.

\item $k_v \preceq_{\exists} k_{v'}$ in $L_{\text{rings}}$ and $\Gamma_v\preceq_{\exists} \Gamma_{v'}$  in $L_{\text{oag}}$.
\end{enumerate}
\end{fact}
\begin{proof}
By Fact \ref{relembprop} and Lemma 6.1 \cite{Kuhl}.
\end{proof}

\section{Taming perfectoid fields} \label{tamingperfectoids}
We \quotes{tame} certain valued fields by means of first passing to an elementary extension and then finding a (roughly) tame coarsening of the valuation. This includes the case of perfectoid fields, for which the proof was outlined in \S \ref{PARTI}.

\subsection{Extensions with small differents}
Recall from \ref{unramext}, that given a field $K$ and $a\in \overline{K}$ with minimal polynomial $m_a(X)\in K[X]$, we write $\delta(a)=m_a'(a)$ for the \textit{different} of the element $a$. 
\bl \label{frobtrick}
Let $(K,v)$ be a perfect valued field of characteristic $p>0$. Suppose $t\in \mathfrak{m}_v$ is such that $\Z vt$ is cofinal in $\Gamma_v$. Let $(K',v')/(K,v)$ be a finite extension. Then there exists $\alpha \in \Oo_{v'}$ with $0\leq v'(\delta(\alpha))\leq  vt$ and $K'=K(\alpha)$.
\el
\begin{proof}
Since $K'/K$ is finite separable, we can write $K'=K(\beta)$ for some $\beta \in \Oo_{v'}$ such that $\delta(\beta)\neq 0$. Since $\Z vt$ is cofinal in $\Gamma_v$, we have $0\leq  v'(\delta(\beta))\leq N\cdot v t$ for some $N\in \N$. Since $K'$ is also perfect, we also have $K'=K(\beta^{1/p^M})$ for any $M \in \N$. The coefficients of $m_{\beta^{1/p^M}}(X)$ of $\beta^{1/p^M}$ are $p^M$-th roots of the coefficients of $m_{\beta}(X)$ of $\beta$. We then compute that
$$\delta(\beta^{1/p^M})=m_{\beta^{1/p^M}}'(\beta^{1/p^M} )=m_{\beta}'(\beta)^{1/p^M}=\delta(\beta)^{1/p^M}$$ 
For $M\geq \log_p (N)$, we will thus have that 
$$ 0\leq  v'(\delta(\beta^{1/p^M}))\leq v t $$ 
It follows that $\alpha=\beta^{1/p^M}$ has the desired properties.
\end{proof}
\bl \label{trickyelementaryext}
Let $(K,v)$ be a henselian valued field of residue characteristic $p>0$. Suppose $\varpi\in \mathfrak{m}_v$ is such that the residue field of the valuation ring $\Oo_{v}[\varpi^{-1}]$ is perfect of characteristic $p$. Let $(K',v')/(K,v)$ be a finite extension. Then the following are equivalent: 
\begin{enumerate}[label=(\roman*)]
    \item $K'/K$ is unramified with respect to $\Oo_v[\varpi^{-1}]$.
    \item There exists $a \in \Oo_{v'}$ with $0\leq v'(\delta(a))\leq  v\varpi$ and $K'=K(a)$.
\end{enumerate}
\el 
\begin{proof}
Write $\Oo_{v^+}=\Oo_v[\varpi^{-1}]$ (resp. $\Oo_{v'^+}=\Oo_{v'}[\varpi^{-1}]$).\\
(i)$\Rightarrow $(ii): Consider the decompositions below
\[
  \begin{tikzcd}
  K' \arrow[r, "v'^+"] & k_{v'^+} \arrow[r, "\overline{v'}"] & k_{v'}   \\
 K \arrow[r, "v^+"]\arrow[u, no head]   & k_{v^+}  \arrow[r, "\overline{v}"]  \arrow[u, no head]   & k_{v}  \arrow[u, no head]  
     \end{tikzcd}
\]
Set $t=\varpi \mod \mathfrak{m}_{v^+}$. 
By Lemma \ref{frobtrick}, we can write $k_{v'^+}=k_{v^+}(\alpha)$ for some $\alpha \in \Oo_{\overline{v'}}$ with $0\leq \overline{v'}(\delta(\alpha))\leq  \overline{v}t$. Let $a\in \Oo_{v'^+}$ be any lift of $\alpha$. Since $\alpha \in \Oo_{\overline{v'}}$, we get that $a\in \Oo_{v'}$. Since $[K':K]=[k_{v'^+}:k_{v^+}]$, we have that $K'=K(a)$. Note that $m_a(X)$ reduces to $m_{\alpha}(X)$  modulo $\mathfrak{m}_{v'^+}$. Therefore $\delta(\alpha)=\delta(a) \mod \mathfrak{m}_{v'^+}$ and hence $0\leq  v'(\delta(a)) \leq  v\varpi $. It follows that $a$ has the desired properties.\\
(ii)$\Rightarrow$(i): Observe that $\delta(a) \in \Oo_{v'^+}^{\times}$ and conclude from Fact \ref{unramfact}(iii)$\Rightarrow$(i).
\end{proof}
%
\bl \label{tamingmixed}
Let $(K,v)$ be a henselian valued field of residue characteristic $p>0$ such that $\Oo_v/(p)$ is semi-perfect. Let $(K',v')/(K,v)$ be a finite extension with $e(v'/v)=1$. Then for each $m\in \N^{>0}$, there exists $a \in \Oo_{v'}$ with $0\leq v'(\delta(a))\leq 1/m\cdot  vp$ and $K'=K(a)$.
\el 
\begin{proof}
We assume that $\text{char}(K)=0$; if $\text{char}(K)=p$ the conclusion is clear. For fixed $m$, the desired conclusion can be expressed as a first-order sentence in $L_{\text{val}}$ with parameters for the coefficients of a minimal polynomial for $K'/K$ (cf. the proof of Lemma 2.4 \cite{AnscombeJahnkeNIP}). We can therefore assume that $(K,v)$ is $\aleph_1$-saturated. 

Consider the standard decompositions for $K$ and $K'$ with respect to $p$ 
\[
  \begin{tikzcd} 
 K' \arrow[r,"v_0'"] & K'^\circ \arrow[r, "\overline{w'}"] & k_{w'}  \arrow[r, "\overline{v'}"] & k_{v'}  \\
  K\arrow[u, no head] \arrow[r,"v_0"] & K^\circ \arrow[u, no head] \arrow[r,"\overline{w}"] & k_{w}  \arrow[r,"\overline{v}"]  \arrow[u, no head]  & k_{v} \arrow[u, no head] 
     \end{tikzcd}
\]
By Lemma \ref{sphericalcompletelemma} and Fact \ref{sphericaldefect}, we get that $\overline{w}$ is henselian defectless. Then $w=v_0\circ \overline{w}$ is also henselian defectless, being a composite of such valuations. The semi-perfect ring $\Oo_v/(p)$ surjects onto $\Oo_{\overline{v}}$ and thus $k_w=\text{Frac}(\Oo_{\overline{v}})$ is perfect. Therefore $k_{w'}/k_w$ is finite separable. Since $e(w'/w)\cdot e(\overline{v'}/\overline{v})=e(v'/v)=1$, we also get that $e(w'/w)=1$.
We conclude that $(K',w')/(K,w)$ is unramified. By Fact \ref{unramfact}(ii), there exists $a\in \Oo_{w'}$ such that $\delta(a) \in \Oo_{w'}^{\times}$ and $K'=K(a)$. We can even assume $a\in \Oo_{v'}$ by running the proof of Fact \ref{unramfact}(i)$\Rightarrow$(ii) with $\alpha \in \Oo_{\overline{v}}$. Since $\delta(a)\in \Oo_{w'^{\times}}$, we have in particular that $0\leq v'(\delta(a))\leq 1/m\cdot  vp$. Therefore the element $a$ has the desired properties.
\end{proof}
The next result is crucial in what follows. We stress that it is valid (and useful) even in the degenerate where $\Oo_v[\varpi^{-1}]=K$.
\bp \label{equivdefectless}
Let $(K,v)$ be a henselian valued field of residue characteristic $p>0$ such that $\Oo_v/(p)$ is semi-perfect. Given $\varpi \in \mathfrak{m}_v$, the following are equivalent: 
\begin{enumerate}[label=(\roman*)]
\item $\Oo_v[\varpi^{-1}]$ is defectless.  

\item For any finite extension $(K',v')/(K,v)$ with $e(v'/v)=1$, there exists $a \in \Oo_{v'}$ with $0\leq v(\delta(a))\leq  v\varpi$ and $K'=K(a)$.
\item $\Oo_v[\varpi^{-1}]$ is algebraically maximal.
\end{enumerate}
\ep  
\begin{proof}
If $vp<m\cdot v \varpi$ for some $m\in \N$, then all clauses hold automatically. Indeed, (i) and (iii) hold because $\Oo_{v}[\varpi^{-1}]$ is of residue characteristic zero and (ii) holds by Lemma \ref{tamingmixed}. We therefore assume that $vp>\Z v \varpi$ for the rest of the proof. 

Write $\Oo_{v^+}=\Oo_v[\varpi^{-1}]$ and consider the decomposition
$$K \stackrel{v_0}\longrightarrow K^\circ \stackrel{\overline{v^+}} \longrightarrow k_{v^+} \stackrel{\overline{v}} \longrightarrow k_{v} $$ 
where $v_0$ is the finest coarsening of residue characteristic $0$ in case $\text{char}(K)=0$ and is the trivial valuation in case $\text{char}(K)=p$. Since $vp>\Z v\varpi$, we have that $p\Oo_v \subseteq \mathfrak{m}_{v^+}$ and hence the semi-perfect ring $\Oo_v/(p)$ surjects onto $\Oo_{\overline{v}}=\Oo_v/\mathfrak{m}_{v^+}$. It follows that $k_{v^+}$ is a perfect field of characteristic $p$. We also write $v^{\circ}$ for the valuation induced by $v$ on $K^{\circ}$. Since $vp$ is not minimal positive in $\Gamma_v$ and $\Oo_v/(p)$ is semi-perfect, we get that $\Gamma_{v^{\circ}}$ is $p$-divisible.  \\
(i)$\Rightarrow$(ii): Let $(K',v')/(K,v)$ be such that $e(v'/v)=1$. Consider an analogous decomposition for $K'$ to the one above for $K$. Since 
$$e(v'^+/v^+)\cdot e(\overline{v'}/\overline{v})=e(v'/v)=1$$ 
we get that $e(v'^+/v^+)=1$. The extension $\Oo_{v'^+}/\Oo_{v^+}$ is defectless by assumption. Since $k_{v^+}$ is perfect, the extension $k_{v'^+}/k_{v^+}$ is finite separable. This means that $K'/K$ is unramified with respect to $\Oo_{v^+}$. We conclude from Lemma \ref{trickyelementaryext}(i)$\Rightarrow$(ii). \\
(ii)$\Rightarrow$(iii): Let $(K',v'^+)/(K,v^+)$ be an immediate extension and consider the diagram of places
\[
  \begin{tikzcd} 
 K' \arrow[r,"v_0'"] & K'^\circ \arrow[r, "\overline{v'^+}"] & k_{v'^+}  \arrow[r, "\overline{v'}"] & k_{v'}  \\
  K\arrow[u, no head] \arrow[r,"v_0"] & K^\circ \arrow[u, no head] \arrow[r,"\overline{v^+}"] & k_{v^+}  \arrow[r,"\overline{v}"]  \arrow[u, no head]  & k_{v} \arrow[u, no head] 
     \end{tikzcd}
\]
Since $e(v_0'/v_0)\cdot e(\overline{v'^+}/\overline{v^+})=e(v'^+/v^+)=1$, we get $e(v_0'/v_0)=1$. Since  $[K':K]$ is a power of $p$ and $\Gamma_{v^{\circ}}$ is $p$-divisible, we get that $e(v'^{\circ}/v^{\circ})=1$. This implies that 
$$e(v'/v)=e(v_0'/v_0)\cdot e(v'^{\circ}/v^{\circ})=1$$ 
By (ii), there exists $a \in \Oo_{v'}$ such that $0\leq v(\delta(a))\leq  v\varpi$ and $K'=K(a)$. By Lemma \ref{trickyelementaryext}(ii)$\Rightarrow$(i), we get that $(K',v'^+)/(K,v^+)$ is unramified. This forces $K=K'$.\\
(iii)$\Rightarrow$(i): 
By assumption $(K,v^+)$ is algebraically maximal. Note that $\Gamma_{v^+}$ is roughly $p$-divisible, being a quotient of the roughly $p$-divisible group $\Gamma_{v}$. Recall that $k_{v^+}$ is perfect. We conclude that $(K,v^+)$ is roughly tame. By Remark \ref{roughtamerem}(iii), roughly tame fields are defectless.
\end{proof}
\subsection{An elementary class of \quotes{almost tame} fields} \label{nicedeeplyram}
\begin{definition} \label{classc}
Fix a prime $p$. Let $\mathcal{C}$ be the class of valued fields $(K,v)$ together with a distinguished $\varpi \in \mathfrak{m}_v \backslash \{0\}$, such that: 
\begin{enumerate}
\item $(K,v)$ is a henselian valued field of residue characteristic $p$.

\item The ring $\Oo_v/(p)$ is semi-perfect.

\item The valuation ring $\Oo_v[\varpi^{-1}]$ is algebraically maximal.
\end{enumerate}
Structures in $\mathcal{C}$ will naturally be construed as $L_{\text{val}}(\varpi)$-structures. 
\end{definition} 

Given a language $L$ and a class $\mathcal{D}$ of $L$-structures, we say that $\mathcal{D}$ is an \textit{elementary class} in $L$ if $\mathcal{D}$ is the class of models of some $L$-theory $T$. It may not be immediately obvious that $\mathcal{C}$ is an elementary class. This is where our work from the previous section will pay off. 
\bp \label{defectelementary}
The class $\mathcal{C}$ is an elementary class in $L_{\text{val}}(\varpi)$. In particular, it is closed under ultraproducts and elementary equivalence.
\ep 
\begin{proof}
For condition (3) of $\mathcal{C}$ use Proposition \ref{equivdefectless}(ii). The rest is standard. 
\end{proof}
Beware that condition (3) is not elementary, even in conjunction with (1). Let us give an example of an elementary extension of henselian fields $(K_0,v_0)\preceq (K,v)$ and $t\in \mathfrak{m}_{v_0}$ such that $\Oo_{v_0}[t^{-1}]$ is algebraically maximal but $\Oo_v[t^{-1}]$ is not. 
\begin{example}
We start with a construction due to F. K. Schmidt (cf. Example 3.1 \cite{kuhlmann2010defect}). Let $s\in \F_p(\!(t^p)\!)$ be transcendental over $\F_p(t)$ and $K_0$ be the henselization of $\F_p(t,s)$ inside $\F_p(\!(t)\!)$. Then $K_0(s^{1/p})/K_0$ is an inseparable immediate extension, so $(K_0,v_0)$ is not algebraically maximal. On the other hand, $\Oo_{v_0}[t^{-1}]=K_0$ is trivial hence algebraically maximal.  

Now let $(K,v)$ be an $\aleph_1$-saturated elementary extension of $(K_0,v_0)$. Write $\Oo_{v^+}=\Oo_v[t^{-1}]$ and  $\overline{v}$ for the induced valuation of $v$ on $k_{v^+}$. Note that $(k_{v^+},\overline{v})$ is discrete. By Fact \ref{sphericalcompletelemma}, we get that $(k_{v^+},\overline{v})$ is (spherically) complete, hence isomorphic to $(\F_p(\!(z)\!),v_z)$. We thus have a decomposition as follows
\[
  \begin{tikzcd} 
 K \arrow[r, "v^+"] &\F_p(\!(z)\!) \arrow[r, "v_z"] & \F_p  
     \end{tikzcd}
\]
Note that $v$ is not algebraically maximal because $K(s^{1/p})/K$ is an immediate extension. Since $v=v^+\circ v_z$ and $v_z$ is algebraically maximal, this implies that $v^+$ is not algebraically maximal. 
\end{example}
The class $\mathcal{C}$ may be viewed as an interpolation between perfectoid fields and henselian (roughly) deeply ramified fields (cf.  \S 6 \cite{GR}, \cite{KuhlBla}):
\begin{rem} \label{comparisonwithotherclasses}

\begin{enumerate}[label=(\roman*)]
\item Any perfectoid field $(K,v)$ equipped with any $\varpi \in \mathfrak{m}_v \backslash \{0\}$ is in $\mathcal{C}$. Note that the trivial valuation ring $\Oo_v[\varpi^{-1}]=K$ is automatically algebraically maximal.

  \item  Let $(K,v)$ be a mixed characteristic valued field with $vp$ not being minimal positive in $\Gamma_v$. Then $(K,v)$ together with $\varpi=p$ is in $\mathcal{C}$ if and only if $(K,v)$ is henselian and roughly deeply ramified. Deeply ramified fields are a bit more restrictive in that discrete archimedean components are not allowed in $\Gamma_v$. 
   
\item Condition (2) forces valued fields in $\mathcal{C}$ to be perfect. Positive characteristic henselian deeply ramified fields need not be perfect, only their completions are. 

\item The existence of $\varpi \in \mathfrak{m}_v \backslash \{0\}$ such that $\Oo_v[\varpi^{-1}]$ is algebraically maximal is not guaranteed for henselian deeply ramified fields of positive characteristic.  
\end{enumerate}
These deviations from the class of deeply ramified fields are actually crucial for having Ax-Kochen/Ershov principles down to $\Oo_v/(\varpi)$ and value group; see Observation \ref{wrongake} and Example \ref{positiverem}.
\end{rem}
\bl \label{trickremark}
Suppose that $(K,v)$ together with $\varpi$ is in $\mathcal{C}$. Let $w$ be any coarsening of $v$ with $w\varpi=0$. Write $\overline{v}$ for the induced valuation of $v$ on $k_w$ and $t=\varpi \mod \mathfrak{m}_w$. Then $(k_w,\overline{v})$ together with $t$ is in $\mathcal{C}$. 
\el 
\begin{proof}
Since $(K,v)$ is henselian, the same is true for $(k_w,\overline{v})$. Moreover, the residue ring $\Oo_{\overline{v}}/(t)$ is semi-perfect because $\Oo_v/(\varpi)$ maps isomorphically to it via the residue map of $w$.
The only non-trivial thing to check is that $\Oo_{\overline{v}}[t^{-1}]$ is algebraically maximal. Write $\Oo_{v^+}=\Oo_v[\varpi^{-1}]$ and $\Oo_{\overline{v}^+}=\Oo_{\overline{v}}[t^{-1}]$. 
Consider the diagram of places below
\[
\begin{tikzcd}[column sep=3.5em]
K \arrow[rr, bend left, "v^+"] \arrow[r, "w"] & k_w \arrow[r, "\overline{v}^+"] & k_{v^+}
\end{tikzcd}
\]
Since $\Oo_{v^+}$ is algebraically maximal and $v^+=w\circ \overline{v}^+$, it follows easily that $\Oo_{\overline{v}^+}$ is also algebraically maximal. 
\end{proof}
\bc \label{trickremark1}
\begin{enumerate}[label=(\roman*)]
\item For each $i\in \omega$, let $(K_i,v_i)$ be a perfectoid field of a fixed residue characteristic $p>0$ and $\varpi_i\in \mathfrak{m}_{v_i}$. Let $(K_U,v_U)$ be an ultraproduct and $\pi=\ulim \varpi_i$. Then $(K_U,v_U)$ equipped with $\pi$ is in $\mathcal{C}$. 
\item Let $(K,v)$ be a perfectoid field and $(K_U,v_U)$ be a non-principal ultrapower. Then $(K_U,v_U)$ equipped with any $\pi \in \mathfrak{m}_{v_U}$ is in $\mathcal{C}$. 
\item In the setting of (ii), let $w$ be any coarsening of $v_U$ with $w\pi=0$. Write $\overline{v}$ for the induced valuation of $v_U$ on $k_w$ and $t=\pi \mod \mathfrak{m}_w$ . Then $(k_w,\overline{v})$ equipped with $t$ is also in $\mathcal{C}$. 
\end{enumerate}
\ec 
\begin{proof}
For (i) use Proposition \ref{defectelementary} and Remark \ref{comparisonwithotherclasses}(i). Part (ii) is a special case of (i) for $K_i=K$. Part (iii) follows from (ii) and Lemma \ref{trickremark}.
\end{proof}
\subsection{Non-standard Tate/Gabber-Ramero}
Recall the definition of roughly tame fields from \S \ref{modeltheoryoftame} and of regularly dense ordered abelian groups from \S \ref{oagsection}.
\bt \label{nonstdtategr}
Let $(K_0,v_0)$ be a henselian valued field of residue characteristic $p>0$ such that $\Oo_{v_0}/(p)$ is semi-perfect and $\varpi\in \mathfrak{m}_{v_0}$ be such that $\Oo_{v_0}[\varpi^{-1}]$ is algebraically maximal. Let $(K,v)$ be an $\aleph_1$-saturated elementary extension of $(K_0,v_0)$ and $w$ be the coarsest coarsening of $v$ with $w\varpi>0$. Then: 
\begin{enumerate}[label=(\roman*)]
\item If $v_0p$ is not minimal positive in $\Gamma_{v_0}$, the valued field $(K,w)$ is roughly tame.

\item If $\Gamma_{v_0}$ is regularly dense, every finite extension $(K',w')/(K,w)$ is unramified. Equivalently, the extension $\Oo_{w'}/\Oo_w$ is finite \'etale. 
\end{enumerate}
 
\et  
\begin{proof}
By Proposition \ref{defectelementary}, we know that $(K,v)$ is also a henselian valued field of residue characteristic $p$, with $\Oo_{v}/(p)$ semi-perfect and that $\Oo_{v}[\varpi^{-1}]$ is algebraically maximal. \\
(i) Write $\Oo_{w^+}=\Oo_w[\varpi^{-1}]$ and consider the standard decomposition with respect to $\varpi$ below
$$K\stackrel{w^+} \longrightarrow k_{w^+} \stackrel{\overline{w}} \longrightarrow k_{w} \stackrel{\overline{v}}\longrightarrow k_{v} $$ 
By assumption, we have that $w^+$ is algebraically maximal.
By Lemmas \ref{sphericalcompletelemma} and \ref{sphericaldefect}, we have that $\overline{w}$ is henselian defectless and hence algebraically maximal. Since $w=w^+\circ \overline{w}$, we get that $w$ is algebraically maximal. The semi-perfect ring $\Oo_v/(p)$ surjects onto $\Oo_{\overline{v}}$ and therefore $k_w$ is perfect. Since $\Oo_{v}/(p)$ semi-perfect and $vp$ is not minimal in $\Gamma_v$, we get that $\Gamma_v$ is roughly $p$-divisible. The same is true for $\Gamma_w$, being a quotient of $\Gamma_v$. We conclude that $(K,w)$ is roughly tame.\\
(ii) Since $\Gamma_v$ is an elementary extension of $\Gamma_{v_0}$, it is regularly dense. Therefore $\Gamma_w$ is divisible, being a proper quotient of $\Gamma_v$. Using (i), this gives that every finite extension of $(K,w)$ is unramified. The moreover part follows from Fact \ref{unramfact}.
\end{proof}
We show how a defect extension transforms into an unramified one:
\begin{example}
Let $K=\F_p(\!(t)\!)^{1/p^{\infty}}$ and $\alpha$ be a root of the Artin-Schreier polynomial $X^p-X-1/t$. The extension $K(\alpha)/K$ is a classical example of a defect extension (cf. Example 3.12 \cite{kuhlmann2010defect}). 
Write $K_U$ for the corresponding ultrapower and $w$ for the coarsest coarsening of $v_U$ with $wt>0$. By \L o\'s, we have that $K_U(\alpha)=K_U(\beta)$, where $\beta$ is a root of the polynomial $X^p-X-1/\tau$ with $\tau=\ulim t^{1/p^n}$. Now observe that $X^p-X-1/\tau$ reduces to an irreducible Artin-Schreier polynomial in $k_{w}$. It follows that $K_U(\alpha)/K_U$ is unramified with respect to $w$. 
\end{example}

\section{Ax-Kochen/Ershov principles} \label{akesec}
We prove Ax-Kochen/Ershov principles for the class $\mathcal{C}$ of \S \ref{nicedeeplyram}. This will later be used to show that $K^{\flat}$ embeds elementarily in $k_w$. We also deduce some model-theoretic applications in \S \ref{newaxkochenphenomena} and \S \ref{proppresbytilt} which are of independent interest. 
\subsection{Ax-Kochen/Ershov for the class $\mathcal{C}$}

\subsubsection{Relative subcompleteness}
\bt \label{relsubperf}
Let $(K,v)\subseteq (K',v'),(K^*,v^*)$ be henselian valued fields of residue characteristic $p>0$ such that:
\begin{itemize}
\item The rings $\Oo_v/(p), \Oo_{v'}/(p)$ and $\Oo_{v^*}/(p)$ are semi-perfect.
\item  There is $\varpi \in \mathfrak{m}_v$ such that $\Oo_v[\varpi^{-1}]$, $\Oo_{v'}[\varpi^{-1}]$ and $\Oo_{v^*}[\varpi^{-1}]$ are algebraically maximal. 
\item $\Gamma_{v}\preceq_{\exists} \Gamma_{v'}$ in $L_{\text{oag}}$. 
\end{itemize}
Then the following are equivalent: 
\begin{enumerate}[label=(\roman*)]
\item $(K',v')\equiv_{(K,v)} (K^*,v^*)$ in $L_{\text{val}}$. 

\item \label{condition2} $\Oo_{v'}/(\varpi )\equiv_{\Oo_v/(\varpi)} \Oo_{v^*}/(\varpi)$ in $L_{\text{rings}}$ and $\Gamma_{v'} \equiv_{\Gamma_v} \Gamma_{v^*}$ in $L_{\text{oag}}$.
\end{enumerate}
Moreover, if $\Gamma_v$ is regularly dense, the condition $\Gamma_{v}\preceq_{\exists} \Gamma_{v'}$ can be omitted. If $\Gamma_{v'}$ and $\Gamma_{v^*}$ are regularly dense, then the value group condition in \ref{condition2} can be omitted. 
\et 
\begin{proof}
We may assume that $vp$ is not minimal positive in $\Gamma_v$ because otherwise this follows from Ax-Kochen/Ershov for unramified mixed characteristic henselian fields with perfect residue field (cf. Theorem 7.18 \cite{vdd}).\\
$(i)\Rightarrow (ii)$: Clear since both $\Oo_v/(\varpi)$ and $\Gamma_v$ are interpretable structures in $(K,v)$ (and likewise for $K'$).\\
$(ii)\Rightarrow (i)$: 
By Proposition \ref{defectelementary}
and the Keisler-Shelah Theorem (Theorem 6.1.15 \cite{changkeisler}), we may
take non-principal ultrapowers of $(K,v)$, $(K',v')$ and $(K^*,v^*)$ with respect to the same ultrafilter and thus assume that $(K,v)$ is $\aleph_1$-saturated
and that we have an isomorphism of rings 
$\Oo_{v'}/(\varpi)  \cong_{\Oo_v/(\varpi)} \Oo_{v^*}/(\varpi)$
and an isomorphism of ordered abelian groups $ \Gamma_{v'} \cong_{\Gamma_v} \Gamma_{v^*}$.

Let $w$ (resp. $w^*$ and $w'$) be the coarsest coarsening of $v$ (resp. $v^*$ and $v'$) such that $w\varpi >0$ (resp. $w^* \varpi >0$ and $w'\varpi>0$). By Lemma \ref{reducedring}, we have that $\Oo_{\overline{v}}=(\Oo_v/(\varpi ))_{\text{red}}$ (resp. $\Oo_{\overline{v'}}=(\Oo_{v'}/(\varpi))_{\text{red}}$ and $\Oo_{\overline{v^*}}=(\Oo_v/(\varpi ))_{\text{red}}$). By our assumption above, we have that 
$\Oo_{v'}/(\varpi)  \cong_{\Oo_v/(\varpi)} \Oo_{v^*}/(\varpi)$.
Modding out by the nilradicals, we get a ring isomorphism 
$$\Oo_{\overline{v'}}  \cong_{\Oo_{\overline{v}}} \Oo_{\overline{v^*}}$$
By passing to the fraction fields, this induces a valued field isomorphism 
$$\phi:  (k_{w'},\overline{v'}) \cong_{(k_{w}, \overline{v})} (k_{w^*}, \overline{v^*})$$
Note that the extension $k_{w'}/k_w$ is automatically separable since $k_w$ is perfect. Similarly, we have an isomorphism of value groups $ \Gamma_{v'} \cong_{\Gamma_v} \Gamma_{v^*}$. Modding out by the maximal convex subgroups not containing $v\varpi$, this descends to an isomorphism of ordered abelian groups $\Gamma_{w'} \cong_{\Gamma_w} \Gamma_{w^*}$. Moreover, by Lemma \ref{valuegrouplemma}, we have that $\Gamma_w\preceq_{\exists} \Gamma_{w'}$ and hence $\Gamma_{w'}/\Gamma_w$ is torsion-free.  

By Theorem \ref{nonstdtategr}(i), the valued fields $(K,w)$, $(K',w')$ and $(K^*,w^*)$ are roughly tame. We now use the resplendent version of Fact \ref{akerelsub}, where the additional structure on the residue fields is given by the induced valuations $\overline{v},\overline{v'}$ and $\overline{v^*}$. We obtain that $(K',w',\overline{v'})\equiv_{(K,w,\overline{v})} (K^*,w^*,\overline{v^*})$. Applying the Keisler-Shelah theorem once again, 
we may replace $K'$ and $K^*$ with suitable $|K|^+$-saturated elementary extensions 
to obtain an isomorphism
$$ \Phi:(K',w',\overline{v'})  \cong_{(K,w)}(K^*,w^*,\overline{v^*}) $$
\textbf{Claim: }$\Phi:(K',v')\to (K^*,v^*)$ is a valued field isomorphism over $(K,v)$. 
\begin{proof}
Since $\Phi$ is a field isomorphism over $K$, it suffices to show that it is valuation preserving. 
Given $x\in \Oo_{w'}$, we write $\overline{x}$ for its residue with respect to $w'$ (and likewise for $w^*$). For any $x\in \Oo_{v'}$, we have 
$$v'(x)\geq 0 \iff w'(x)\geq 0 \mbox{ and } \overline{v'}(x) \geq 0 \iff $$
$$\iff w^*(\Phi(x)) \geq 0 \mbox{ and } \overline{v^*}(\phi(x))\geq 0 \iff $$ 
$$\iff w^*(\Phi(x)) \geq 0 \mbox{ and } \overline{v^*}(\overline{\Phi(x)})\geq 0\iff  v^*(\Phi(x))\geq 0$$
where the penultimate equivalence uses that $\Phi$ induces $\phi$.
\qedhere $_{\textit{Claim}}$ \end{proof}

The moreover part follows from an inspection of the above proof and some basic facts about divisible ordered abelian groups. If $\Gamma_v$ is regularly dense, then $\Gamma_w$ is divisible and therefore existentially closed. Then $\Gamma_w\preceq_{\exists} \Gamma_{w'}$ follows, so the assumption $\Gamma_v\preceq_{\exists} \Gamma_{v'}$ is not needed. If $\Gamma_{v'}$ and $\Gamma_{v^*}$ are regularly dense, then $\Gamma_{w'}\equiv_{\Gamma_w}\Gamma_{w^*}$ automatically follows by quantifier elimination for the theory of divisible ordered abelian groups, so the assumption $\Gamma_{v'}\equiv_{\Gamma_v}\Gamma_{v^*}$ is not needed. 
\end{proof}

\subsubsection{Elementary substructures}
\bt \label{perfprecversion}
Let $(K,v)\subseteq (K',v')$ be two henselian valued fields of residue characteristic $p>0$ such that $\Oo_v/(p)$ and $\Oo_{v'}/(p)$ are semi-perfect. Suppose there is $\varpi \in \mathfrak{m}_v$ such that both $\Oo_v[\varpi^{-1}]$ and $\Oo_{v'}[\varpi^{-1}]$ are algebraically maximal. Then the following are equivalent: 
\begin{enumerate}[label=(\roman*)]
\item $(K,v)\preceq (K',v')$ in $L_{\text{val}}$.

\item \label{condition2} $\Oo_v/(\varpi )\preceq \Oo_{v'}/(\varpi)$ in $L_{\text{rings}}$ and $\Gamma_v\preceq \Gamma_{v'}$ in $L_{\text{oag}}$.
\end{enumerate}
Moreover, if $\Gamma_v$ and $\Gamma_{v'}$ are regularly dense, then the value group condition in \ref{condition2} can be omitted. 
\et 
\begin{proof}
This is a special case of Theorem \ref{relsubperf} for $K=K'$.
\end{proof}
We now stress the importance of $\Oo_v[\varpi^{-1}]$ and $\Oo_{v'}[\varpi^{-1}]$ being algebraically maximal. Given a field $k$ of characteristic $p>0$ and $a\in k$, write $\wp_a(X)=X^p-X-a$ for the associated Artin-Schreier polynomial.
\bob \label{wrongake}
Let $U$ be a non-principal ultrafilter on $\N$. Equip  $K=\F_p(\!(t)\!)^{1/p^{\infty}}$ and $K'=(\F_p(\!(t)\!)_U)^{1/p^{\infty}}$ with the $t$-adic (resp. non-standard $t$-adic) valuation $v$ (resp. $v'$) and view $(K,v)$ as a valued subfield of $(K',v')$. Then $\Gamma_v$ and $\Gamma_{v'}$ are regularly dense, we have $\Oo_K/(t)\preceq \Oo_{K'}/(t)$ but 
$(K,v)\not \preceq  (K',v')$.
\eob
\begin{proof}
A standard computation shows that $\Oo_{K}/(t)\cong \F_p[t^{1/p^{\infty}}]/(t)$. Viewing $t$ as an element in $\Oo_{K'}$ via the diagonal embedding, we also compute 
$$\Oo_{K'}/(t) \cong \varinjlim_{n\in \N} (\F_p[\![t]\!]_U[t^{1/p^n}]/(t))\cong \varinjlim_{n\in \N} \F_p[t^{1/p^n}]/(t) \cong \F_p[t^{1/p^{\infty}}]/(t)$$
Regarding the value groups, we have that $\Gamma_v=\frac{1}{p^{\infty}}\Z$ and $\Gamma_{v'}=\frac{1}{p^{\infty}}\Z_U$. It is not hard to see that $\Gamma_v$ and $\Gamma_{v'}$ are both regularly dense ordered abelian groups. 
Note that $(K,v)$ together with $t$ is in the class $\mathcal{C}$ introduced in \S \ref{nicedeeplyram}. 

We claim that $(K',v')$ together with $t$ is not in $\mathcal{C}$. 
Let $\ell\neq p$ be prime, set $c= \ulim 1/t^{\ell^n} $ and consider the Artin-Schreier extension $L=K'[X]/(\wp_{c}(X))$. \\
\textbf{Claim: }$\Oo_L[t^{-1}]/\Oo_{K'}[t^{-1}]$ is an immediate extension. 
\begin{proof}
Suppose otherwise. The value group of $\Oo_{K'}[t^{-1}]$ is a quotient of $\Gamma_v$ and is therefore divisible. Moreover, the residue field of $\Oo_{K'}[t^{-1}]$ is perfect. We must then have that $\Oo_L[t^{-1}]/\Oo_{K'}[t^{-1}]$ is an unramified Galois extension of degree $p$. We can therefore write $L\cong K'[X]/(\wp_{a}(X))$, for some $a\in K'$ such that $va \geq -N\cdot vt$. We may even assume that $N=1$, by successively extracting $p$-th roots of $a$, if necessary (cf. \S \ref{poschar}). Since  $\wp_{c}(X)$ and $\wp_{a}(X)$ generate the same extension, we get that
$$c-a \in \wp(K') \iff c^{p^n}-a^{p^n}\in \wp(\F_p(\!(t)\!)_U)\mbox{ for some }n\in \N$$
$$ \iff c-a^{p^n}\in \wp(\F_p(\!(t)\!)_U)\mbox{ for some }n\in \N$$
The last equivalence uses that $c^{p^n}-c\in \wp(\F_p(\!(t)\!)_U)$, which can be seen from 
$$c^{p^n}-c=(c^{p^n}-c^{p^{n-1}})+(c^{p^{n-1}}-c^{p^{n-2}})+...+(c^p-c) $$
and by observing that each term lies in $\wp(\F_p(\!(t)\!)_U)$.

Write $a^{p^n}=\ulim b_m$, where $b_m \in \F_p(\!(t)\!)$ with $vb_m\geq -p^n\cdot vt$. For $m>\log_{\ell} (-p^n vt)$, we see that $v(1/t^{\ell^m}-b_m)=-\ell^m$. Since $p\nmid -\ell^m$, we get that $1/t^{\ell^m}-b_m \notin \wp(\F_p(\!(t)\!))$. It follows that  
$$\{m \in \N:1/t^{\ell^m}-b_m \in \wp(\F_p(\!(t)\!)) \} \not \in U$$
which is contrary to the fact that $c-a^{p^n}\in \wp(\F_p(\!(t)\!)_U)$. We conclude that no such $a$ exists and that $\Oo_L[t^{-1}]/\Oo_{K'}[t^{-1}]$ is an immediate extension. 
\qedhere $_{\textit{Claim}}$ \end{proof}
We thus have that $(K',v')$ together with $t$ is not in $\mathcal{C}$. By Proposition \ref{defectelementary}, $\mathcal{C}$ is an elementary class in $L_{\text{val}}(\varpi)$. We conclude that $(K,v)\not \equiv_{(\F_p(t),v_t)} (K',v')$ and in particular $(K,v)\not \preceq (K',v')$.
\end{proof}
\subsubsection{Elementary equivalence (positive characteristic)}

\bt \label{perfakep}
Let $(K,v), (K',v')$ be two perfect henselian valued fields over $(\F_p(t),v_t)$ such that both $\Oo_v[t^{-1}]$ and $\Oo_{v'}[t^{-1}]$ are algebraically maximal. Suppose that
$\Oo_v/(t)\equiv \Oo_{v'}/(t)$ in $L_{\text{rings}}$ and $(\Gamma_v,vt)\equiv  (\Gamma_{v'},v't)$ in $L_{\text{oag}}(vt)$. Then $(K,v)\equiv (K',v')$  in $L_{\text{val}}$.
\et 
\begin{proof}
Similar to Theorem \ref{perfprecversion}, ultimately using Fact \ref{akekuhl}. 
\end{proof}
Recall that deeply ramified fields of positive characteristic need not be perfect, only their completions are. However, the above result does not hold in the deeply ramified setting as the following example shows:
\begin{example} \label{positiverem}
Let $K$ be the completed perfect hull of $ \F_p(\!(t)\!)$, which can also be obtained as the $t$-adic completion of a non-perfect henselian field $K'$. For example, let $z\in \F_p(\!(t)\!)$ be transcendental over $\F_p(t)$ and $K'$ be the henselization of $\F_p(t^{1/p^{\infty}})(z)$ inside the completed perfect hull of $ \F_p(\!(t)\!)$.
Then the side conditions of Theorem \ref{perfakep} are met but $K\not \equiv K'$ since $K$ is perfect while $K'$ is not.
\end{example}
\subsubsection{Existential closedness}

\bt \label{existentialclosedness}
Let $(K,v)$ be a henselian valued field of residue characteristic $p>0$ such that $\Oo_v/(p)$ is semi-perfect and suppose that there is $\varpi \in \mathfrak{m}_v$ such that $\Oo_v[\varpi^{-1}]$ is algebraically maximal. Given $(K,v)\subseteq (K',v')$, the following are equivalent: 
\begin{enumerate}[label=(\roman*)]
\item $(K,v)\preceq_{\exists} (K',v')$ in $L_{\text{val}}$.

\item \label{condition2} $\Oo_v/(\varpi )\preceq_{\exists} \Oo_{v'}/(\varpi)$ in $L_{\text{rings}}$ and $\Gamma_v\preceq_{\exists} \Gamma_{v'}$ in $L_{\text{oag}}$.
\end{enumerate}
Moreover, if $\Gamma_v$ is regularly dense, the value group condition in \ref{condition2} can be omitted.
\et 
\begin{proof}
As in Theorem \ref{perfprecversion}, ultimately using Fact \ref{precversion}.
\end{proof}

\subsection{New Ax-Kochen/Ershov phenomena} \label{newaxkochenphenomena}
Given a field $k$ and an ordered abelian group $G$, we write $k(\!(t^G)\!)$ for the corresponding Hahn series field. 

To motivate our discussion, consider the following three open problems: 
\begin{problem} 
\begin{enumerate}[label=(\roman*)]
\item  Is $(\F_p(t)^h,v_t)\preceq (\F_p(\!(t)\!),v_t)$? 
\item  Is $(\F_p(\!(t)\!),v_t) \preceq (\F_p(\!(t)\!)(\!(z^{\Q})\!), v)$, where $v=v_z\circ v_t$?  
\item Is $(\varinjlim_{n\in \N} \F_{p^n}(\!(t)\!),v_t) \preceq (\overline{\F}_p(\!(t)\!),v_t)$?
\end{enumerate}

\end{problem}
We now prove their \quotes{perfected} variants (over a general perfect base field):
\bc \label{newakephenomena}
Let $k$ be a perfect field with $\text{char}(k)=p>0$. Then:
\begin{enumerate}[label=(\roman*)]
\item $(k(t^{1/p^{\infty}})^h,v_t) \preceq (k(\!(t)\!)^{1/p^{\infty}},v_t) \preceq (\widehat{ k(\!(t)\!)^{1/p^{\infty}} },v_t)$.
\item $(k(\!(t)\!)^{1/p^{\infty}}, v_t) \preceq (k(\!(t)\!)^{1/p^{\infty}}(\!(z^\Q)\!),v)$, where $v=v_z\circ v_t$.
\item $ (\varinjlim k_i(\!(t)\!)^{1/p^{\infty}},v_t) \preceq (\overline{k}(\!(t)\!)^{1/p^{\infty}},v_t)$, where $k_i$ runs over finite extensions of $k$.
\end{enumerate}
\ec 
\begin{proof}
Note that $\Oo_v[t^{-1}]$ is algebraically maximal and that the value group is regularly dense for each of the above fields. We now use Theorem \ref{perfprecversion}: \\
$(i),(ii)$: For all five fields, the residue ring modulo $t$ is equal to $k[t^{1/p^{\infty}}]/(t)$.\\
$(iii)$ In both cases, the residue ring modulo $t$ is equal to $\overline{k}[t^{1/p^{\infty}}]/(t)$.
\end{proof}

\subsection{Properties preserved via tilting} \label{proppresbytilt}
We show that tilting and untilting preserve several model-theoretic relations:
\bc \label{tiltingpreserves}
Suppose $(K,v)$ and $(K',v')$ are two perfectoid fields. Then:
\begin{enumerate}[label=(\roman*)]

\item If $(K_0,v_0)$ is a common perfectoid subfield of $(K,v)$ and $(K',v')$, then 
$$(K,v) \equiv_{(K_0,v_0)} (K',v') \mbox{ in }L_{\text{val}}\iff  (K^{\flat},v^{\flat}) \equiv_{(K_0^{\flat},v_0^{\flat})} (K'^{\flat},v'^{\flat}) \mbox{ in }L_{\text{val}} $$

\item $(K,v)\preceq (K',v')\mbox{  in }L_{\text{val}}\iff (K^{\flat},v^{\flat})\preceq (K'^{\flat},v'^{\flat})\mbox{ in }L_{\text{val}}$.

\item  $(K,v)\equiv (K',v') \mbox{ in }L_{\text{val}} \implies (K^{\flat},v^{\flat})\equiv (K'^{\flat},v'^{\flat}) \mbox{ in }L_{\text{val}}$.

\item  $(K,v)\preceq_{\exists} (K',v')\mbox{  in }L_{\text{val}}\iff (K^{\flat},v^{\flat})\preceq_{\exists} (K'^{\flat},v'^{\flat})\mbox{ in }L_{\text{val}}$.

\end{enumerate}

\ec
\begin{proof}
$(i)$ 
Let $\varpi \in \Oo_{K_0}$ be a pseudo-uniformizer admitting a compatible system of $p$-power roots and set $\varpi^{\flat}=(\varpi,\varpi^{1/p},...) \in \Oo_{K_0^{\flat}}$ (see Remark \ref{scholzremark}). By Lemma \ref{lemscholz1}(iii), we have $\Oo_{K_0}/(\varpi)\cong \Oo_{K_0^{\flat}}/(\varpi^{\flat})$ and similarly for $K$ and $K'$. By Theorem \ref{relsubperf}, we have
$$(K,v)\equiv_{(K_0,v_0)} (K',v') \stackrel{\ref{relsubperf}} \iff \Oo_K/(\varpi )\equiv_{\Oo_{K_0}/(\varpi)} \Oo_{K'}/(\varpi) \iff $$
$$\iff\Oo_{K^{\flat}}/(\varpi^{\flat})\equiv_{\Oo_{K_0^{\flat}}/(\varpi^{\flat})} \Oo_{K'^{\flat}}/(\varpi^{\flat}) \stackrel{\ref{relsubperf}}  \iff (K^{\flat},v^{\flat})\equiv_{(K_0^{\flat},v_0^{\flat})} (K'^{\flat},v'^{\flat})$$
$(ii)$ This is a special case of (i) for $K_0=K$.\\
$(iii)$ Choose $\varpi\in \Oo_K$ (resp. $\varpi'\in \Oo_{K'}$) such that $v\varpi=vp$ (resp. $v'\varpi'=v'p$). Then
$$\Oo_{K^{\flat}}/(\varpi^{\flat})\cong \Oo_K/(p)\equiv \Oo_{K'}/(p) \cong \Oo_{K'^{\flat}}/(\varpi'^{\flat})$$
where the equivalence in the middle follows from our assumption. We conclude from Theorem \ref{perfakep}.  \\
$(iv)$ Similar to $(iii)$, using Theorem \ref{existentialclosedness}.
\end{proof}
The converse of Corollary \ref{tiltingpreserves}(iii) fails in general because of the ambiguity of untilting:
\begin{example} \label{manyuntiltsrem}
 Take $K=\widehat{\Q_p(p^{1/p^{\infty}})}$ and $K'=\widehat{\Q_p(\zeta_{p^{\infty}})}$ and recall that $K^{\flat}\cong K'^{\flat}$ from Example \ref{perfex}. It is easy to check that $K\not \equiv K'$, e.g., by observing that $p^{1/p}\notin K'$.
\end{example}
Recall that untilting becomes unambiguous by specifying some $\xi \in W(\Oo_{K^{\flat}})$ (cf. \S \ref{methodperf}). We will use this to obtain a converse to Corollary \ref{tiltingpreserves}(iii). First, recall the following result by van den Dries:
\bt [van den Dries]{\label{thO/p^nO}} 
Let $(K,v),(K',v')$ be two henselian valued fields of mixed characteristic. Then $(K,v)\equiv (K',v')$ in $L_{\text{val}}$ if and only if $\mathcal{O}_v/p^n \Oo_v\equiv \Oo_{v'}/p^n\Oo_{v'}$ in $L_{\text{rings}}$ for all $n\in \N$ and $(\Gamma_v,vp)\equiv (\Gamma_{v'},v'p)$ in $L_{oag}$ together with a constant for $vp$.
\et 
\begin{proof}
See the proof of Theorem 2.1.5 \cite{KK1}.
\end{proof}
For the correspondence between distinguished elements (up to units) $\xi \in W(\Oo_{K^{\flat}})$ and untilts of $K^{\flat}$, see \S 3.5.5 \cite{KK1}.
\bp \label{fixtypewitt}
Say $(K,v)$ and $(K',v')$ be two perfectoid fields. Let $\xi_K=(\xi_0,\xi_1,...)$ (resp. $\xi_{K'}=(\xi'_0,\xi'_1,...)$) be a distinguished element in $W(\Oo_{K^{\flat}})$ (resp. $W(\Oo_{K'^{\flat}})$) associated to $K$ (resp. $K'$) such that 
$$\text{tp}_{(K^{\flat},v^{\flat})}(\xi_0,\xi_1,...)=\text{tp}_{(K'^{\flat},v'^{\flat})}(\xi_0',\xi_1',...)$$ 
Then $(K,v)\equiv (K',v')$ in $L_{\text{val}}$. 
\ep 
\begin{proof}
Note that our hypothesis on the types in particular implies that $(K^{\flat},v^{\flat})\equiv (K'^{\flat},v'^{\flat})$ in $L_{\text{val}}$. We start with the following:\\
\textbf{Claim}: We have $\Oo_K/(p^n)\equiv \Oo_{K'}/(p^n)$,  for all $n\in \N$.
\begin{proof}
Fix $n\in \N$. We have $\Oo_K/(p^n)\cong W_n(\Oo_{K^{\flat}})/(\overline{\xi_K})$ and similarly $\Oo_{K'}/(p^n)\cong W_n(\Oo_{K'^{\flat}})/(\overline{\xi_{K'}})$, where $\overline{\xi}_K$ is the truncated Witt vector $(\xi_0,\xi_1,...,\xi_{n-1})$  and $\overline{\xi}_{K'}$ is the truncated Witt vector $(\xi_0',\xi_1',...,\xi'_{n-1})$. By Lemma 3.4.5 \cite{KK1}, the ring $W_n(\Oo_{K^{\flat}})/(\overline{\xi_K})$ (resp. $W_n(\Oo_{K'^{\flat}})/(\overline{\xi_{K'}})$) is interpretable in $K^{\flat}$ (resp. $K'$) with parameter set $\{\xi_0,\xi_1,...,\xi_{n-1}\}$ (resp. $\{\xi_0',\xi_1',...,\xi_{n-1}'\}$). Moreover, the interpretation is uniform, i.e., it does not depend on $K$ or $\overline{\xi}_K$. Since $(K^{\flat},v^{\flat})\equiv (K'^{\flat},v'^{\flat})$ and $\text{tp}_{(K^{\flat},v^{\flat})}(\xi_0,\xi_1,...,\xi_{n-1})=\text{tp}_{(K'^{\flat},v'^{\flat})}(\xi_0',\xi_1',...,\xi_{n-1}')$, we get that  
$$W_n(\Oo_{K^{\flat}})/(\overline{\xi_K})\equiv W_n(\Oo_{K'^{\flat}})/(\overline{\xi_{K'}})$$
We conclude that $\Oo_K/(p^n)\equiv \Oo_{K'}/(p^n)$.
\qedhere $_{\textit{Claim}}$ \end{proof}
Observe that 
$$(\Gamma_v,vp)\cong (\Gamma_{v^{\flat}},v^{\flat}\xi_0)\equiv (\Gamma_{v^{\flat}},v^{\flat}\xi_0')\cong (\Gamma_{v'},v'p)$$ 
where the elementary equivalence in the middle uses that $\text{tp}_{(K^{\flat},v^{\flat})}(\xi_0)=\text{tp}_{(K'^{\flat},v'^{\flat})}(\xi_0')$. We conclude that $(K,v)\equiv (K',v')$ in $L_{\text{val}}$ from van den Dries' Theorem \ref{thO/p^nO}.
\end{proof}
We also obtain a converse to the decidability transfer of \cite{KK1}:
\bc \label{from0top}
Let $(K,v)$ be a perfectoid field with tilt $(K^{\flat},v^{\flat})$. Then the $L_{\text{val}}$-theory of $(K^{\flat},v^{\flat})$ is decidable relative to the $L_{\text{val}}$-theory of $(K,v)$.
\ec
\begin{proof}
Let $t\in \Oo_{K^{\flat}}$ be such that $\Oo_K/(p)\cong \Oo_{K^{\flat}}/(t)$ by Lemma \ref{lemscholz1}(iii). By Proposition \ref{defectelementary}, the class of perfect henselian valued fields $(F,u)$ of characteristic $p>0$, such that $\Oo_F[t^{-1}]$ is algebraically maximal is elementary in $L_{\text{val}}(t)$. Moreover, it is recursively axiomatized. We further add a recursive set of axioms to ensure that $\Oo_F/(t)\equiv \Oo_{K^{\flat}}/(t)$. This is possible because $\Oo_{K^{\flat}}/(t)\cong \Oo_K/(p)$ and the latter is interpretable in $(K,v)$. We call $T$ the resulting theory. By Theorem \ref{perfakep}, for any sentence $\phi$ in $L_{\text{val}}$, we have that 
$$T\models \phi \iff (K^{\flat},v^{\flat})\models \phi$$ 
Finally, a brute force enumeration of all proofs from $T$ gives an algorithm for deciding whether $T \models \phi$ or $T \models \lnot \phi$ for any sentence $\phi$ in $L_{\text{val}}$. 
\end{proof}

\section{Proof of the main theorem} \label{modeltheorfontwintsec}
In \S \ref{elementarysubfield}, we exhibit $K^{\flat}$ as an elementary valued subfield of $k_w$. In \S \ref{moreoverpart}, we 
verify that $ K_U-\text{f\'et} \simeq k_w-\text{f\'et}$ restricts to $K-\text{f\'et} \simeq K^{\flat}-\text{f\'et}$ as expected. We then conclude the proof of our main theorem and make a few remarks.
\subsection{$K^{\flat}$ as an elementary subfield of $k_w$} \label{elementarysubfield}
Let $K$ be a perfectoid field. Recall from \S \ref{perfsec} that the underlying set of the tilt is
$$K^{\flat}=\{ (x_n)_{n\in \omega} \in K^{\omega}: x_{n+1}^p=x_n\}$$
\begin{definition}
Let $U$ be an ultrafilter on $\N$ and $ K_U$ be the corresponding ultrapower. We define the \textit{natural} map
 $$\natural: K^{\flat} \to K_U: (x_n)_{n\in \omega} \mapsto \ulim x_n$$
Write $x^{\natural}$ for the image of $x \in K^{\flat}$ under the natural map $\natural$.
\end{definition}

\bl \label{canonicalift}
Let $(K,v)$ be a perfectoid field of residue characteristic $p$ and $\varpi \in \mathfrak{m}\backslash \{0\}$. Let $U$ be a non-principal ultrafilter on $\N$ and $( K_U,v_U)$ be the corresponding ultrapower. Let also $w$ be the coarsest coarsening of $v_U$ such that $w \varpi>0$. Then: 
\begin{enumerate}[label=(\roman*)]
\item We have $\natural (K^{\flat \times}) \subseteq \Oo_w^{\times}$, $\natural (\Oo_{K^{\flat}})\subseteq \Oo_{ v_U}$ and $\natural (\mathfrak{m}_{K^{\flat}})= \natural (\Oo_{K^{\flat}}) \cap \mathfrak{m}_{ v_U}$.
\item The map $\natural$ is a multiplicative morphism which is additive modulo $p$, i.e., $(x+y)^{\natural} \equiv x^{\natural}+y^{\natural} \mod p\Oo_{ v_U}$ for any $x,y \in \Oo_{K^{\flat}}$.

\item The map 
$$\iota: (K^{\flat},v^{\flat})\to (k_w,\overline{v}): x\mapsto x^{\natural}  \mod \mathfrak{m}_w $$ 
is a well-defined valued field embedding.
\end{enumerate}
\el 
\begin{proof}
(i) For any $(x,x^{1/p},...)\in K^{\flat \times }$ and $m\geq 1$, we have that
$$v(x^{1/p^n}) \in (-1/m\cdot v\varpi, 1/m\cdot v\varpi)$$  
for all sufficiently large $n$. Since $U$ is non-principal, it contains every cofinite subset of $\N$. It follows that $x^{\natural}\in \Oo_{w}^{\times}$. By definition of $v^{\flat}$, we also have that  $(x,x^{1/p},...) \in \Oo_{K^{\flat}}$ (resp. $x\in \mathfrak{m}_{K^{\flat}}$) if and only if $x^{1/p^n} \in \Oo_{K}$ (resp. $x^{1/p^n} \in \mathfrak{m}_{K}$) for each $n\in \N$. This implies that $\natural (\Oo_{K^{\flat}})\subseteq \Oo_{ v_U}$ and $\natural (\mathfrak{m}_{K^{\flat}})= \natural (\Oo_{K^{\flat}}) \cap \mathfrak{m}_{ v_U}$. \\  
(ii) By Lemma \ref{lemscholz1}(iv), the map
$$\sharp: K^{\flat}\to K:(x_n)_{n\in \omega}\mapsto x_0$$
is a multiplicative morphism whose restriction $\sharp:\Oo_{K^{\flat}}\to \Oo_K$ is additive modulo $p$. For each $n$, the same is also true for  
$$\sharp_n: K^{\flat}\to K: (x_n)_{n\in \omega} \mapsto x_n$$
It follows that $\natural$ has the desired properties by \L o\'s' theorem.\\
(iii) Since $\natural (K^{\flat \times}) \subseteq \Oo_w^{\times}$, we have a well-defined map 
$$\iota: K^{\flat}\to k_w:x\mapsto x^{\natural}  \mod \mathfrak{m}_w$$ 
By (ii), the map $\natural$ is a multiplicative morphism which is additive modulo $p$. In particular, it is additive modulo $\mathfrak{m}_w$ and therefore $\iota$ is a field embedding. Since $\natural (\Oo_{K^{\flat}})\subseteq \Oo_{ v_U}$ and $\Oo_{\overline{v}}=\Oo_{v_U}\mod \mathfrak{m}_w$, we get that $\iota(\Oo_{K^{\flat}})\subseteq \Oo_{\overline{v}}$. Similarly, since $\natural (\mathfrak{m}_{K^{\flat}})= \natural (\Oo_{K^{\flat}}) \cap \mathfrak{m}_{ v_U}$, we get that $\iota(\mathfrak{m}_{K^{\flat}})=\iota(\Oo_{K^{\flat}}) \cap  \mathfrak{m}_{\overline{v}}$. 
We conclude that $\iota: (K^{\flat},v^{\flat})\to (k_w,\overline{v})$ is a valued field embedding.
\end{proof}
%
It is now convenient to work with a pseudo-uniformizer $\varpi$ which admits a compatible system of $p$-power roots (Remark 3.5 \cite{Scholze}). We write $\varpi^{\flat}=(\varpi,\varpi^{1/p},...)\in K^{\flat}$ and $ \pi=\ulim \varpi^{1/p^n} \in K_U$. 
\bt  \label{keylemmainfthick}
Let $(K,v)$ be a perfectoid field of residue characteristic $p$ and $\varpi $ be as above. Let $U$ be a non-principal ultrafilter on $\N$ and $( K_U,v_U)$ be the corresponding ultrapower. Let $w$ be the coarsest coarsening of $v_U$ such that $w \varpi>0$.
Then: 
\begin{enumerate}[label=(\roman*)]
\item  The map $\natural$ induces an elementary embedding of rings 
$$\Oo_{K^{\flat}}/(\varpi^{\flat})\to \Oo_{ v_U}/(\pi):x+(\varpi^{\flat}) \mapsto x^{\natural} + (\pi)$$

\item The map $\iota$ induces an elementary embedding of rings
$$  \Oo_{K^{\flat}}/(\varpi^{\flat})\to \Oo_{\overline{v}}/(\overline{\pi}): x+(\varpi^{\flat})\mapsto \iota(x)+(\overline{\pi})$$ 
where $\overline{\pi}$ is equal to the image of $\pi$ modulo $\mathfrak{m}_w$. 
\item The map $\iota:(K^{\flat},v^{\flat})\to (k_w,\overline{v})$ is an elementary embedding of valued fields.
\end{enumerate}
\et  
\begin{proof}
(i) For each $i\in \N$, we have the iterated Frobenius map 
$$\Phi_{i}:\Oo_K/(\varpi)\to \Oo_K/(\varpi):x +(\varpi)\mapsto x^{p^i}+(\varpi) $$
Note that $\Phi_i$ is surjective and $\text{Ker}(\Phi_i)=(\varpi^{1/p^i})$. Passing to ultralimits, we get a non-standard Frobenius map
$$\Phi_{\infty}:\Oo_{v_U}/(\varpi)\to \Oo_{v_U}/(\varpi):\ulim x_i +(\varpi)\mapsto \ulim x_i^{p^i}+(\varpi)$$ 
Note that $\Phi_{\infty}$ is surjective and $\text{Ker}(\Phi_{\infty})=(\pi)$. Thus, $\Phi_{\infty}$ induces an isomorphism 
$\overline{\Phi}_{\infty}:\Oo_{v_U}/(\pi) \stackrel{\cong}\rightarrow \Oo_{v_U}/(\varpi)$. We now exhibit the map $\Oo_{K^{\flat}}/(\varpi^{\flat})\to \Oo_{ v_U}/(\pi)$ as a composition of the maps below
$$\Oo_{K^{\flat}}/(\varpi^{\flat})\stackrel{\cong}\longrightarrow \Oo_K/(\varpi)\stackrel{\delta} \longrightarrow \Oo_{v_U}/(\varpi)
\xrightarrow{\cong}  \Oo_{ v_U}/(\pi)$$
Here $\delta$ is the diagonal embedding 
$$\delta:\Oo_K/(\varpi) \to  \Oo_{v_U}/(\varpi): x +(\varpi) \mapsto \ulim x + (\varpi)$$  
which is elementary by \L o\'s. The rightmost isomorphism is the inverse of $\overline{\Phi}_{\infty}$. The composite map is equal to the desired map and is elementary.\\
(ii) By (i), we know that the map 
$$\Oo_{K^{\flat}}/(\varpi^{\flat})\to \Oo_{ v_U}/(\pi):x+(\varpi^{\flat}) \mapsto x^{\natural} + (\pi)$$ 
is an elementary embedding of rings.  We also have a natural isomorphism 
$$\Oo_{ v_U}/(\pi) \to \Oo_{\overline{v}}/(\overline{\pi}):x+(\pi)\mapsto \overline{x} + (\overline{\pi})$$ 
induced from the reduction map of $w$.
Note that $\overline{\pi}= \iota(\varpi^{\flat})$, so that the map in question is the composite map 
$$\Oo_{K^{\flat}}/(\varpi^{\flat})\stackrel{\preceq } \longhookrightarrow \Oo_{ v_U}/(\pi)\stackrel{\cong} \longrightarrow \Oo_{\overline{v}}/(\overline{\pi})$$ 
which is elementary, being the composition of elementary maps.\\
(iii) We choose the elements $\varpi^{\flat}\in K^{\flat}$ and $\overline{\pi}\in k_w$ as distinguished pseudo-uniformizers. We identify $(K^{\flat},v^{\flat})$ with a valued subfield of $(k_w,\overline{v})$ via $\iota$, thereby identifying $\varpi^{\flat}$ with $\overline{\pi}$. Next, we verify the hypotheses of Theorem \ref{perfprecversion} for $K^{\flat}$ and $k_w$. For $K^{\flat}$ use Remark \ref{comparisonwithotherclasses}(i) and for $k_w$ use Lemma \ref{trickremark1}(iii). The value groups $\Gamma_{v^{\flat}}$ and $\Gamma_{\overline{v}}$ are both regularly dense. For the latter, note that $\Gamma_{\overline{v}}$ is a convex subgroup of the regularly dense group $\Gamma_{v_U}$. By part (i), we also have that  $\Oo_{K^{\flat}}/(\varpi^{\flat})\preceq \Oo_{\overline{v}}/(\overline{\pi})$. By Theorem \ref{perfprecversion}, we conclude that $(K^{\flat},v^{\flat})\preceq (k_w,\overline{v})$.
\end{proof}
\subsection{Specializing to Fontaine-Wintenberger} \label{moreoverpart}
Recall the definition of the discriminant from \S \ref{unramext}.
\bl \label{p-throots}
Let $(K,v)$ be a henselian valued field of residue characteristic $p>0$ and $\varpi \in \Oo_K$ be such that $0<v\varpi\leq vp$. Let $f(X)=X^n+a_{n-1}X^{n-1}+...+a_0\in \Oo_v[X]$ be an irreducible polynomial with $0\leq v(\text{disc}(f) )<v\varpi$ and $g(X)=X^n+a_{n-1}'X^{n-1}+...+a_0'\in \Oo_v[X]$ be such that $a_i'^p\equiv a_i \mod (\varpi)$. Then $v(\text{disc}(g))=\frac{1}{p}v(\text{disc}(f))$ and $K[X]/(f(X))\cong K[X]/(g(X))$.
\el 
\begin{proof}
Write $L=K(\beta)$, with $\beta$ a root of $g(X)$. Recall that the discriminant can be expressed as a polynomial in the coefficients. Since $0< v\varpi \leq vp$ and $a_i'^p\equiv a_i \mod (\varpi)$, we deduce that $\text{disc}(g)^p \equiv \text{disc}(f) \mod (\varpi)$. Since $0\leq v(\text{disc}(f))<v\varpi$, it follows that $v(\text{disc}(g))=\frac{1}{p}v(\text{disc}(f))$. Since $a_i'^p\equiv a_i \mod (\varpi)$ and $0<v\varpi\leq vp$, we compute that
$$0=g(\beta)^p\equiv f(\beta^p) \mod (\varpi)\Rightarrow f(\beta^p)\equiv 0 \mod (\varpi)$$
Since $v(f(\beta^p))>v(\text{disc}(f))$, we deduce the existence of $\alpha \in L$ with $f(\alpha)=0$ by the Hensel-Rychlik Lemma (cf. Lemma 12 \cite{AK12}). The conclusion follows.
\end{proof}
We henceforth identify $K^{\flat}$ with its image in $k_w$ via the embedding $\iota$ of Lemma \ref{canonicalift}(iii). Given $f(X)\in K^{\flat}[X]$, write $f^{\natural}(X)\in K_U[X]$ for the polynomial whose coefficients are obtained by applying $\natural$ to the coefficients of $f(X)$.
\bc \label{idenitfyingGK}
Let $(K,v)$ be a perfectoid field and $U$ be a non-principal ultrafilter on $\N$. Then:
\begin{enumerate}[label=(\roman*)]
\item Let $f(X)\in K^{\flat}[X]$ be an irreducible monic polynomial. Then $f^{\natural}(X)$ generates a finite extension of $K_U$ coming from $K$.
\item Conversely, let $L/K$ be a finite extension. Then $L_U/K_U$ induces a finite extension of $k_w$ coming from $K^{\flat}$.
\end{enumerate}
\ec  
\begin{proof}
(i) Once again, recall that the discriminant of a polynomial is itself an integer polynomial in the coefficients, in a way that does not depend on the field. By Lemma \ref{canonicalift}(ii), we get that $\text{disc}(f^{\natural})\equiv \text{disc}(f)^{\natural} \mod (p)$. By Lemma \ref{canonicalift}(i), we have $\text{disc}(f)^{\natural}\in \Oo_w^{\times}$ and thus $\text{disc}(f^{\natural})\in \Oo_w^{\times}$. Write $f^{\natural}(X)=\ulim f_i(X)$, where $f_i(X)\in K[X]$. By \L o\'s, we get that $0\leq \text{disc}(f_i)<vp$ for almost all $i$. By Lemma \ref{p-throots}, we get that the extensions of $K$ generated by the $f_i$'s eventually stabilize to a fixed extension $L$ of $K$. Therefore, the extension of $K_U$ generated by $f^{\natural}(X)$ comes from $L$.
\\
(ii) By Theorem \ref{nonstdtategr}(ii), we have that $L_U/K_U$ is unramified with respect to $\Oo_w$. By Fact \ref{unramfact}(iv), we get that $L_U$ is generated by some monic irreducible polynomial $F(X)\in \Oo_{v_U}[X]$ with $\text{disc}(F) \in \Oo_{w}^{\times}$. In particular, we have that $0\leq v_U(\text{disc}(F))<v_U\varpi$. By \L o\'s, the extension $L$ can be generated by some monic irreducible polynomial $f(X)\in \Oo_K[X]$ with $0\leq v(\text{disc}(f))<v\varpi$. Since $\sharp:\Oo_{K^{\flat}}\to \Oo_K$ is surjective modulo $p$, we can find a monic $g(X)\in \Oo_{K^{\flat}}[X]$ with $g^{\sharp}(X)\equiv f(X) \mod p \Oo_K$. Note that $L$ is also generated by $g^{\sharp}(X)$ by the Hensel-Rychlik Lemma. By Lemma \ref{p-throots}, it is also generated by $g_n^{\sharp}(X)$, where the coefficients of $g_n(X)\in K^{\flat}[X]$ are $p^n$-th roots of the coefficients of $g(X)$. By \L o\'s, the extension $L_U$ can be generated by $g^{\natural}(X)$ and the conclusion follows.
\end{proof}
This recovers the natural equivalence between the categories of finite extensions of $K$ and $K^{\flat}$, as prescribed by Scholze on pg. 3 \cite{ScholzeSurvey}. We finally put all the pieces together: 
\bt \label{modeltheoreticFont}
Let $(K,v)$ be a perfectoid field and $\varpi \in \mathfrak{m}\backslash \{0\}$. Let $U$ be a non-principal ultrafilter on $\N$ and $(K_U,v_U)$ be the corresponding ultrapower. Let $w$ be the coarsest coarsening of $v_U$ such that $w\varpi>0$. Then: 
\begin{enumerate}[label=(\Roman*)]
\item Every finite extension of $(K_U,w)$ is unramified.

\item The tilt $(K^{\flat},v^{\flat})$ embeds elementarily into $(k_w,\overline{v})$. 

\end{enumerate}
Moreover, the equivalence $ K_U-\text{f\'et} \simeq k_w-\text{f\'et}$ restricts to $K-\text{f\'et} \simeq K^{\flat}-\text{f\'et}$.
\et 
\begin{proof}
Part I is a special case of Theorem \ref{nonstdtategr}(ii). Part II is Theorem \ref{keylemmainfthick}(iii).  For the moreover part, recall that \'etale algebras over a field are finite products of finite separable extensions. It then suffices to check the analogous statement for finite (separable) field extensions, which follows from Corollary \ref{idenitfyingGK}.
\end{proof}
We now make a few comments on Question \ref{questionformal}. Regarding the extent to which the theory of $K^{\flat}$ determines the theory of $K$, we note the following:
\begin{rem}\label{tameakerem} 
\begin{enumerate}[label=(\roman*)] 
\item By Corollary \ref{tiltingpreserves}, given any perfectoid base field $K_0\subseteq K$, the theory of $K^{\flat}$ with parameters from $K_0^{\flat}$ fully determines the theory of $K$ with parameters from $K_0$. As an alternative, we saw in Proposition \ref{fixtypewitt} that the theory of $K^{\flat}$ determines the theory of $K$ once we also specify the type of the Witt vector $\xi_K \in W(\Oo_{K^{\flat}})$. 
\item Without the above stipulations, tilting will typically involve some loss of first-order information (cf. Example \ref{manyuntiltsrem}). This is related to the failure of an Ax-Kochen/Ershov principle for tame fields of mixed characteristic down to residue field and value group, something that was first observed by Anscombe-Kuhlmann \cite{KuhlAns}. To provide a counterexample in our setting, just take two non-elementary equivalent perfectoid fields with the same tilt. Their ultrapowers are still non-elementary equivalent, yet they are both equipped with tame valuations such that the value groups are divisible and the residue fields are both elementary equivalent to the tilt.
\end{enumerate}
\end{rem}
With the aid of the diagram below
\[
  \begin{tikzcd}[column sep=4.5em, row sep=2.5em]
    K_U \arrow[r, "w"]  & k_w \arrow[r, "\overline{v}"]  & k_{v_U} \\
    K \arrow[r, dashed, bend left=10, "\flat"] \arrow[u, "\preceq", no head]   & K^{\flat}\arrow[r, "v^{\flat}"] \arrow[u, "\preceq", no head]   & k_{v} \arrow[u, "\preceq", no head] 
     \end{tikzcd}
\]
 several elementary properties can be readily transferred between $K$ and $K^{\flat}$:
\begin{rem} \label{transferproperties}

Properties (i)-(iii) are well-known, but (iv) appears to be new: 
\begin{enumerate}[label=(\roman*)]
    \item Let ACF denote the theory of algebraically closed fields. We have that 
     $$K\models \text{ACF}\iff K_U\models \text{ACF} \iff k_w\models \text{ACF} \iff K^{\flat}\models \text{ACF}$$ 
     All equivalences are formal except the one in the middle which uses that every finite extension of $(K_U,w)$ is unramified. 
     
     \item More generally, one readily sees from (I) and (II) that $G_K\equiv G_{K^{\flat}}$ in the language of profinite groups of Cherlin-van den Dries-Macintyre \cite{CDM}. Indeed, since in general $k\equiv l$ implies that $G_k\equiv G_l$, we get
     $$G_K\equiv G_{K_U}\cong G_{k_w}\equiv G_{K^{\flat}}$$ 
In fact, this elementary version of Fontaine-Wintenberger admits a very short proof (cf. \S 4.2 \cite{KartasThesis}). However, we really need the moreover part to get the isomorphism $G_K\cong G_{K^{\flat}}$.
     \item We have that 
       $$(K,v)\mbox{ is defectless} \iff (K_U,v_U)\mbox{ is defectless} \iff $$ 
       $$(k_w,\overline{v}) \mbox{ is defectless} 
       \iff (K^{\flat},v^{\flat})\mbox{ is defectless}$$
      All equivalences are formal except the middle one which uses that the defect behaves multiplicatively with respect to composition of places. 
    \item Given $i\in \N$, a field $k$ is called $C_i$ if every $k$-hypersurface of degree $d$ in $d^i$-dimensional projective space has a $k$-rational point. We have that
    $$K \mbox{ is }C_i \iff K_U\mbox{ is }C_i \implies k_w \mbox{ is }C_i
       \iff K^{\flat} \mbox{ is }C_i$$
Again, all implications are formal except the middle one. It is not hard to check that the residue field of any $C_i$ valued field is itself $C_i$. 
\end{enumerate}
\end{rem}
We do not know whether the converse to (iv) holds (even for $i=1$):
\begin{problem}
Suppose $K^{\flat}$ is $C_i$. Is $K$ also $C_i$?
\end{problem}

\subsection{Rank 1 version} \label{rank1version}
Towards transferring information between $K$ and $K^{\flat}$, it is potentially useful to have a version of Theorem \ref{modeltheoreticFont} where $w$ is of rank $1$.  This can be achieved by replacing the ultrapower $K_U$ with its core field $K_U^{\circ}$ (see Remark \ref{corefield}). 
\begin{rem}
\begin{enumerate}[label=(\roman*)]
\item Since $\Z vp$ is cofinal in the value group, we get a well-defined diagonal embedding $(K,v)\hookrightarrow (K_U^{\circ},v_U^{\circ})$. This embedding is also elementary, for instance by Theorem \ref{perfprecversion}. Thus, 
$K_U^\circ$ plays effectively the same role as the ultrapower $K_U$. 

\item The equivalence $K_U-\text{f\'et}\simeq k_w-\text{f\'et}$ factors naturally through $K_U^{\circ}-\text{f\'et}$ and we have that $ K_U^{\circ}-\text{f\'et} \simeq k_w-\text{f\'et}$ restricts to $K-\text{f\'et} \simeq K^{\flat}-\text{f\'et}$. 
\end{enumerate}
\end{rem}
\bt \label{modeltheoreticFont'}
Let $(K,v)$ be a perfectoid field of characteristic $0$. Let $U$ be a non-principal ultrafilter on $\N$ and $(K_U^\circ,v_U^\circ)$ be as above. Let $w$ be the coarsest coarsening of $v^{\circ}$ such that $wp>0$. Then: 
\begin{enumerate}[label=(\Roman*)]
\item The valued field $(K_U^\circ,w)$ is spherically complete with value group $\mathbb{R}$.

\item The tilt $(K^{\flat},v^{\flat})$ embeds elementarily into $(k_w,\overline{v})$. 

\end{enumerate}
Moreover, the equivalence $ K_U^{\circ}-\text{f\'et} \simeq k_w-\text{f\'et}$ restricts to $K-\text{f\'et} \simeq K^{\flat}-\text{f\'et}$.
\et 
The problem of transferring first-order information between $K$ and $K^{\flat}$ is therefore equivalent to the problem of transferring said information between $K_U^{\circ}$ and its residue field $k_w$. 

\section{Almost mathematics revisited}
In this final section, we explain why Theorem \ref{modeltheoreticFont}(I) is a non-standard version of the almost purity theorem over perfectoid valuation rings. More generally, we translate notions from almost ring theory to notions in ring theory. This translation is achieved by means of a functor which associates an object in the almost category to a localization of its ultrapower. 

\subsection{Some background}
We first recall all necessary background from almost mathematics. Our main sources are \cite{GR} and \S 4 \cite{Scholze}. 
\subsubsection{Almost modules}
Let $(K,v)$ be a valued field of rank $1$ with valuation ring $\Oo_K$ and maximal ideal $\mathfrak{m}$ such that $\mathfrak{m}^2=\mathfrak{m}$. The latter condition equivalently means that the valued field $(K,v)$ is non-discrete  (e.g., a perfectoid field).
\begin{definition}
Let $M$ be an $\Oo_K$-module. We say that $x\in M$ is almost zero if $\mathfrak{m}x=0$. The module $M$ is almost zero if all of its elements are almost zero, i.e., $\mathfrak{m}M=0$.
\end{definition}
Recall that, given an abelian category $\mathcal{A}$, a \textit{Serre subcategory} of $\mathcal{A}$ is a nonempty full subcategory $\mathcal{C}$ of $\mathcal{A}$ which is closed under extensions, i.e., given a short exact sequence
$$0\longrightarrow A\longrightarrow B \longrightarrow C \longrightarrow 0 $$
with $A,C\in \mathcal{C}$, we also have $B \in \mathcal{C}$. Write $\Oo_K\lmod$ for the category of $\Oo_K$-modules. The following is easy to show using that $\mathfrak{m}^2=\mathfrak{m}$:
\begin{fact} [Lemma 4.2 \cite{Scholze}] \label{thick}
The full subcategory $\Sigma$ of almost zero $\Oo_K$-modules is a Serre subcategory of $\Oo_K\lmod$.
\end{fact}
Fact \ref{thick} allows us to form a new category, where almost zero information is systematically neglected. Namely, we define the Serre quotient (or Gabriel quotient) below
 $$\Oo_K^{a}\lmod=\Oo_K\lmod/\Sigma$$
called the category of \textit{almost} $\Oo_K$-modules. We recall what this means:
\begin{itemize}
\item The objects of $\Oo_K^{a}\lmod$ are the objects of $\Oo_K\lmod$. For clarity, we shall write $M^a$ for the object of $\Oo_K^{a}\lmod$ corresponding to $M$.
\item $\Hom(M^a,N^a)=\varinjlim \Hom(M',N/N')$, where the limit ranges over submodules $M'\subseteq M$ and $N' \subseteq N$ such that $M/M'\in \Sigma$ and $N'\in \Sigma$. 
\end{itemize}
We have a canonical \textit{almostification functor}
$$\Oo_K\lmod\to \Oo_K^{a}\lmod:M \mapsto M^a$$
which sends a morphism $f:M\to N$ to the corresponding element in the direct limit with $M'=M$ and $N'=0$. The almostification functor is exact, essentially surjective and its kernel is $\Sigma$.

The set $\Hom(M^a,N^a)$ is clearly an honest $\Oo_K$-module from the above definition. It also admits a simpler description: 
\begin{fact} [\S 2.2.11 \cite{GR}] \label{abeliantensoralmost}
Given two $\Oo_K$-modules $M,N$, we have 
$$\Hom(M^a,N^a)=\Hom(\mathfrak{m}\otimes M,N)$$
The functor $M^a\mapsto (M^a)_!=\mathfrak{m}\otimes M$ is a left adjoint to the almostification functor. Moreover, if $M$ is an almost $\Oo_K$-module, then $(M_!)^a=M$.
\end{fact}
We also define $\text{alHom}(M^a,N^a)=\Hom(M^a,N^a)^a$. 
\begin{fact} [\S 2.2.6 \cite{GR}]
The category $\Oo_K^{a}$ is an abelian tensor category, where kernels, cokernels and tensor products are defined in the unique way compatible with their definition in $\Oo_K\lmod$, e.g., 
$$ M^a\otimes N^a=(M\otimes N)^a$$
for any $\Oo_K$-modules $M$ and $N$.
\end{fact}
\begin{notation} \label{almostnotation}
We will sometimes use the following terminology, occasionally used in \cite{GR}, although not explicitly introduced: 
\begin{enumerate}[label=(\roman*)]
\item Given a morphism $M\to N$ of $\Oo_K$-modules, we will say that it is \textit{almost injective} (resp. almost surjective) if the kernel (resp. cokernel) is an almost zero module.

\item  In the above situation, we also say that $M^a\to N^a$ is \textit{injective} (resp. surjective).

\item We say that $M\to N$ is an almost isomorphism if it is both almost injective and almost surjective.
\end{enumerate}

\end{notation}

\subsubsection{Almost algebras and almost elements}
The notion of an almost $\Oo_K$-algebra can now be defined abstractly. We refer the reader to \S 2.2.6 \cite{GR} for the details of the following definition:
\begin{definition} [\S 2.2.6 \cite{GR}, pg. 19 \cite{Scholze}]
\begin{enumerate}[label=(\roman*)]
\item An $\Oo_K^a$-algebra is a commutative unitary monoid $A$ in the abelian tensor category $\Oo_K^a$-mod. This means that there is a multplication map $\mu_A:A\otimes A\to A$ and a unit map $\eta: \Oo_K^a \to A$, satisfying the usual categorical versions of the axioms for a unitary commutative monoid. 
\end{enumerate}
Let $A$ be an $\Oo_K^a$-algebra. 
\begin{enumerate}[label=(\roman*)]
  \setcounter{enumi}{1}
\item An $A$-module is an $\Oo_K^a$-module $M$ together with a scalar multiplication map $A\otimes M \to M$, again defined in a categorical fashion. 
\item An $A$-algebra is a commutative unitary monoid $B$ in the abelian tensor category $A\lmod$, or equivalently an $\Oo_K^a$-algebra together with an algebra morphism $A\to B$.
\end{enumerate}
\end{definition}
We will shortly see that almost algebras come from almostifications of honest algebras. First, we need the functor of almost elements:  
\begin{fact} [Proposition 2.2.14 \cite{GR}]
There is a right adjoint to the almostification functor
$$ \Oo_K^a\lmod \to \Oo_K\lmod:M\mapsto M_* $$
given by the functor of almost elements 
$$M_*=\Hom(\Oo_K^a,M^a)$$
Moreover, we have that $(M_*)^a=M$. If $M=N^a$, then $M_*=\Hom(\mathfrak{m},N)$.
\end{fact}
\begin{rem} \label{almostelrem}
\begin{enumerate}[label=(\roman*)] 

\item The adjoints $({-})_*$ and $({-})_!$ both yield full embeddings of $\Oo_K^a\lmod$ into $\Oo_K\lmod$ since they are right inverses to the almostification functor $({-})^a$. 

\item The functor $({-})_*$ is monoidal, i.e., it preserves unit maps and there is a canonical map 
$$F: M_*\otimes_{\Oo_K} N_* \to (M\otimes_{\Oo_K^a}  N)_*$$ 
In particular, if $A$ is an $\Oo_K^a$-algebra, then $A_*$ is naturally an $\Oo_K$-algebra and $(A_*)^a=A$. 
We stress that $F$ is usually not an isomorphism, i.e., $({-})_*$ \textit{does not} preserve tensor products in general. For instance, if $K=\Q_p(p^{1/p^{\infty}})$ and $K'=K(p^{1/2})$, then $(\Oo_{K'}^a)_*=\Oo_K$ and one can check that $p^{1/2}\otimes p^{-1/2}\in (\Oo_{K'}^a \otimes_{\Oo_K^a} \Oo_{K'}^a)_*$ but $p^{1/2}\otimes p^{-1/2}\notin \Oo_{K'}\otimes_{\Oo_K} \Oo_{K'}$. 

\item On the other hand, $({-})_!$ does preserve tensor products but it does not preserve unit maps since $(\Oo_K^a)_!=\mathfrak{m}\neq \Oo_K$. In particular, $({-})_!$ does not preserve monoidal structures. 
\end{enumerate}
\end{rem}
In \S \ref{embeddingsection}, we define yet another embedding of $\Oo_K^a\lmod$ into a concrete category of modules. 
In contrast with $({-})_*$, this also preserves tensor products, at least under some finite generation assumption on the almost modules being tensored.
\subsubsection{Almost commutative algebra}
We are now ready to define almost analogues of classical ring-theoretic notions.
\begin{definition} \label{definitionalmostca}
Let $A$ be an $\Oo_K^a$-algebra.
\begin{enumerate}[label=(\roman*)]

\item An $A$-module is flat if the functor $X\mapsto M\otimes_A X$ on $A$-modules is exact.

\item  An $A$-module is almost projective if the functor $X\mapsto \text{al}\Hom_A(M,X)$ on $A$-modules is exact. 

\item If $R$ is an $\Oo_K$-algebra and $N$ is an $R$-module, then $M=N^a$ is said to be almost finitely generated (resp. almost finitely presented) $R^a$-module if and only if for all $\varepsilon\in \mathfrak{m}$, there is some finitely generated (resp. finitely presented) $R$-module $N_{\varepsilon}$ with a map $f_{\varepsilon}:N_{\varepsilon}\to N$ such that the kernel and cokernel of $f_{\varepsilon}$ are killed by $\varepsilon$. We further say that $M$ is uniformly almost finitely generated if there is some integer $n$ such that $N_{\varepsilon}$ above can be chosen to be generated by $n$ elements.
\end{enumerate}
\end{definition}

\begin{rem} \label{almostringrem}
Let $A$ be an $\Oo_K^a$-algebra and $M$ be an $A$-module.
\begin{enumerate}[label=(\roman*)]
\item The functor $X\mapsto M\otimes_A X$ is always right exact and $X\mapsto \text{al}\Hom_A(M,X)$ is always left exact, in accordance with the analogous facts in usual commutative algebra.

\item If $M$ is finitely generated, then $M^a$ is clearly uniformly almost finitely generated. On the other hand, the module $\mathfrak{m}^a$ is uniformly almost finitely generated but $\mathfrak{m}$ is \textit{not} almost isomorphic to a finitely generated module. 

\item The usual categorical notion of projectivity in $\Oo_K^a$-mod is ill-behaved because even $\Oo_K^a$ is not a projective module over itself (see Remark 4.8 \cite{Scholze}). 
It is for this reason that the term \quotes{almost projective} is being used.

\end{enumerate}

\end{rem}
For lack of reference, we also record the following observation:
\bl \label{almostringlem}
Let $A$ be an $\Oo_K$-algebra and $M$ be an $A$-module. Then $M^a$ is flat as an $A^a$-module if and only if for any finitely generated $A$-modules $N_1,N_2$ and any almost injective map $N_1\to N_2 $, the map $M\otimes N_1\to M\otimes N_2$ is almost injective. 
\el 
\begin{proof}
The forward direction is clear, so we only prove the converse. Let $M_1^a\to M_2^a$ be an arbitrary injective map. Replacing $M_1$ with $(M_1^a)_*$ and $M_2$ with $(M_2^a)_*$, we can assume that this comes from an injective map $\psi: M_1\to M_2$. We may now write $\psi=\varinjlim \psi_j$, where $\psi_j:M_{1,j}\to M_{2,j}$ is injective and each $M_{i,j}$ is finitely generated. By assumption, we get that $M\otimes M_{1,j}\to M\otimes M_{2,j}$ is almost injective. Passing to direct limits, we get that $\varinjlim(M\otimes M_{1,j})\to \varinjlim (M\otimes M_{2,j})$ is almost injective. To see this, recall that direct limits preserve exact sequences of modules (see Tag 00DB \cite{sp}) and observe that almost zero modules are closed under direct limits. Finally, recall that tensor products commute with direct limits (cf. Tag 00DD \cite{sp}) and therefore $\varinjlim (M\otimes M_{i,j})=M\otimes M_i$ for $i=1,2$. It follows that $M^a\otimes M_1^a\to M^a\otimes M_2^a$ is injective.
\end{proof}
Recall that projective modules are always flat and the converse is true for finitely presented modules. This has an analogue in the almost setting:
\begin{fact} [\cite{GR}] \label{flatvsproj}
Let $A$ be an $\Oo_K^a$-algebra and $M$ be an $A$-module. Then: 
\begin{enumerate}[label=(\roman*)]
\item If $M$ is almost projective as an $A$-module, then $M$ is flat as an $A$-module.
\item $M$ is flat and (uniformly) almost finitely presented if and only if it is almost projective and (uniformly) almost finitely generated.
\end{enumerate}
\end{fact}
\begin{proof}
For part (i), see Remark 2.4.12(i) \cite{GR} and for part (ii) see Proposition 2.4.18 \cite{GR}.
\end{proof}
We use the following terminology, which is used in \cite{Scholze} but not in \cite{GR}.
\begin{definition}
Let $A$ be an $\Oo_K^a$-algebra and $M$ be an $A$-module. We say that $M$ is \textit{(uniformly) finite projective} if it is almost projective and (uniformly) almost finitely generated as an $A$-module.
\end{definition}

We finally arrive at the central concept of an \'etale extension (cf. \S 3.1 \cite{GR}):
\begin{definition} 
Let $A$ be an $\Oo_K^a$-algebra and $B$ be an $A$-algebra.
\begin{enumerate}[label=(\roman*)]
\item  We say that $B$ is unramified over $A$ if $B$ is almost projective as a $B\otimes_A B$-module under the multiplication map $\mu:B\otimes_A B\to B$.
\item We say that $B$ is \'etale over $A$ if it is flat and unramified.
\item We say that $B$ is finite \'etale if it is \'etale and $B$ is an almost finitely presented $A$-module.
\end{enumerate}
\end{definition}
The usual diagonal idempotent criterion for unramified morphisms has an analogue in the almost context. In \cite{Scholze}, this criterion is used as the definition of an unramified morphism. 
\begin{fact} [Proposition 3.1.4 \cite{GR}] \label{almostetalerem} 
Let $A$ be an $\Oo_K^a$-algebra and $B$ be an $A$-algebra. Write $\mu:B\otimes_A B\to B$ for the multiplication map. Then the following are equivalent:
\begin{enumerate}[label=(\roman*)]
\item $B$ is unramified over $A$.
\item There exists an almost element $e\in (B\otimes_A B)_*$ such that $e^2=e$, $\mu(e)=1$ and $x\cdot e=0$ for all $x\in \text{ker}(\mu)_*$.
\end{enumerate}
\end{fact}
\begin{example} [Example 4.3.2 \cite{BB}] \label{quadex}
Let $K$ be the $p$-adic completion of $\Q_p(p^{1/p^{\infty}})$ and $K'=K(p^{1/2})$. We claim that $\Oo_{K'}^a$ is an \'etale $\Oo_K^a$-algebra. To see that $\Oo_{K'}^a/\Oo_K^a$ is unramified, we find an element $e$ as in Fact \ref{almostetalerem}(ii). Let
$$e=\frac{1}{2} (p^{-1/2}\otimes p^{1/2}+1)\in K'\otimes_K K' $$
We claim that $e\in (\Oo_{K'}^a\otimes_{\Oo_K^a} \Oo_{K'}^a)_* $. For each $n\in \N$, we need to prove that $p^{1/p^n}\cdot e \in \Oo_{K'}\otimes_{\Oo_K} \Oo_{K'}$. It is clear that $p\cdot e \in \Oo_{K'}\otimes_{\Oo_K} \Oo_{K'}$. For larger $n$, observe that $p^{-1/2}\otimes p^{1/2}=p^{-1/2p^n}\otimes p^{1/2p^n}$ in $K'\otimes_K K'$, so we can write 
$$e=\frac{1}{2} (p^{-1/2p^n}\otimes p^{1/2p^n}+1) $$
From this presentation, it is clear that $p^{1/p^n}\cdot e \in \Oo_{K'}\otimes_{\Oo_K} \Oo_{K'}$. We also check that $e^2=e$, $\mu(e)=1$ and $x\cdot e=0$ for all $x\in \text{ker}(\mu)_*$. Regarding flatness, we note in general that for any $\Oo_K$-algebra $A$ the almostification functor 
$$A\lmod\to A^a\lmod:M\mapsto M^a$$ 
sends flat $A$-modules to flat $A^a$-modules (see \S 2.4.8 \cite{GR}).
\end{example}

\subsection{Non-standard version} \label{nonstdvers}
We fix some notation which will be used throughout this section:
\begin{itemize}

\item Let $(K,v)$ be a non-discrete valued field of rank $1$ and $\varpi \in \mathfrak{m}\backslash \{0\}$. 
\item Let $U$ be a non-principal ultrafilter on $\N$ and $(K_U,v_U)$ be the corresponding ultrapower. 
\item Let $w$ be the coarsest coarsening of $v_U$ with $w\varpi>0$.
\item   We write $S$ for the multiplicative set of elements of infinitesimal valuation, i.e.,
$$S=\{x\in \Oo_{v_U}:n\cdot vx<v\varpi \mbox{ for all }n\in \N \}$$ 
\item Given a sequence $(M_i)_{i\in \omega}$ of $\Oo_{K}$-modules, we write $S^{-1}(\ulim M_i)$ for the localization of the ultraproduct $\ulim M_i$ at $S$. 
\end{itemize}
\end{definition}
\subsubsection{Almost sequences of modules} \label{almostseq}
It will be convenient to introduce a category larger than the category of almost modules, which also encodes information about sequences of modules. 
\begin{definition}
A sequence $(M_i)_{i\in \omega }$ of $\Oo_K$-modules is called almost zero if for every $\varepsilon \in \mathfrak{m}$, we have that
$$\{i\in \omega: \varepsilon \cdot M_i=0\}\in U$$
We write $\Sigma^+$ for the full subcategory of almost zero objects in $(\Oo_K\lmod)^\omega$.
\end{definition}
This is clearly compatible with the almost mathematics definition via the diagonal functor $\Delta: \Oo_K\lmod \to (\Oo_K\lmod)^\omega:M\mapsto (M,M,...)$:
\bl \label{almostzerolem}
Let $M$ be an $\Oo_K$-module. Then $M\in \Sigma$ if and only if $\Delta M \in \Sigma^+$. 
\el 
Note that almost zero sequences do not always arise as sequences of almost zero modules:
\begin{example}
Let $K$ be a perfectoid field and $\varpi$ be a pseudo-uniformizer admitting a system of $p$-power roots. For each $i\in \omega$, let $M_i= \Oo_K/(\varpi^{1/p^i})$. One readily checks that $(M_i)_{i\in \omega} \in \Sigma^+$ but $M_i \notin \Sigma$ for all $i\in \omega$. 
\end{example}


\bp \label{embeddingcat1}
Let $A$ be an $\Oo_K$-algebra. Consider the assignment
$$S^{-1}(\ulim ({-})): (A\lmod)^\omega \to S^{-1}A_U\lmod:(M_i)_{i\in \omega} \mapsto S^{-1}(\ulim M_i)$$
which acts on morphisms in the obvious way.  Then:
\begin{enumerate}[label=(\roman*)]
\item $S^{-1}(\ulim ({-}))$ is an exact functor. 
\item We have $\text{Ker}(S^{-1}(\ulim ({-})))=\Sigma^+$. In particular, $\Sigma^+$ is a Serre subcategory of $(\Oo_K\lmod)^\omega$.
\item $S^{-1}(\ulim ({-}))$ is monoidal, i.e., there exists a natural map  
$$S^{-1}(\ulim M_i)\otimes_{S^{-1}A_U} S^{-1}(\ulim N_i) \to S^{-1} (\ulim M_i\otimes_A N_i)$$
for any modules $M_i$ and $N_i$. Moreover, if $(M_i)_{i\in \omega}$ and $(N_i)_{i\in \omega}$ are uniformly finitely generated, then the above map is an isomorphism.
\item $\Sigma^+$ is an ideal of $((\Oo_K\lmod)^\omega,\otimes)$.
\end{enumerate}
\ep 
\begin{proof}
(i) By Lemma \ref{ultramodules}, we have that the ultraproduct functor 
$$\mathcal{U}: (A\lmod)^\omega \to A_U\lmod:(M_i)_{i\in \omega} \mapsto \ulim M$$
is exact. Localization at the multiplicative set $S$ gives rise to a functor 
$$ A_U\lmod\to S^{-1}A_U\lmod :N \mapsto S^{-1}N_U$$
which is also exact by Tag 00CS \cite{sp}. Now $S^{-1}(\ulim ({-}))$ is visibly the composition of the two functors above and is therefore exact. \\
(ii) We need to prove that a sequence $(M_i)_{i\in \omega}$ is almost zero if and only if $S^{-1}(\ulim M_i)=0$. Assume that $S^{-1}(\ulim M_i)=0$ and suppose that $\varepsilon \in \mathfrak{m}$ is such that $\{i\in \omega: \varepsilon \cdot M_i=0\} \notin U$. For almost all $i$, we can then find some $m_i\in M_i$ such that $\varepsilon \cdot m_i\neq 0$. Therefore, we get $\varepsilon \cdot m\neq 0$, where $m=\ulim m_i$. Since $S^{-1}(\ulim M_i)=0$, there exists $s\in S$ such that $s\cdot m=0$. Since $s$ divides $\varepsilon$ in $\Oo_{v_U}$, this gives $\varepsilon\cdot m=0$, which is a contradiction. Conversely, suppose that for every $\varepsilon \in \mathfrak{m}$, we have that $\{i\in \omega: \varepsilon \cdot M_i=0\}\in U$. Applying that non-principal ultraproducts are $\aleph_1$-saturated, we can now find $s\in S$ such that $s\cdot \ulim M_i=0$.\footnote{To see this directly, choose elements $\varepsilon_i \in \mathfrak{m}$ generating $\mathfrak{m}$ with $v\varepsilon_{i+1}<v\varepsilon_i$. Let $k(i)$ be the maximum $j\leq i$ such that $\varepsilon_{j} \cdot M_i=0$ if such a $j$ exists and otherwise set $k(i)=1$. Now take $s=\ulim \varepsilon_{k(i)}$ and use \L o\'s to show that $s\in S$ and $s\cdot M_U=0$.}
In particular, we get that $S^{-1}(\ulim M_i)=0$. We conclude that $\text{Ker}(S^{-1}(\ulim ({-})))=\Sigma^+$. 
By part (i),  $\Sigma^+$ is a Serre subcategory, being the kernel of an exact functor.  \\
(iii) By Lemma \ref{ultramodules}, we have a natural map
$$F:\ulim M_i\otimes_{A_U} \ulim N_i \to \ulim (M_i\otimes_A N_i)$$
By Lemma \ref{ultramodules}(iii), this is an isomorphism if $(M_i)_{i\in \omega}$ and $(N_i)_{i\in \omega}$ are uniformly finitely generated. Since localization commutes with tensor products, we have that
$$S^{-1}(\ulim M_i\otimes \ulim N_i)=S^{-1}(\ulim M_i)\otimes S^{-1}(\ulim N_i)$$ 
and the desired map is simply $S^{-1}F$.\\ 
For part (iv), note that if $\varepsilon\cdot M=0$, then $\varepsilon\cdot( M\otimes N)=0$, for any module $N$.
\end{proof}

\begin{definition}
Let $A$ be an $\Oo_K$-algebra. The category of almost sequences of $A$-modules is the Serre quotient 
$$\Delta A^a\lmod= (A\lmod)^\omega/\Sigma^+ $$
We write 
$$(A\lmod)^\omega\to \Delta A^a: (M_i)_{i\in \omega}\mapsto (M_i)_i^a $$
for the almostification functor. 
\end{definition}

\bc [cf. Corollary \ref{abeliantensoralmost}] \label{lumodisnice}
The category $\Delta A^a\lmod$ is an abelian tensor category, where we define kernels, cokernels and tensor products in the unique way compatible with their definition in $(A-mod)^{\omega}$. For instance, we have
$$(M_i)_i^a\otimes (N_i)_i^a=(M_i\otimes N_i)_i^a$$
\ec 
\begin{proof}
Clear from Proposition \ref{embeddingcat1}.
\end{proof}
\begin{notation} 
In parallel with \ref{almostnotation}, we will use the following terminology:
\begin{enumerate}[label=(\roman*)]
\item We will say that a morphism $(M_i)_{i\in \omega}\to (N_i)_{i\in \omega}$ is almost injective (resp. almost surjective) if the kernel (resp. cokernel) is an almost zero sequence of modules. 
\item In that case, we also say that $(M_i)_i^a\to (N_i)_i^a$ is injective (resp. surjective), although strictly speaking morphisms between almost sequences of modules are \textit{not} functions between sets.
\item We say that $(M_i)_{i\in \omega}\to (N_i)_{i\in \omega}$ is an almost isomorphism (or that $(M_i)_{i\in \omega}$ is almost isomorphic to $(N_i)_{i\in \omega}$) if it is both almost injective and almost surjective.
\end{enumerate}

\end{notation}
\bl \label{almostfg}
Let $A$ be an $\Oo_K$-algebra. Then: 
\begin{enumerate}[label=(\roman*)]
\item The diagonal functor $\Delta:A\lmod\to (A\lmod)^\omega$ descends to an exact faithful functor 
$$ (A^a\lmod,\otimes) \to (\Delta A^a\lmod,\otimes) :N^a\mapsto (\Delta N)^a$$
which preserves tensor products. This functor will also be denoted by $\Delta $. 
\end{enumerate}
Let $M$ be an $A$-module. Then:
\begin{enumerate}[label=(\roman*)]
  \setcounter{enumi}{1}
\item $M^a$ is (uniformly) almost finitely generated if and only if there is an almost isomorphism $(M_i)_{i\in \omega}\to \Delta M$ with $(M_i)_{i\in \omega}$ (uniformly) finitely generated.
\item $M^a$ is (uniformly) almost finitely presented if and only if there is an almost isomorphism $(M_i)_{i\in \omega}\to \Delta M$ with $(M_i)_{i\in \omega}$ (uniformly) finitely presented.
\end{enumerate}
\el 
\begin{proof}
(i) This follows from Lemma \ref{almostzerolem} and the universal property of Serre quotients (see Tag 02MS \cite{sp}). Tensor products in $A^a\lmod$ (resp. $\Delta A^a\lmod$) are compatible with $A\lmod$ (resp. $\Delta A\lmod$) and are thus preserved. Parts (ii) and (iii) are clear from the definitions.
\end{proof}

\subsubsection{$A^a\lmod$ as a subcategory of $S^{-1}A_U\lmod$} \label{embeddingsection}
We use the (suggestive) notation $S^{-1}({-})_U$ for several related functors; it will always be clear, either from context or by explicitly saying so, which variant is being used. 
\bp \label{embeddingcat}
Let $A$ be an $\Oo_K$-algebra. Then: 
\begin{enumerate}[label=(\roman*)]
\item We have an exact faithful functor of abelian categories 
$$ S^{-1}(\ulim ({-}) ): \Delta A^a\lmod \to S^{-1}A_U\lmod:(M_i)_i^a\mapsto S^{-1}(\ulim M_i)$$
which is also monoidal, i.e., there is a natural map 
$$S^{-1}(\ulim M_i)\otimes_{S^{-1}A_U} S^{-1}(\ulim N_i)\to S^{-1}(\ulim M_i\otimes_A N_i)$$
Moreover, this map is an isomorphism if $(M_i)_{i\in \omega}$ and $(N_i)_{i\in \omega}$ are both almost isomorphic to uniformly finitely generated sequences of $A$-modules.
\item We have an exact faithful functor of abelian categories
$$S^{-1}({-} )_U: A^a\lmod\to S^{-1}A_U\lmod:M^a\mapsto S^{-1}M_U$$
which is also monoidal, i.e., there is a natural map 
$$S^{-1}M_U\otimes_{S^{-1}A_U} S^{-1}N_U\to S^{-1}(M\otimes_A N)_U$$
Moreover, this map is an isomorphism if $M^a$ and $N^a$ are uniformly almost finitely generated over $A^a$.
\end{enumerate}
\ep 
\begin{proof}
(i) By Proposition \ref{embeddingcat1}, we have an exact monoidal functor 
$$S^{-1}(\ulim ({-}) ): (A\lmod)^\omega \to S^{-1}A_U\lmod:(M_i)_{i\in \omega} \mapsto S^{-1}(\ulim M_i)$$
with kernel $\Sigma^+$. By the universal property of Serre quotients (see Tag 06XK \cite{sp}), this induces an exact faithful monoidal functor as needed. The natural map is the one provided by Proposition \ref{embeddingcat1}(iii). \\
(ii) By Lemma \ref{almostfg}, we have an exact faithful monoidal functor 
$$ \Delta: (A^a\lmod,\otimes) \to (\Delta A^a\lmod,\otimes) :M^a\mapsto (\Delta M)^a$$
Note that $S^{-1}({-})_U$ is the composition of $S^{-1}(\ulim ({-}) )$ and $\Delta$, so it is an exact faithful monoidal functor. Now let $M^a$ and $N^a$ be uniformly almost finitely generated. By Lemma \ref{almostfg}, we have $M=(M_i)_i^a$ and $N=(N_i)_i^a$, for some uniformly finitely generated sequences of $A$-modules $M_i$ and $N_i$. We conclude from part (i). 
\end{proof}

\bc \label{preservationrem}
Let $A$ be an $\Oo_K$-algebra. Then the functor $S^{-1}({-} )_U: A^a\lmod \to S^{-1}A_U\lmod:M^a\mapsto S^{-1}M_U $ preserves and reflects short exact sequences, i.e.,
$$0 \longrightarrow M_1^a \longrightarrow M_2^a \longrightarrow M_3^a \longrightarrow 0 $$
is an exact sequence in $A^a\lmod $ if and only if the corresponding sequence
$$0 \longrightarrow S^{-1}M_{1,U} \longrightarrow S^{-1}M_{2,U}  \longrightarrow S^{-1}M_{3,U}  \longrightarrow 0 $$
is exact in $S^{-1}A_U$-mod. Similarly for $S^{-1}({-} )_U: \Delta A^a\lmod \to S^{-1}A_U\lmod$.
\ec
\begin{proof}
Exact faithful functors both preserve and reflect exact sequences; see Theorem 3.21 \cite{freyd} and the subsequent comment on pg. 67 \textit{loc.cit.}
\end{proof}
In particular, the functor $S^{-1}({-} )_U$ reflects isomorphisms. Namely, if $f\in \Hom(M^a, N^a)$ is such that $S^{-1}(f)_U$ is an isomorphism, then $f$ is also an isomorphism.
Beware that $S^{-1}M_U\cong S^{-1}M_U'$ does not imply that $M^a\cong M'^a$ in general. The point is that $S^{-1}({-} )_U$ is \textit{not full} and hence a morphism between $S^{-1}M_U$ and $S^{-1}M_U'$ may not come from a morphism between $M^a$ and $M'^a$. 
\begin{example} [cf. Example 2.4 \cite{ScholzeHodge}]
Let $K$ be a perfectoid field, $\varpi$ a pseudo-uniformizer and $r \in \mathbb{R} \backslash \Gamma$. Consider the non-principal ideal 
$$I_r=\{x\in \Oo_K: v(x)>r\}\subseteq \Oo_K$$
Applying $S^{-1}({-})_U$ to $I_r$ gives the \textit{principal} ideal $$J_r=\{x\in \Oo_w: w(x)\geq r\} \subseteq \Oo_w$$ 
Now $J_r$ is clearly isomorphic to $\Oo_w$ as an $\Oo_w$-module but one can show that $I_r^a\not \cong \Oo_K^a$ in $\Oo_K^a\lmod$. 
\end{example}
We now start developing a dictionary between the two frameworks.
\bl \label{fgalmostmodule}
Let $A$ be an $\Oo_K$-algebra and $(M_i)_{i\in \omega}$ be a sequence of $A$-modules. Then:
\begin{enumerate}[label=(\roman*)]
\item Every finitely generated submodule of $S^{-1}(\ulim M_i)$ is of the form $S^{-1}(\ulim N_i)$, for some uniformly finitely generated $A$-submodules $N_i$ of $M_i$. 

\item $(M_i)_{i\in \omega}$ is almost isomorphic to a sequence of uniformly finitely generated $A$-modules if and only if $S^{-1}(\ulim M_i)$ is a finitely generated $S^{-1}A_U$-module. 

\item Let $M$ be an $A$-module. Then $M^a$ is uniformly almost finitely generated if and only if $S^{-1}M_U$ is a finitely generated $S^{-1}A_U$-module. 
\item Let $M$ be an $A$-module. Then $M^a$ is uniformly almost finitely presented if and only if $S^{-1}M_U$ is a finitely presented $S^{-1}A_U$-module. 
\end{enumerate}
\el 
\begin{proof}
(i) Let $\phi$ be the inclusion map of a finitely generated submodule of $S^{-1}(\ulim M_i)$
$$\phi: S^{-1}A_U\cdot m_1+...+S^{-1}A_U\cdot m_n \to S^{-1}(\ulim M_i) $$ 
where $m_i\in S^{-1}(\ulim M_i)$. By rescaling the $m_i$'s with some $s\in S$, we may assume that  $m_i \in \ulim M_i$. We then get a natural inclusion  
$$\Phi:A_U\cdot m_1+...+A_U\cdot m_n \to \ulim M_i$$ 
such that $S^{-1}\Phi=\phi$. We can write $m_j=\ulim m_{j,i}$ and $\Phi=\ulim \Phi_i$, where 
$$\Phi_i: A\cdot m_{1,i}+...+A\cdot m_{n,i}\to M_i$$
is the inclusion map. Set $N_i=A\cdot m_{1,i}+...+A\cdot m_{n,i}$ and note that $S^{-1}(\ulim N_i)=S^{-1}A_U\cdot m_1+...+S^{-1}A_U\cdot m_n$ as needed.\\
(ii) Suppose that $(M_i)_{i}^a=(M_i')_i^a$ and that the $M_i$'s are uniformly finitely generated. By Proposition \ref{embeddingcat}, we have that $S^{-1}(\ulim M_i)=S^{-1}(\ulim M_i')$ and the latter is clearly finitely generated. Conversely, suppose that $S^{-1}(\ulim M_i)$ is finitely generated. By (i), we can find uniformly finitely generated $A$-submodules $N_i$ of $M_i$ such that $S^{-1}(\ulim N_i)=S^{-1}(\ulim M_i)$. By Corollary \ref{preservationrem}, we get that the map $(N_i)_{i\in \omega}\to (M_i)_{i\in \omega}$, obtained by patching the inclusions $N_i\to M_i$, is an almost isomorphism.\\
(iii) From (ii) and Lemma \ref{almostfg}(ii).\\
(iv) Suppose $M^a$ is uniformly almost finitely presented. By Lemma \ref{almostfg}(iii), $\Delta M$ is almost isomorphic to a sequence $(M_i')_{i\in \omega}$ of uniformly almost finitely presented modules. Now $M_U'$ is clearly a finitely presented $A_U$-module and hence $S^{-1}M_U'$ is a finitely presented $S^{-1}A_U$-module. The conclusion follows because $S^{-1}M_U=S^{-1}M_U'$. 

Conversely, suppose 
there exist $m,n\in \N$ and a short exact sequence of $S^{-1}A_U$-modules
$$0\longrightarrow S^{-1} A_U^m \stackrel{\phi}\longrightarrow S^{-1}A_U^n \stackrel{\psi}\longrightarrow S^{-1}M_U\longrightarrow 0$$
This induces an isomorphism $\overline{\psi}: S^{-1}A_U^n/S^{-1}A_U^m\to S^{-1}M_U$. Write $(\varepsilon_i)_{i=1}^m$ and $(e_i)_{i=1}^n$ for the standard bases of $S^{-1}A_U^m$ and $S^{-1}A_U^n$. By rescaling the $e_i$'s with some $s\in S$, we may assume that $\psi(e_i)\in M_U$. This defines a map $\Psi: A_U^n \to M_U$ with $\psi=S^{-1}\Psi$. By rescaling the $\varepsilon_i$'s, we may assume that each $\phi(\varepsilon_i)$ is in the $A_U$-linear span of the $e_i$'s and moreover that  $\Psi(\phi(\varepsilon_i))=0$. We thus obtain $A_U$-linear maps
$$A_U^m \stackrel{\Phi}\longrightarrow A_U^n \stackrel{\Psi}\longrightarrow M_U $$
with $\Psi\circ \Phi=0$, $\phi=S^{-1}\Phi$ and $\psi=S^{-1}\Psi$. Moreover, the map $\Phi: A_U^m\to A_U^n$ is injective. Arguing as in (i), we obtain $A$-linear maps
$$A^m \stackrel{\Phi_i}\longrightarrow A^n \stackrel{\Psi_i}\longrightarrow M $$
with $\Phi=\ulim \Phi_i$ and $\Psi=\ulim \Psi_i$. In particular, we have that $\Psi_i\circ \Phi_i=0$ and that $\Phi_i$ injective for almost all $i$. For almost all $i$, we identify $A^m$ with its image in $A^n$ via $\Phi_i$. With this identification in place, we obtain a map $\overline{\Psi}_i:A^n/A^m\to M$ induced from $\Psi_i$. Set $\overline{\Psi}_i=0$ for all other $i$. Since $\overline{\psi}=S^{-1}\overline{\Psi}_U$ is an isomorphism, we get that $(\overline{\Psi}_i)_{i\in \omega}$ is an almost isomorphism by Corollary \ref{preservationrem}.  We conclude from Lemma \ref{almostfg}(iii).
\end{proof}
\bt \label{dictionary}
Let $A$ be an $\Oo_K$-algebra and $M$ be an $A$-module. Then:
\begin{enumerate}[label=(\roman*)]
\item Suppose $M^a$ is uniformly almost finitely generated. Then $M^a$ is flat as an $A^a$-module if and only if $S^{-1}M_U$ is flat as an $S^{-1}A_U$-module.
\item $M^a$ is uniformly finite projective as an $A^a$-module if and only if $S^{-1}M_U$ is finite projective as an $S^{-1}A_U$-module.
\end{enumerate}
Let $B$ be an $A$-algebra such that $B^a$ is is uniformly almost finitely generated as an $A^a$-module. Then:
\begin{enumerate}[label=(\roman*)]
  \setcounter{enumi}{2}
\item $B^a$ is unramified over $A^a$ if and only if $S^{-1}B_U$ is unramified over $S^{-1}A_U$.

\item $B^a$ is \'etale over $A^a$ if and only $S^{-1}B_U$ is \'etale over $S^{-1}A_U$.
\end{enumerate}
\et 
\begin{proof}
(i) Suppose that $M^a$ is flat. By Tag 00HH \cite{sp}, it suffices to show that for any finitely generated ideal $J \trianglelefteq S^{-1}A_U$, the multiplication map $J \otimes S^{-1}M_U\to S^{-1}M_U$ is injective.  By Lemma \ref{fgalmostmodule}(i), we can write $J =S^{-1} (\ulim I_i)$, where the $I_i$'s are uniformly finitely generated ideals of $A$. Since $M^a$ is flat, we have that the multiplication map $\phi_i: I_i\otimes_A M\to M$ is almost injective for each $i$.
By Lemma \ref{almostfg}(ii), $\Delta M$ is almost isomorphic to a sequence of uniformly finitely generated modules. By Proposition \ref{embeddingcat}(i), we have that the canonical map 
$$S^{-1}(\ulim I_i)\otimes_{S^{-1}A_U} S^{-1} M_U\to S^{-1}(\ulim (I_i \otimes_A M))$$ 
is an isomorphism. We now have a commutative diagram 
\[
\begin{tikzcd}[column sep=4.5em]
S^{-1}(\ulim (I_i \otimes_A M)) \arrow{r}{S^{-1}(\ulim \phi_i)} & S^{-1}M_U \\
S^{-1}(\ulim I_i)\otimes_{S^{-1}A_U} S^{-1}M_U \arrow{ur} \arrow{u}{\cong}
\end{tikzcd}
\]
where the map $S^{-1}(\ulim \phi_i)$ is injective because $(\phi_i)_i^a$ is injective and $S^{-1}(\ulim ({-}))$ is exact. 

Conversely, suppose $S^{-1}M_U$ is flat as an $S^{-1}A_U$-module. By Lemma \ref{almostringlem}, it suffices to show that for any  finitely generated $A$-modules $M_1$ and $M_2$ and any almost injective map $\phi:M_1\to M_2$ in $A\lmod$, the map 
$$\id_M \otimes \phi: M\otimes_{A} M_1\to M\otimes_{A} M_2$$ 
is almost injective. The map $S^{-1}\phi_U: S^{-1}M_{1,U} \to S^{-1}M_{2,U}$ is injective because $S^{-1}({-})_U$ is exact. Since $S^{-1}M_U$ is flat, we get that the map
$$ \id_{S^{-1}M_U} \otimes S^{-1}\phi_U : S^{-1}M_U \otimes_{S^{-1}A_U} S^{-1}M_{1,U} \to S^{-1}M_U \otimes_{S^{-1}A_U} S^{-1}M_{2,U}$$ 
is injective. By Proposition \ref{embeddingcat}(ii) and since $M^a,M_1^a$ and $M_2^a$ are uniformly almost finitely generated, we have canonical isomorphisms 
$$S^{-1}M_U \otimes_{S^{-1}A_U} S^{-1}M_{i,U} \stackrel{\cong}\to S^{-1}(M\otimes_A M_i)_U$$ 
for $i=1,2$. These isomorphisms fit into the commutative diagram below
\[
\begin{tikzcd}[column sep=4.5em]
S^{-1}M_U \otimes_{S^{-1}A_U} S^{-1}M_{1,U}  \arrow{r}{\id_{S^{-1}M_U} \otimes S^{-1}\phi_U }& S^{-1}M_U \otimes_{S^{-1}A_U} S^{-1}M_{2,U}  \\
 S^{-1}(M\otimes_A M_1)_U \arrow{u}{\cong} \arrow{r}{S^{-1}(\id_M \otimes \phi)} & S^{-1}(M\otimes_A M_2)_U \arrow{u}{\cong}
\end{tikzcd}
\]
Since $\id_{S^{-1}M_U} \otimes S^{-1}\phi_U$ is injective, we also get that 
$S^{-1}(\id_M \otimes \phi)_U$ is injective. By Corollary \ref{preservationrem}, $S^{-1}({-})_U:A^a\lmod\to S^{-1}A_U\lmod$ reflects exact sequences and hence $\id_M \otimes \phi$ is almost injective. \\
(ii) By Fact \ref{flatvsproj}(ii), $M^a$ is uniformly finite projective if and only if $M^a$ is flat and uniformly finitely presented. By (i) and Lemma \ref{fgalmostmodule}(iv), this is equivalent to $S^{-1}M_U$ being flat and finitely presented as an $S^{-1}A_U$-module. The latter is equivalent to $S^{-1}M_U$ being projective as an $S^{-1}A_U$-module by Tag 058RA \cite{sp}.  \\
(iii) Note throughout that, as a consequence of Proposition \ref{embeddingcat}(ii), we have 
$$S^{-1}(B\otimes_A B)_U=S^{-1}B_U\otimes_{S^{-1}A_U} S^{-1}B_U$$ 
Now $B^a$ is unramified over $A^a$ if and only if $B^a$ is almost projective as a $B^a\otimes_{A^a} B^a$-module. Since $B^a$ is clearly uniformly almost finitely generated over $B^a\otimes_{A^a} B^a$, this is also equivalent to $B^a$ being uniformly finite projective as a $B^a\otimes B^a$-module. By (ii), this is equivalent to $S^{-1}B_U$ being finite projective as a $S^{-1}B_U\otimes_{S^{-1}A_U} S^{-1}B_U$-module. The latter is equivalent to $S^{-1}B_U$ being unramified over $S^{-1}A_U$.\\
(iv) Clear from (i) and (iii).
\end{proof}
\begin{example} [cf. Example 4.9 \cite{Scholze}]\label{sqrootexmplefg}
Assume $p\neq 2$. Let $K$ be the $p$-adic completion of $\Q_p(p^{1/p^{\infty}})$ and $K'=K(p^{1/2})$. 
\begin{enumerate}[label=(\roman*)]
\item Recall that the extension $\Oo_{K'}^a/\Oo_K^a$ is \'etale from Example \ref{quadex}. The $\Oo_K$-module $\Oo_{K'}$ is uniformly almost finitely presented (with $n=2$). Indeed, for each $m\geq 1$, we have an injective map
$$\phi_m: \Oo_K\oplus p^{1/2p^m}\Oo_K\to \Oo_{K'}$$
with cokernel killed by $p^{1/2p^m}$. We conclude that $\Oo_{K'}^a/\Oo_K^a$ is finite \'etale.

\item Let $K_U,K_U'$ be the corresponding non-principal ultrapowers. By \L o\'s, we can write $K_U'=K_U(\pi^{1/2})$, where $\pi=\ulim p^{1/p^n}$. 
The map $(\phi_m)_{m\in \omega}$ from part (i) is an almost isomorphism in $\Delta \Oo_K^a\lmod$. Applying the exact functor $S^{-1}(-)_U$ gives that the canonical inclusion 
$$S^{-1}\phi_U:\Oo_w\oplus \pi^{1/2} \Oo_w\to \Oo_{w'}$$
is an isomorphism, hence $\Oo_{w'}=\Oo_w[\pi^{1/2}]$. Since $\pi\in \Oo_w^{\times}$ and $p\neq 2$, the extension $\Oo_{w'}/\Oo_w$ is a finite \'etale extension. 
\end{enumerate}

\end{example}
Our framework is not the first one to achieve a passage from almost ring theory to ring theory. After writing this paper, O. Gabber made us aware of the following very similar construction which first appeared in \S 17.5 \cite{gabber2018foundations}:
\begin{remark}[Gabber's construction] \label{Gabberscons}
Let $A$ be an $\Oo_K$-algebra. Given an  $A$-module $M$, we define 
$$M^{\diamond}=M^\omega/M^{(\omega)}$$ 
Namely, $M^{\diamond}$ is the reduced product of $\omega$ copies of $M$ with respect to the cofinite (or Fr\'echet) filter on $\omega$. Let 
$$S=\{(\varpi^{\alpha_n})_{n\in \omega} \in \Oo_K^{\omega}: \lim_{n\to \infty} \alpha_n \to 0\}$$
By  Lemma 17.5.5 \cite{gabber2018foundations}, we have an exact faithful functor 
$$S^{-1}(-)^{\diamond}: A^a\lmod \to S^{-1}A^\diamond\lmod: M^a\mapsto S^{-1}M^\diamond $$
Gabber used this functor to pass from an almost Cohen-Macaulay algebra to an honest one and simplified a step in the proof of the direct summand conjecture (cf. \S 6.5 \cite{dospinescu}). 
\end{remark}
It is often preferable to work with ultraproducts rather than reduced products because they have nicer preservation properties and are more susceptible to model-theoretic analysis. Note for instance that $(\Oo_{K})_U$ is still a valuation ring, while $\Oo_K^{\diamond}$ is not even an integral domain. In \S \ref{char0}, we utilized the standard decomposition to give a short proof that every finite extension of $K_U$ is unramified with respect to $\Oo_w=S^{-1}(\Oo_{K})_U$.
\subsubsection{Almost purity over perfectoid valuation rings} \label{almostpuritysec}
We now deduce the almost purity theorem over perfectoid valuation rings from its non-standard counterpart. First, we need a lemma:
\bl \label{semiperfectlem}
Let $(K,v)$ be an $\aleph_1$-saturated valued field of positive residue characteristic $p>0$. 
Let $\varpi\in \mathfrak{m}_v\backslash \{0\}$ with $0< v\varpi \leq vp$ and $w$ be the coarsest coarsening of $v$ such that $w\varpi>0$. Suppose that $\Gamma_w$ does not have a minimal positive element. Then, the following are equivalent: 
\begin{enumerate}[label=(\roman*)]
\item $\Oo_v/\varpi \Oo_v$ is semi-perfect.

\item $k_w$ is perfect.

\item $\Oo_w/\varpi \Oo_w$ is semi-perfect. 

\end{enumerate}
\el 
\begin{proof}
Let $w^+$ be the finest coarsening of $v$ with $w^+\varpi=0$ and consider the standard decomposition with respect to $\varpi$ below
$$K\stackrel{w^+} \longrightarrow k_{w^+} \stackrel{\overline{w}} \longrightarrow k_{w} \stackrel{\overline{v}}\longrightarrow k_{v} $$ 
By Lemma \ref{sphericalcompletelemma}, the valued field $(k_{w^+},\overline{w})$ is spherically complete with value group $\Z$ or $\mathbb{R}$. Since $\Gamma_{w}$ has no minimal positive element and $\Gamma_{\overline{w}}$ is a convex subgroup of $\Gamma_w$, we must have that $\Gamma_{\overline{w}}=\mathbb{R}$. \\
(i)$\Rightarrow$(ii): We have $\Oo_{\overline{v}}=(\Oo_v/\varpi\Oo_v)_{\text{red}}$ and $k_w=\text{Frac}(\Oo_{\overline{v}})$. Since $\Oo_v/\varpi\Oo_v$ is semi-perfect, the same is true for $(\Oo_v/\varpi\Oo_v)_{\text{red}}$ and therefore $k_w$ is perfect. \\
(ii)$\Rightarrow$(iii): 
By Fact \ref{tamechar}, the valued field $(k_{w^+},\overline{w})$ is tame. Tame fields are perfect and deeply ramified (see Theorem 1.2 \cite{KuhlBla}). In particular, we get that $\Oo_{\overline{w}}/\varpi \Oo_{\overline{w}}$ is semi-perfect. Finally, note that $\Oo_w/\varpi \Oo_w$ maps isomorphically to $\Oo_{\overline{w}}/\varpi \Oo_{\overline{w}}$ via the residue map of $w^+$.\\
(iii)$\Rightarrow$(i): 
First, using that $\Gamma_{\overline{w}}=\mathbb{R}$, we can find $\varepsilon$ such that
$ p\cdot w\varepsilon <w \varpi <2p \cdot w \varepsilon$. Observe that $\varpi \Oo_w\subseteq \varepsilon^p \Oo_v$ using the chain of inclusions below
$$\varpi \Oo_w \subseteq \varepsilon^p \mathfrak{m}_w \subseteq \varepsilon^p \mathfrak{m}_v \subseteq \varepsilon^p \Oo_v $$
Now let $x \in \Oo_v$. Since $\Oo_w/\varpi \Oo_w$ is semi-perfect, there is $y_1 \in \Oo_w$ such that $x\equiv y_1^p \mod \varpi \Oo_w$. 
Since $\varpi \Oo_w\subseteq \varepsilon^p \Oo_v$, there is $x_1\in \Oo_v$ such that
$$x=y_1^p+\varepsilon^p \cdot x_1$$ 
Note also that $y_1\in \Oo_v$. By similar reasoning, there are $x_2,y_2\in \Oo_v$ such that $x_1=y_2^p+\varepsilon^p \cdot x_2$. 
We then compute that 
$$x=y_1^p+\varepsilon^p \cdot y_2^p+ \varepsilon^{2p} \cdot x_2$$ 
Computing modulo $\varpi \Oo_v$, we get
$$x\equiv (y_1+ \varepsilon \cdot y_2 )^p \mod \varpi \Oo_v $$
We conclude that $\Oo_v/\varpi \Oo_v$ is semi-perfect.
\end{proof}
\begin{Theorem} [Almost purity] \label{almostpurity}
Let $(K,v)$ be a perfectoid field and $(K',v')/(K,v)$ be a finite extension. Then:
\begin{enumerate}[label=(\roman*)]
\item $(K',v')$ is a perfectoid field.

\item The extension $\Oo_{K'}^a/\Oo_K^a$ is finite \'etale.
\end{enumerate}
\end{Theorem} 
\begin{proof}
Write $K_U$ (resp. $K_U'$) for a non-principal ultrapower and $w$ (resp. $w'$) for the coarsest coarsening of $v_U$ with $w\varpi>0$ (resp. $w'\varpi>0$).\\
(i) The only non-trivial thing to show is that $\Oo_{K'}/p\Oo_{K'}$ is semi-perfect. Since $\Oo_K/p\Oo_K$ is semi-perfect, we get that $k_w$ is perfect by Lemma \ref{semiperfectlem}. The same is true for $k_{w'}$ since it is a finite extension of $k_w$. By Lemma \ref{semiperfectlem}, we get that $\Oo_{K_U'}/p\Oo_{K_U'}$ is semi-perfect. We conclude that $\Oo_{K'}/p\Oo_{K'}$ is semi-perfect.\\
(ii) By Theorem \ref{modeltheoreticFont}(I), we have that $(K_U',w')/(K_U,w)$ is unramified. Equivalently, by Fact \ref{unramfact}, the extension $\Oo_{w'}/\Oo_w$ is finite \'etale. In particular, $\Oo_{w'}$ is a free $\Oo_w$-module of rank $n$ with $n=[K_U':K_U]=[K':K]$. By Theorem \ref{dictionary}(ii), we get that $\Oo_{K'}^a$ is a uniformly finite projective $\Oo_K^a$-module. In particular, $\Oo_{K'}^a$ is a uniformly finitely presented $\Oo_K^a$-module. We conclude that $\Oo_{K'}^a/\Oo_K^a$ is finite \'etale from Theorem \ref{dictionary}(iv).
\end{proof}
We refer the reader to pg. 29 \cite{Scholze} and \cite{Ishiro}, which provide several different proofs of almost purity over perfectoid valuation rings.

\subsubsection*{Acknowledgements} 
We wish to thank Sylvy Anscombe, Philip Dittmann, Arno Fehm, Ofer Gabber, Franz-Viktor Kuhlmann, François Loeser, Lucas Mann, Silvain Rideau-Kikuchi and Tom Scanlon for fruitful discussions and encouragement. Special thanks to P. Dittmann, A. Fehm, F.-V. Kuhlmann and F. Loeser who also provided us with thoughtful comments on earlier versions and to O. Gabber who informed us about his construction. Finally, we are grateful to an anonymous referee for their thorough report.


\printbibliography

\end{document}